\documentclass[a4paper,10pt]{article}

\usepackage[english]{babel}
\usepackage{geometry}
\usepackage{amsmath}
\usepackage{amsfonts}
\usepackage{amsthm,amssymb}
\usepackage[latin1]{inputenc}  
\usepackage[T1]{fontenc}       
\usepackage{amssymb}
\usepackage[pdftex]{graphicx}
\usepackage{dsfont}
\usepackage{psfrag}
\usepackage{stmaryrd}
\usepackage{hyperref} 
\hypersetup{ 
backref=true, 
pagebackref=true,
hyperindex=true, 
colorlinks=true, 
breaklinks=true, 
urlcolor= blue, 
linkcolor= red, 
bookmarks=true, 
bookmarksopen=true} 
\usepackage{color}
\definecolor{marin}{rgb}   {0.,   0.3,   0.7} 
\definecolor{rouge}{rgb}   {0.8,   0.,   0.} 
\definecolor{sepia}{rgb}   {0.8,   0.5,   0.}

\newtheorem{theorem}{Theorem}[section]
\newtheorem{corollary}[theorem]{Corollary}
\newtheorem{proposition}[theorem]{Proposition}
\newtheorem{remark}{Remark}

\newtheorem{lemma}[theorem]{Lemma}

\newcommand{\R}{\mathbb{R}}

\newcommand{\K}{\mathcal{K}}
\newcommand{\Sp}{\mathbb{S}}
\newcommand{\N}{\mathbb{N}}

\newcommand{\T}{\mathbb{T}}
\newcommand{\Z}{\mathbb{Z}}

\title{On the Production of Dissipation by\\
Interaction of Forced Oscillating Waves\\ 
in Fluid Dynamics }
\author{ Aur\'elien Klak
\footnote{Universit\'e de Rennes I, IRMAR, UMR 6625-CNRS,Campus de Beaulieu, 35042 Rennes, France, E.mail: aurelien.klak@univ-rennes1.fr}}
\date{}

\begin{document}

\maketitle


\begin{abstract}
In the context of some bidimensionnal Navier-Stokes model, we exhibit a family of exact oscillating solutions $\{u_{\varepsilon}\}_{\varepsilon}$ defined on some strip $[0,T]\times\R^2$ which does not depend on $\varepsilon\in]0,1]$. The exact solutions is described thanks to a complete expansions which reveal a boundary layer in time $t=0$. The interactions of the various scales ($1$, $1/\varepsilon$ and $1/\varepsilon^2$)  produce a macroscopic effect given by the addition of a diffusion. To justify the existence of $\{u_\varepsilon\}_{\varepsilon}$, we need to perform various Sobolev estimates that rely on a refined balance between the informations coming from the hyperbolic and parabolic parts of the equations.
\end{abstract}




\section{Introduction}

\label{Introduction} 
In Section \ref{Introduction}, we introduce the underlying equations
and the functional framework. Then, we state our main result.


\subsection{The equations}

\label{Equations} 
The time and space variables are $t\in\R_{+}$ and $x:=(x_{1},x_{2})\in\R^{2}$.
The state variables are the density $\rho\in\R_{+}$ and the two components
$u_{1}$ and $u_{2}$ of the velocity of the fluid $u:={}^{t}(u_{1},u_{2})\in\R^{2}$.
Given a function $u:\R^{2}\longrightarrow\R^{2}$, note as usual:
\[
\text{div}\, u:=\partial_{1}u_{1}+\partial_{2}u_{2}\,,\qquad\partial_{1}:=\frac{\partial}{\partial x_{1}}\,,\quad\partial_{2}:=\frac{\partial}{\partial x_{2}}\,.
\]
 In what follows, $\varepsilon\in\,]0,1]$ is a parameter approaching
zero. Introduce the dissipation: 
\[
\mathcal{P}_{\varepsilon}u=\,^{t}\left(\mathcal{P}_{\varepsilon}^{1}u\,,\mathcal{P}_{\varepsilon}^{2}u\right):=\mu\,\varepsilon^{2}\,\Delta_{x}u+\lambda\ \varepsilon^{3}\ \nabla\,\text{div}\ u
\]
 where $\mu,\,\lambda\in\R_{+}^{*}$ are fixed. Let $h$ be some smooth
periodic function with mean zero: 
\[
h\,:\,\T\longrightarrow\R\,,\qquad\T:=\R/\Z\,,\qquad h\in{\mathcal{C}}^{\infty}(\T;\R)\,,\qquad\int_{\T}\, h(\theta)\ d\theta=0\,.
\]

Consider the following oscillation which is polarized on the second
component: 
\[
F_{\varepsilon}(x)={}^{t}(0,F_{\varepsilon}^{2})(x):=\varepsilon^{-2}\ {}^{t}\Bigl(0,\mu\,\partial_{\theta\theta}^{2}h\left(\varepsilon^{-2}x_{1}\right)\Bigr)\,,\qquad\varepsilon\in\,]0,1]\,.
\]
 Our starting point is the study of a model based on two-dimensional
compressible isentropic equations of the Navier-Stokes type, as can
be found in \cite{MR2296806,MR1637634}, forced by the source term
$F_{\varepsilon}\,$: 
\[
\begin{cases}
\partial_{t}\rho+\text{div}(\rho\, u)=0\,,\\
\partial_{t}(\rho u)+\text{div}(\rho\, u\otimes u)+\nabla\rho^{\gamma}=\rho\left(\mathcal{P}_{\varepsilon}u-F_{\varepsilon}\right)\,,
\end{cases}\quad u\otimes u:=\begin{pmatrix}u_{1}u_{1} & u_{1}u_{2}\\
u_{1}u_{2} & u_{2}u_{2}
\end{pmatrix}\,,
\]
 where $\gamma$ the adiabatic constant is supposed to be larger than
one. To obtain a quasi linear system having a symmetric form, it is
classical \cite{MR928191} to introduce the state variable $p:=\frac{\sqrt{\gamma}}{C}\rho^{C}$
with $C:=\frac{\gamma-1}{2}$. Then, we have to deal with: 
\begin{equation}
\begin{cases}
\partial_{t}p+u\cdot\nabla p+C\, p\,\text{div}\, u=0\,,\\
\partial_{t}u+u\cdot\nabla u+C\, p\,\nabla p=\mathcal{P}_{\varepsilon}u-F_{\varepsilon}\,.
\end{cases}\label{eqref}
\end{equation}
 Observe that: 
\[
\mathcal{P}_{\varepsilon}{}^{t}(0,h_{\varepsilon})-F_{\varepsilon}=0\,,\qquad h_{\varepsilon}(x):=h\Bigl(\frac{x_{1}}{\varepsilon^{2}}\Bigr)\,,\qquad\forall\,\varepsilon\in\,]0,1]\,.
\]
 It follows that, for all $\varepsilon\in\,]0,1]$, the oscillation
$^{t}(0,0,h_{\varepsilon})$ satisfies Equation~\eqref{eqref}.

Our aim is to consider the problem of the \textit{stability} of such
families of solutions. To this end, at the initial time $t=0$, we
modify $^{t}(0,0,h_{\varepsilon})$ by adding some perturbation. More
precisely, we start with: 
\begin{equation}
(p,u^{1},u^{2})(0,x)\,=\,\bigl(0,0,h\bigr)\Bigl(\frac{x_{1}}{\varepsilon^{2}}\Bigr)\,+\,\bigl(\,\varepsilon^{\nu}\, q_{0,\varepsilon}\,,\,\varepsilon^{M}\, v_{0,\varepsilon}^{1}\,,\,\varepsilon^{M}\, v_{0,\varepsilon}^{2}\,\bigr)\Bigl(\frac{x_{1}}{\varepsilon^{2}},\frac{x_{2}}{\varepsilon}\Bigr)\label{eqrefdatas}
\end{equation}
 where $(\nu,M)\in\N^{2}$ with $\nu$ large enough and $M\geq7/2$
(retain that $\nu\gg M$), whereas: 
\[
(q_{0,\varepsilon},v_{0,\varepsilon}^{1},v_{0,\varepsilon}^{2})(\theta,y)\in H^{\infty}(\T\times\R;\R^{3})\,,\qquad y:=\frac{x_{2}}{\varepsilon}\in\R\,.
\]

One effect of the above perturbation is to introduce a dependence
on $x_{2}\in\R$ (or $y\in\R$). Despite the smallness of $\varepsilon^{\nu}$
(and maybe $\varepsilon^{M}$), when solving (\ref{eqref})-(\ref{eqrefdatas}),
we have to understand the interactions that occur between the very
fast oscillations in the direction $x_{1}$ (with wavelength $\varepsilon^{2}$)
and the fast variations in the transversal direction $x_{2}$ (with
wavelength $\varepsilon$). On this way, we are faced with questions
about \textit{turbulence}, in the spirit of models proposed in \cite{MR2182053,MR2233700,che2008}.\\

\medskip{}

Another insight on the subject can be obtained by looking at (\ref{eqref})
in the variables $(\theta,y)\in\T\times\R$. Then, we are faced with
a hyperbolic-parabolic system implying some singular (in $\varepsilon\in\,]0,1]$)
symmetric quasilinear part: 
\begin{equation}
\begin{cases}
\partial_{t}p\ \,+\varepsilon^{-2}\,\left(u^{1}\,\partial_{\theta}p+\varepsilon\, u^{2}\,\partial_{y}p\right)\ \ \ +C\,\varepsilon^{-2}\, p\,\left(\partial_{\theta}u^{1}+\varepsilon\,\partial_{y}u^{2}\right)=0\,,\\
\partial_{t}u^{1}+\varepsilon^{-2}\,\left(u^{1}\,\partial_{\theta}u^{1}+\varepsilon\, u^{2}\,\partial_{y}u^{1}\right)+C\,\varepsilon^{-2}\, p\ \partial_{\theta}p=\widetilde{\mathcal{P}}_{\varepsilon}^{1}u\,,\\
\partial_{t}u^{2}+\varepsilon^{-2}\,\left(u^{1}\,\partial_{\theta}u^{2}+\varepsilon\, u^{2}\,\partial_{y}u^{2}\right)+C\,\varepsilon^{-1}\, p\ \partial_{y}p=\widetilde{\mathcal{P}}_{\varepsilon}^{2}u-F_{\varepsilon}^{2}\,,
\end{cases}\label{eqrefthetay}
\end{equation}
 and some viscosity which is degenerate on the density and becomes
large when $\varepsilon\rightarrow0$: 
\[
\widetilde{\mathcal{P}}_{\varepsilon}u:=\begin{pmatrix}\widetilde{\mathcal{P}}_{\varepsilon}^{1}u\\
\widetilde{\mathcal{P}}_{\varepsilon}^{2}u
\end{pmatrix}=\frac{1}{\varepsilon^{2}}\begin{pmatrix}\mu\left(\partial_{\theta\theta}u^{1}+\varepsilon^{2}\partial_{yy}u^{1}\right)+\lambda\varepsilon\left(\partial_{\theta\theta}u^{1}+\varepsilon\partial_{\theta y}u^{2}\right)\\
\mu\left(\partial_{\theta\theta}u^{2}+\varepsilon^{2}\partial_{yy}u^{2}\right)+\lambda\varepsilon\left(\varepsilon\partial_{\theta y}u^{1}+\varepsilon^{2}\partial_{yy}u^{2}\right)
\end{pmatrix}.
\]

In this article, we show that (for $\nu$ large enough and $M\geq7/2$)
the \textit{oscillating Cauchy problem} \eqref{eqref}-\eqref{eqrefdatas}
is locally well posed in time. We prove (Theorem \ref{propestimationreste})
the existence of a time $T\in\R_{+}^{*}$ independent of $\varepsilon\in\,]0,1]$
with solutions $(p_{\varepsilon},u_{\varepsilon}^{1},u_{\varepsilon}^{2})=(\varepsilon^{\nu}\, q_{\varepsilon},\varepsilon^{M}\, v_{\varepsilon}^{1},h_{\varepsilon}+\varepsilon^{M}\, v_{\varepsilon}^{2})$
of (\ref{eqref})-(\ref{eqrefdatas}) on the interval $[0,T]$. We
also exhibit (Propositions \ref{propconstructionvelocity} and \ref{propconstructionpressure})
a complete expansion as $\varepsilon$ approaches $0$ for the expression
$(q_{\varepsilon},v_{\varepsilon}^{1},v_{\varepsilon}^{2})$. We find
(in a sense to be specified later) that $q_{\varepsilon}\simeq q_{\varepsilon}^{a}$
and $v_{\varepsilon}:=(v_{\varepsilon}^{1},v_{\varepsilon}^{2})\simeq v_{\varepsilon}^{a}:=(v_{\varepsilon}^{a1},v_{\varepsilon}^{a2})$
with: 
\begin{equation}
{\displaystyle q_{\varepsilon}^{a}(t,y,\theta)\,=\sum_{k=0}^{N+1}\,\varepsilon^{k}\, q_{k}^{\varepsilon}(t,y,\theta)\,,\qquad v_{\varepsilon}^{a}(t,y,\theta)\,=\sum_{k=0}^{N+1}\,\varepsilon^{k}\,\Bigl(v_{k}^{s}(t,y,\theta)+v_{k}^{f}\bigl(\frac{t}{\varepsilon^{2}},y,\theta\bigr)\Bigr)\,.}\label{asymplty}
\end{equation}
 These expansions reveal some time boundary layer at time $t=0$ (recorded
at the level of the contribution $v_{k}^{f}(\tau,\cdot)$ which is
exponentially decreasing with respect to the variable $\tau$) together
with some mean evolution behaviour (described by $v_{k}^{s}$). A
noticeable aspect is the production of some dissipation when looking
at the transport equation (\ref{eqhomogenizationg}) on $v_{k}^{s}$.
The present approach is not in the continuation of usual $k-\varepsilon$
models \cite{MR804006}. But, in the same spirit, it confirms (and
justifies) that the interaction of oscillations can indeed be described
at a macroscopic level by the introduction of some \textit{turbulent
viscosity}.


\subsection{The functional framework}

\subsubsection{Sobolev spaces}

Here $K$ denotes $\R$, $\T\times\R$ or $\R^{2}$. Let $\alpha=(\alpha_{1},\alpha_{2})\in\N^{2}$.
The length of $\alpha$ is $|\alpha|:=\alpha_{1}+\alpha_{2}$. The
notation $\partial^{\alpha}$ is for the differential operator $\partial_{\theta}^{\alpha_{1}}\partial_{y}^{\alpha_{2}}$.
\\

- Given $m\in\N\cup\{+\infty\}$ and $p\in\N^{*}\cup\{+\infty\}$,
recall that $W^{m,p}$ is: 
\[
W^{m,p}:=\bigl\{ f\in L^{p}(K)\ ;\ \partial^{\alpha}f\in L^{p}(K),\ |\alpha|\leq m\bigr\}\,,\qquad H^{m}:=W^{m,2}\,.
\]
 When $m\in\N$, the space $W^{m,p}$ can be equipped with the following
semi-norm and norm: 
\[
\forall\, p\in\N^{*}\cup\{+\infty\},\qquad\left\Vert f\right\Vert _{\overset{\circ\quad\,}{W^{m,p}}}:=\!\sum_{\alpha\in\N^{2},\ |\alpha|=m}\left\Vert \partial^{\alpha}f\right\Vert _{L^{p}}\,,\qquad\left\Vert f\right\Vert _{W^{m,p}}:=\sum_{k=0}^{m}\,\left\Vert f\right\Vert _{\overset{\circ\quad\,}{W^{k,p}}}\,.
\]
 For $s\in\R_{+}\setminus\N$, we can still define spaces $W^{s,p}$
and $H^{s}$ by interpolation theory. \\

- Let $(m,n)\in\N^{2}$ with $n\leq m$. Define the functional spaces:
\begin{align*}
\mathcal{W}_{T}^{m,n} & :=\Bigl\lbrace\, f\,;\, f\in C^{j}\bigl([0,T];W^{m-j,\infty}\bigr),\ \forall\, j\in\left\{ 0,\ldots,n\right\} \Bigr\rbrace\,,\\
\mathcal{H}_{T}^{m,n} & :=\Bigl\lbrace\, f\,;\, f\in C^{j}\bigl([0,T];H^{m-j}\bigr),\ \forall\, j\in\left\{ 0,\ldots,n\right\} \Bigr\rbrace\,,\qquad T\in\R_{+}\cup\{+\infty\}\,,
\end{align*}
 which can be seen as Banach spaces when provided with the norms:
\[
\left\Vert f\right\Vert _{\mathcal{W}_{T}^{m,n}}:=\sup_{t\in[0,T]}\ \sum_{j=0}^{n}\ \|\partial_{t}^{j}f(t,\cdot)\|_{W^{m-j,\infty}}\,,\qquad\left\Vert f\right\Vert _{\mathcal{H}_{T}^{m,n}}:=\sup_{t\in[0,T]}\ \sum_{j=0}^{n}\ \|\partial_{t}^{j}f(t,\cdot)\|_{H^{m-j}}\,.
\]

- In order to deal with functions $f(t,\cdot)$ defined on $\R_{+}\times K$,
which are exponentially decreasing in the time $t\in\R_{+}$, and
which take their values in the Sobolev space $H^{s}$, define: 
\[
\mathcal{E}_{\delta}^{s}:=\Bigl\{ f\ ;\ \sup_{t\in[0,+\infty[}\,\bigl(\, e^{\delta t}\ ||f(t,\cdot)||_{H^{s}(K)}\,\bigr)<+\infty\Bigr\}\,,\qquad\delta\in\R_{+}^{*}\,.
\]

- Finally, introduce ${\displaystyle \mathcal{E}_{\delta}^{\infty}:=\bigcap_{j\in\N}\mathcal{E}_{\delta}^{j}}$,
${\displaystyle \mathcal{V}_{T}^{\infty,0}:=\bigcap_{j\in\N}\mathcal{V}_{T}^{j,0}}$
and ${\displaystyle \mathcal{V}_{T}^{\infty}:=\bigcap_{j\in\N}\mathcal{V}_{T}^{j,j}}$
where $\mathcal{V}\in\{\mathcal{H},\mathcal{W}\}$.

\subsubsection{Families of functions}

In this paragraph, we fix some $\varepsilon_{0}\in\,]0,1]$ and look
at families of the type $\{f_{\varepsilon}\}_{\varepsilon\in\,]0,\varepsilon_{0}]}$.\\

\smallskip{}

- Assume that $f_{\varepsilon}\in W^{m,p}(K)$ for all $\varepsilon\in\,]0,\varepsilon_{0}]$.
To control the size of $f_{\varepsilon}$, we can use the following
weighted anisotropic semi-norm and norm: 
\[
\forall\, p\in\N^{*}\cup\{+\infty\},\qquad\left\Vert f_{\varepsilon}\right\Vert _{\overset{\circ\quad\,}{W_{(1,\varepsilon)}^{m,p}}}:=\!\sum_{\alpha\in\N^{2},\ |\alpha|=m}\left\Vert \varepsilon^{\alpha_{1}}\partial^{\alpha}f_{\varepsilon}\right\Vert _{L^{p}}\,,\qquad\left\Vert f_{\varepsilon}\right\Vert _{W_{(1,\varepsilon)}^{m,p}}:=\sum_{k=0}^{m}\,\left\Vert f_{\varepsilon}\right\Vert _{\overset{\circ\quad\,}{W_{(1,\varepsilon)}^{k,p}}}\,.
\]
 We will say that $\{f_{\varepsilon}\}_{\varepsilon}$ is bounded
in $W_{(1,\varepsilon)}^{m,p}$ when: 
\[
{\displaystyle \quad\vert\!\vert\!\vert f_{\cdot}\vert\!\vert\!\vert_{W_{(1,\cdot)}^{m,p}}:=\sup_{\varepsilon\in]0,\varepsilon_{0}]}\ \left\Vert f_{\varepsilon}\right\Vert _{W_{(1,\varepsilon)}^{m,p}}<+\infty\,.}
\]

- Assume that $f_{\varepsilon}\in\mathcal{V}_{T}^{m,n}$ for all $\varepsilon\in\,]0,\varepsilon_{0}]$
where $\mathcal{V}=\mathcal{W}$ or $\mathcal{V}=\mathcal{H}$. To
control the size of $f_{\varepsilon}$, we can use the following norms:
\[
\left\Vert f_{\varepsilon}\right\Vert _{\mathcal{V}_{T,\varepsilon}^{m,n}}:=\!\sup_{t\in[0,T]}\ \sum_{j=0}^{\lfloor n/2\rfloor}\ \|\varepsilon^{2j}\,\partial_{t}^{j}f_{\varepsilon}(t,.)\|_{V^{m-j}}\,,\quad\left\Vert f_{\varepsilon}\right\Vert _{\mathcal{V}_{T,(1,\varepsilon)}^{m,n}}\!:=\sup_{t\in[0,T]}\ \sum_{j=0}^{\lfloor n/2\rfloor}\,\|\varepsilon^{2j}\,\partial_{t}^{j}f_{\varepsilon}(t,.)\|_{V_{(1,\varepsilon)}^{m-j}}\,.
\]
 We will say that $\{f_{\varepsilon}\}_{\varepsilon}$ is bounded
in $\mathcal{V}_{T,\varepsilon}^{m,n}$ or in $\mathcal{V}_{T,(1,\varepsilon)}^{m,n}$
when we have respectively: 
\[
\vert\!\vert\!\vert f_{\cdot}\vert\!\vert\!\vert_{\mathcal{V}_{T,\cdot}^{m,n}}:=\sup_{\varepsilon\in\,]0,\varepsilon_{0}]}\ \left\Vert f_{\varepsilon}\right\Vert _{\mathcal{V}_{T,\varepsilon}^{m,n}}<+\infty\,,\qquad\vert\!\vert\!\vert f_{\cdot}\vert\!\vert\!\vert_{\mathcal{V}_{T,(1,\cdot)}^{m,n}}:=\sup_{\varepsilon\in\,]0,\varepsilon_{0}]}\ \left\Vert f_{\varepsilon}\right\Vert _{\mathcal{V}_{T,(1,\varepsilon)}^{m,n}}<+\infty\,.
\]

\noindent \textit{{Classical embedding} }: for $s>1$ we have $H^{s}(\T\times\R)\hookrightarrow W^{0,\infty}(\T\times\R)\equiv L^{\infty}(\T\times\R)$.
When taking into account the dependence on $\varepsilon\in\,]0,1]$,
there is a loss of powers in $\varepsilon$. Retain here that: 
\begin{equation}
\exists\, C\in\R_{*}^{+}\,;\qquad\left\Vert f_{\varepsilon}\right\Vert _{L^{\infty}(\T\times\R)}\,\leq\, C\ \varepsilon^{-1/2}\ \|f_{\varepsilon}\|_{H_{(1,\varepsilon)}^{s}(\T\times\R)}\,,\quad\forall\,\varepsilon\in\,]0,\varepsilon_{0}]\,.\label{eqinfininorm}
\end{equation}


\subsubsection{Decomposition of a periodic function}

\label{Decompositionofaperiodic}

Any function $u\in L^{2}(\T;\R)$ can be decomposed as: 
\[
u(\theta)=\left<u\right>+u^{*}(\theta)\,,\qquad\left<u\right>\in\R\,,\qquad u^{*}\in L^{2}(\T;\R)\,,\qquad\left<u\right>\equiv\Pi u:={\displaystyle \int_{\T}u(\theta)\ d\theta\,.}
\]

In what follows, given a symbol $\mathcal{V}\in\left\{ W^{m,p},H^{s},\mathcal{W}_{T}^{m,s},\mathcal{H}_{T}^{s},\mathcal{E}_{\delta}^{s}\right\} $,
we will manipulate functions $f(t,\theta,y)\in\mathcal{V}(\T\times\R)$.
We will often decompose $f$ into its mean and oscillating parts according
to 
\[
u(t,\theta,y)=\left<u\right>(t,y)+u^{*}(t,\theta,y),
\]
 with, \begin{subequations}\label{convenpre} 
\begin{align}
f^{\sslash}(t,y) & :=\Pi f(t,y):=\langle f(t,\cdot,y)\rangle\in\mathcal{V}^{\sslash}\!:=\Pi\mathcal{V}(\T\times\R)\,,\label{convenpre1}\\
f^{\perp}(t,\theta,y) & :=f^{*}(t,\theta,y)\in V^{\perp}\!:=(I-\Pi)\mathcal{V}(\T\times\R)\,.\quad\label{convenpre2}
\end{align}
 \end{subequations} To signal that we consider functions $f(t,\theta,y)$
which do not depend on $\theta\in\T$ ($\Pi f=f$) or whose mean value
is zero ($\Pi f=0$), we will use respectively (as above) the marks
$\sslash$ and $\perp$. By extension, when dealing with some operator
$P$, we will note 
\begin{equation}
P^{\sslash}:=P\Pi\,,\qquad P^{\perp}:=P(I-\Pi)\,.\label{convensec}
\end{equation}
 Be careful, in the case of operators, the composition by $\Pi$ and
$I-\Pi$ is put \textit{on the right}.\\

The derivative $\partial_{\theta}$ acts in the sense of distributions
on the space $L^{2}(\T;\R)$. We find: 
\[
\K:=\text{ker}\,\partial_{\theta}=\{u\equiv c\,;\, c\in\R\}\,,\qquad\K^{\perp}:=(\text{ker}\,\partial_{\theta})^{\perp}=\bigl\{ u\in L^{2}(\T;\R)\,;\,\Pi u=0\bigr\}\,.
\]
 The action $\partial_{\theta}$ has a (right) inverse $\partial_{\theta}^{-1}:\K^{\perp}\longrightarrow\K^{\perp}\cap H^{1}(\T;\R)$
which is given by: 
\[
{\displaystyle \quad\partial_{\theta}^{-1}u(\theta):=\int_{0}^{\theta}u(s)\ ds-\int_{\T}\int_{0}^{\theta}u(s)\ ds\, d\theta\,,\qquad\forall\,\theta\in\T\,.}
\]


\subsection{Main statements}


Since we impose $\mathcal{\nu}\gg M\geq7/2$, the equations on the
components $u^{1}$ and $u^{2}$ can be considered as being partially
decoupled from the equation on $p$. Up to some extent, we can first
deal with: 
\begin{align*}
\mathcal{L}_{1}(\varepsilon,q_{\varepsilon},v_{\varepsilon}) & :=\,\partial_{t}v_{\varepsilon}^{1}+\varepsilon^{-1}\, h\ \partial_{y}v_{\varepsilon}^{1}+\varepsilon^{M-2}\,\left(v_{\varepsilon}^{1}\,\partial_{\theta}v_{\varepsilon}^{1}+\varepsilon\, v_{\varepsilon}^{2}\,\partial_{y}v_{\varepsilon}^{1}\right)+C\,\varepsilon^{2\,\nu-M-2}\, q_{\varepsilon}\,\partial_{\theta}q_{\varepsilon}-\widetilde{\mathcal{P}}_{\varepsilon}^{1}v_{\varepsilon}\,,\\
\mathcal{L}_{2}(\varepsilon,q_{\varepsilon},v_{\varepsilon}) & :=\,\partial_{t}v_{\varepsilon}^{2}+\varepsilon^{-1}\, h\ \partial_{y}v_{\varepsilon}^{2}+\varepsilon^{-2}\,\partial_{\theta}h\ v_{\varepsilon}^{1}+\varepsilon^{M-2}\,\left(v_{\varepsilon}^{1}\,\partial_{\theta}v_{\varepsilon}^{2}+\varepsilon\, v_{\varepsilon}^{2}\,\partial_{y}v_{\varepsilon}^{2}\right)\\
 & \qquad\qquad\qquad\qquad\qquad\qquad\qquad\qquad\qquad\qquad\qquad\ \,+C\,\varepsilon^{2\,\nu-M-1}\, q_{\varepsilon}\,\partial_{y}q_{\varepsilon}-\widetilde{\mathcal{P}}_{\varepsilon}^{2}v_{\varepsilon}\,,
\end{align*}
 and then look at the remaining part as a transport equation on $q_{\varepsilon}\,$:
\[
\mathcal{L}_{0}(\varepsilon,q_{\varepsilon},v_{\varepsilon})\,:=\,\partial_{t}q_{\varepsilon}+\varepsilon^{-1}\, h\ \partial_{y}q_{\varepsilon}+\varepsilon^{M-2}\,\left(v_{\varepsilon}^{1}\,\partial_{\theta}q_{\varepsilon}+\varepsilon\, v_{\varepsilon}^{2}\,\partial_{y}q_{\varepsilon}\right)+C\ \varepsilon^{M-2}\, q_{\varepsilon}\,\left(\partial_{\theta}v_{\varepsilon}^{1}+\varepsilon\,\partial_{y}v_{\varepsilon}^{2}\right)\,.
\]

In what follows, the results will be expressed in terms of the quantities
$q_{\varepsilon}$ and $v_{\varepsilon}$. Of course, the expression
\begin{equation}
(p_{\varepsilon},u_{\varepsilon}^{1},u_{\varepsilon}^{2})=(\varepsilon^{\nu}\, q_{\varepsilon},\varepsilon^{M}\, v_{\varepsilon}^{1},h+\varepsilon^{M}\, v_{\varepsilon}^{2})\label{eqdico}
\end{equation}
 is a solution of (\ref{eqrefthetay}) if and only if: 
\begin{equation}
\mathcal{L}_{j}(\varepsilon,q_{\varepsilon},v_{\varepsilon})\,=\,0\,,\qquad\forall\, j\in\{0,1,2\}\,.\label{eqonLj}
\end{equation}


\subsubsection{Construction of approximated solutions}

We start by constructing approximated solutions for the system \eqref{eqonLj}.
The first step is to look at the two last equations of (\ref{eqonLj}),
where the $O(\varepsilon^{2\,\nu-M-2})\ll1$ contributions (implying
$q_{\varepsilon}$) are neglected. Thus, we start by considering the
system: 
\[
\left\lbrace \begin{array}{l}
\mathcal{L}_{1}^{a}(\varepsilon,v_{\varepsilon}):=\,\partial_{t}v_{\varepsilon}^{1}+\varepsilon^{-1}\, h\ \partial_{y}v_{\varepsilon}^{1}+\varepsilon^{M-2}\,\left(v_{\varepsilon}^{1}\,\partial_{\theta}v_{\varepsilon}^{1}+\varepsilon\, v_{\varepsilon}^{2}\,\partial_{y}v_{\varepsilon}^{1}\right)-\widetilde{\mathcal{P}}_{\varepsilon}^{1}v_{\varepsilon}\,,\\
\mathcal{L}_{2}^{a}(\varepsilon,v_{\varepsilon}):=\,\partial_{t}v_{\varepsilon}^{2}+\varepsilon^{-1}\, h\ \partial_{y}v_{\varepsilon}^{2}+\varepsilon^{-2}\,\partial_{\theta}h\ v_{\varepsilon}^{1}+\varepsilon^{M-2}\,\left(v_{\varepsilon}^{1}\,\partial_{\theta}v_{\varepsilon}^{2}+\varepsilon\, v_{\varepsilon}^{2}\,\partial_{y}v_{\varepsilon}^{2}\right)-\widetilde{\mathcal{P}}_{\varepsilon}^{2}v_{\varepsilon}\,.
\end{array}\right.
\]

\begin{proposition}\label{propconstructionvelocity}Fix an integer
$M\in{\mathbb{N}}$ with $M\geq2$. Choose any integer $N\in{\mathbb{N}}$
and any decay rate $\delta\in\,]0,\mu[$. Select any functions $v_{k}^{0}\in H^{\infty}(\T\times\R;\R^{2})$
indexed by $k\in\llbracket0,N+1\rrbracket$. There are functions 
\[
v_{k}^{s}\in{\displaystyle \bigcap_{T\in\R_{+}^{*}}\mathcal{H}_{T}^{\infty}\,,\qquad v_{k}^{f}\in\mathcal{E}_{\delta}^{\infty}\,,\qquad k\in\llbracket0,N+1\rrbracket}
\]
 such that the family $\{v_{\varepsilon}^{a}\}_{\varepsilon}$ defined
as indicated in (\ref{asymplty}) satisfies the following conditions:

\textit{i)} At the initial time $t=0$, the trace $v_{\varepsilon}^{a}(0,\cdot)$
is prescribed in the following way: 
\begin{equation}
v_{\varepsilon}^{a}(0,\theta,y)=\sum_{k=0}^{N+1}\,\varepsilon^{k}\ v_{k}^{0}(\theta,y)\,.\label{prescribed}
\end{equation}

\textit{ii)} For all time $T\in\R_{+}^{*}$, for all $m\in\N$, the
family $\bigl\lbrace\varepsilon^{-N}\,\mathcal{L}_{j}^{a}(\varepsilon,v_{\varepsilon}^{a})\bigr\rbrace_{\varepsilon}$
with $j=1$ or $j=2$ is bounded in~$\mathcal{H}_{T,\varepsilon}^{m,0}$
in the sense that: 
\begin{equation}
\sup_{\varepsilon\in\,]0,1]}\quad\sup_{t\in[0,T]}\quad\max_{j\in\{1,2\}}\quad\left\Vert \varepsilon^{-N}\,\mathcal{L}_{j}^{a}(\varepsilon,v_{\varepsilon}^{a})\right\Vert _{H^{m}(\T\times\R)}<\,+\infty\,.\label{eqapproxsolv}
\end{equation}
 Furthermore, the expression $\Pi v_{k}^{s}$ is determined through
an equation of the form: 
\begin{align}
\partial_{t}\Pi v_{k}^{s}-\Bigl(\mu+\frac{1}{\mu}\,\Pi\left((\partial_{\theta}^{-1}h)^{2}\right)\Bigr)\,\partial_{yy}\Pi v_{k}^{s}=S_{k}\label{eqhomogenizationg}
\end{align}
 where the source term $S_{k}$ depends only on the $v_{j}^{s}$ with
$j\leq k-1$. \end{proposition} To complete $v_{\varepsilon}^{a}$
into some approximated solution $(q_{\varepsilon}^{a},v_{\varepsilon}^{a})$
of the complete system (\ref{eqrefthetay}), there remains to identify
the pressure component $q_{\varepsilon}^{a}$. To this end, we are
satisfied to solve directly the transport equation $\mathcal{L}_{0}(\varepsilon,q_{\varepsilon}^{a},v_{\varepsilon}^{a})=0$
where $v_{\varepsilon}^{a}$ is adjusted as in Proposition \ref{propconstructionvelocity}.
Since the expression $v_{\varepsilon}^{a}$ is a function of the scales
of time $t$ and $\frac{t}{\varepsilon^{2}}$, that goes for $q_{\varepsilon}^{a}(t,y,\theta)$
too. In what follows, we will not need to precise the way by which
$q_{\varepsilon}^{a}$ depends on the different time scales $\bigl(\, t$,
$\frac{t}{\varepsilon}$, $\frac{t}{\varepsilon^{2}}$, $\cdots\,\bigr)$.

\begin{proposition}\label{propconstructionpressure} The context
is as in Proposition $\ref{propconstructionvelocity}$. Note $\{v_{\varepsilon}^{a}\}_{\varepsilon}$
the family issued from the Proposition~\ref{propconstructionvelocity}.
Select functions $q_{k}^{0}\in H^{\infty}(\T\times\R;\R^{3})$ indexed
by $k\in\llbracket0,N+1\rrbracket$. There are functions $q_{k}^{\varepsilon}$
with: 
\[
\left\{ q_{k}^{\varepsilon}\right\} _{\varepsilon}\in{\displaystyle \bigcap_{T\in\R_{+}^{*}}\mathcal{H}_{T}^{\infty,0}\,,\qquad k\in\llbracket0,N+1\rrbracket\,,\,\qquad\varepsilon\in\,]0,1]}
\]
 such that the expressions $q_{\varepsilon}^{a}$ defined as indicated
in (\ref{asymplty}) are solutions of the Cauchy problem: 
\begin{equation}
\mathcal{L}_{0}(\varepsilon,q_{\varepsilon}^{a},v_{\varepsilon}^{a})=0\,,\qquad q_{\varepsilon}^{a}(0,\theta,y)={\displaystyle \sum_{k=0}^{N+1}\,\varepsilon^{k}\ q_{k}^{0}(\theta,y)\,.}\label{eqapproxsolq}
\end{equation}
 Moreover, for all time $T\in\R_{+}^{*}$, for all $m\in\N$ and for
all $k\in\llbracket0,N+1\rrbracket$, the family $\{q_{k}^{\varepsilon}\}_{\varepsilon}$
is bounded in $\mathcal{H}_{T,(1,\varepsilon)}^{m,0}$ in the sense
that: 
\begin{equation}
\sup_{\varepsilon\in]0,1]}\quad\sup_{t\in[0,T]}\quad\left\Vert q_{k}^{\varepsilon}(t,\cdot)\right\Vert _{H_{(1,\varepsilon)}^{m}(\T\times\R)}<+\infty\,.\label{eqapproxsolq2}
\end{equation}
 \end{proposition}

Coming back to $\mathcal{L}_{1}$ and $\mathcal{L}_{2}$, we can now
make the following statement.

\begin{proposition}\label{corapproxoperateur} Select $m,\, M,\, N,\,\nu\in\N$
satisfying: 
\begin{equation}
M\geq2\,,\qquad m\geq2\qquad\text{and}\qquad2\nu-M-5/2-(m+1)-N\geq0.\label{eqwsN}
\end{equation}
 Note $\{v_{\varepsilon}^{a}\}_{\varepsilon}$ and $\{q_{\varepsilon}^{a}\}_{\varepsilon}$
the families obtained with Propositions $\ref{propconstructionvelocity}$
and $\ref{propconstructionpressure}$. Then, for all $j\in\{1,2\}$,
we have: 
\[
\sup_{\varepsilon\in]0,1]}\quad\sup_{t\in[0,T]}\quad\left\Vert \varepsilon^{-N}\,\mathcal{L}_{j}(\varepsilon,q_{\varepsilon}^{a},v_{\varepsilon}^{a})\right\Vert _{H^{m}(\T\times\R)}\,<\,+\infty\,.
\]
 \end{proposition}


\subsubsection{Existence and stability result}

The parameter $\varepsilon\in\,]0,1]$ being fixed, the local in time
well-posedness of the Cauchy problem (\ref{eqref})-(\ref{eqrefdatas})
is standard, with corresponding solutions 
\begin{equation}
(p_{\varepsilon}^{e},u_{\varepsilon}^{e})=(\varepsilon^{\nu}\, q_{\varepsilon}^{e},\varepsilon^{M}\, v_{\varepsilon}^{e1},h+\varepsilon^{M}\, v_{\varepsilon}^{e2}).\label{eqdiko}
\end{equation}
 It means that, for all $\varepsilon\in\,]0,1]$, there is a time
$T_{\varepsilon}\in\R_{+}^{*}$ (eventually shrinking to zero when
$\varepsilon$ goes to zero) such that $(q_{\varepsilon}^{e},v_{\varepsilon}^{e})$
with $v_{\varepsilon}^{e}:=(v_{\varepsilon}^{e1},v_{\varepsilon}^{e2})$
is a solution of (\ref{eqonLj}) on the time interval $[0,T_{\varepsilon}]$
with initial data as indicated at the level of (\ref{prescribed})
and (\ref{eqapproxsolq}).

\smallskip{}

Fix any $R\in\N$. We can always define on the strip $[0,T_{\varepsilon}]$,
two functions $q_{\varepsilon}^{R}$ and $v_{\varepsilon}^{R}$ through
the identity 
\begin{equation}
(q_{\varepsilon}^{e},v_{\varepsilon}^{e})=(q_{\varepsilon}^{a},v_{\varepsilon}^{a})+\varepsilon^{R}\,(q_{\varepsilon}^{R},v_{\varepsilon}^{R})\,.\label{eqdefreste}
\end{equation}
 Two questions are solved below: the existence of exact solutions
of (\ref{eqref})-(\ref{eqrefdatas}) on a time interval $[0,T_{c}]$
with $T_{c}\in\R_{+}^{*}$ independent of $\varepsilon\in\,]0,1]\,$
and the production of controls on $(q_{\varepsilon}^{R},v_{\varepsilon}^{R})$
showing that $(q_{\varepsilon}^{a},v_{\varepsilon}^{a})$ gives indeed
some good asymptotic description of $(q_{\varepsilon}^{e},v_{\varepsilon}^{e})$
on $[0,T_{c}]$.

\begin{theorem}{[}Existence and stability{]}\label{propestimationreste}
Assume $\lambda<4\mu$. Let $m,\,\nu,\, M,\, N,\, R\in\N$ satisfying
\begin{equation}
M\geq7/2\qquad\text{and}\qquad w_{m}:=\min\left(2\nu-M-5/2-(m+3)-R,N-R\right)\geq0\,.\label{eqws}
\end{equation}

Let $T_{\varepsilon}$ the lifespan of the solution of the Cauchy problem \eqref{eqref}-\eqref{eqrefdatas}.
Then there exist $T_{c}>0$ and $\varepsilon_{crit}>0$ such that
\[
\forall\,\varepsilon\in]0,\varepsilon_{crit}],\quad T_{\varepsilon}\geq T_{c}.
\]

Furthermore, the approximated solution $^{t}(q_{\varepsilon}^{a},v_{\varepsilon}^{a})$,
constructed thanks to Propositions \ref{propconstructionvelocity}
and~\ref{propconstructionpressure}, is a relevant expansion for
the exact solution $^{t}(q_{\varepsilon}^{e},v_{\varepsilon}^{e})$
(associated with $^{t}(p_{\varepsilon}^{e},u_{\varepsilon}^{e})$
threw Equation \eqref{eqdiko}) in the sense that the remainder $^{t}(q_{\varepsilon}^{R},v_{\varepsilon}^{R})$,
defined in \eqref{eqdefreste}, satisfies the following statements.\\

$\textit{i)}$ The family $\{q_{\varepsilon}^{R}\}_{\varepsilon}$
is bounded in $\mathcal{H}_{T_{c},(1,\varepsilon_{c})}^{m+3,0}$:
\begin{equation}
\sup_{\varepsilon\in]0,\varepsilon_{c}]}\sup_{t\in[0,T_{c}]}\left\Vert q_{\varepsilon}^{R}(t,\cdot)\right\Vert _{H_{(1,\varepsilon)}^{m+3}}<\,+\infty.\label{ineqtime}
\end{equation}

\textit{ii)} The families $\{v_{\varepsilon}^{1R}\}_{\varepsilon}$
and $\{\varepsilon\, v_{\varepsilon}^{2R}\}_{\varepsilon}$ are bounded
in $\mathcal{H}_{T_{c},\varepsilon_{c}}^{m+3,0}$: 
\begin{equation}
\sup_{\varepsilon\in]0,\varepsilon_{c}]}\sup_{t\in[0,T_{c}]}\left\Vert v_{\varepsilon}^{1R}(t,\cdot)\right\Vert _{H^{m+3}}<\,+\infty,\qquad\sup_{\varepsilon\in]0,\varepsilon_{c}]}\sup_{t\in[0,T_{c}]}\left\Vert \varepsilon\, v_{\varepsilon}^{2R}(t,\cdot)\right\Vert _{H^{m+3}}<\,+\infty.\label{eqestimationv2R}
\end{equation}
 \end{theorem}


\subsection{The context}

\subsubsection{Historical comments}

Let $N\in\N^{*}$. Consider the following scalar equation of evolution:
\begin{equation}
\partial_{t}f_{\varepsilon}+\frac{1}{\varepsilon}\ h(\text{{\rm v}})\cdot\nabla_{x}f_{\varepsilon}+\frac{1}{\varepsilon^{2}}\ \mathcal{Q}\, f_{\varepsilon}\,=\, S(t,x,\text{{\rm v}})\,,\qquad f_{\varepsilon}(t,x,{\rm v)\in\R}\label{eqkinetic}
\end{equation}
 where $h:\R^{N}\longrightarrow\R^{N}$ is some smooth function, $\mathcal{Q}$
is some linear operator acting on $L^{2}$, and $S(t,x,\text{{\rm v}})$
is some function depending on the variables $(t,x,\text{{\rm v}})\in\R_{+}\times\R^{N}\times\R^{N}$.
The unknown is the function $f_{\varepsilon}(t,x,\text{{\rm v}})$.
Depending on the choice of $\mathcal{Q}$, the Equation (\ref{eqkinetic})
can be the neutron equation \cite{MR802803}, the Fokker-Planck equation
\cite{MR906919} or the Boltzmann transport equation \cite{MR1127004}.
In this context, it is well-known that the family $\{f_{\varepsilon}\}_{\varepsilon\in]0,1]}$
has a weak limit , say $f_{0}$ as $\varepsilon$ goes to 0. In general,
the expression $f_{0}$ satisfies an equation implying a \textit{drift-diffusion}
term of the form $-\,\text{div}_{x}(D\,\nabla_{x}\,\cdot\,)$ where
$D$ is some squared matrix depending on the data. The proofs of the
related statements rely strongly on the structure of the collision
operator $\mathcal{Q}$ which is either a bounded operator or a self-adjoint
operator on some weighted version of $L^{2}$.

\smallskip{}

When the operator is less regular or when there is a lack of symmetries
\cite{MR1803225}, the convergence concerns only the mean value $\varrho_{\varepsilon}$
with respect to $\text{{\rm v}}$, called the density. For some function
$\varrho_{0}$ satisfying adequate restrictions, we have: 
\begin{equation}
\varrho_{\varepsilon}(t,x):=\int_{\R^{N}}\, f_{\varepsilon}(t,x,\text{{\rm v}})\ d\text{{\rm v}}\,\longrightarrow\,\varrho_{0}(t,x)\ .\label{integrv}
\end{equation}

\medskip{}

When looking at the structure of $\mathcal{L}^{a}:={}^{t}(\mathcal{L}_{1}^{a},\mathcal{L}_{2}^{a})$,
there is some analogy with (\ref{eqkinetic}). Indeed, the expression
$\mathcal{L}^{a}(\varepsilon,v_{\varepsilon})$ can be decomposed
into: 
\begin{equation}
\mathcal{L}^{a}(\varepsilon,v_{\varepsilon}):=\partial_{t}v_{\varepsilon}+\frac{1}{\varepsilon}\ {\mathcal{T}}\, v_{\varepsilon}+\frac{1}{\varepsilon^{2}}\ {\mathcal{Q}}\, v_{\varepsilon}+\mathcal{LL}(\varepsilon)\, v_{\varepsilon}+\varepsilon^{M-2}\,\mathcal{NL}(\varepsilon,v_{\varepsilon})\label{genericeq}
\end{equation}
 with: \vskip -8mm 
\begin{align*}
\mathcal{Q}=\begin{pmatrix}-\mu\partial_{\theta\theta} & 0\\
\partial_{\theta}h & -\mu\partial_{\theta\theta}
\end{pmatrix}\, & ,\qquad\mathcal{T}:=\begin{pmatrix}h\,\partial_{y}-\lambda\partial_{\theta\theta} & 0\\
0 & h\,\partial_{y}
\end{pmatrix},\\
\mathcal{LL}(\varepsilon):=\begin{pmatrix}-\mu\partial_{yy} & -\lambda\partial_{\theta y}\\
-\lambda\partial_{\theta y} & -(\mu+\varepsilon\lambda)\partial_{yy}
\end{pmatrix}\, & ,\qquad\mathcal{NL}(\varepsilon,v_{\varepsilon}):=\begin{pmatrix}v_{\varepsilon}^{1}\,\partial_{\theta}v_{\varepsilon}^{1}+\varepsilon\, v_{\varepsilon}^{2}\,\partial_{y}v_{\varepsilon}^{1}\\
v_{\varepsilon}^{1}\,\partial_{\theta}v_{\varepsilon}^{2}+\varepsilon\, v_{\varepsilon}^{2}\,\partial_{y}v_{\varepsilon}^{2}
\end{pmatrix}.
\end{align*}
 There are many analogies between (\ref{eqkinetic}) and (\ref{genericeq}).
In both cases, the hierarchy with respect to the negative powers of
$\varepsilon$ (namely $\varepsilon^{-2}$, $\varepsilon^{-1}$ and
$\varepsilon^{0}$) is the same, with in factor operators sharing
analogous structures. Also, the mean value operation $\langle v_{\varepsilon}\rangle$
is when considering~(\ref{genericeq}) what replaces the integration
with respect to $\text{{\rm v}}$ at the level of (\ref{integrv}).
However, there are two important differences when comparing (\ref{eqkinetic})
and (\ref{genericeq})$\,:$

\smallskip{}

- The Equation (\ref{genericeq}) is a system ($v_{\varepsilon}\in\R^{2}$).
When dealing only with the singular part $\varepsilon^{-1}{\mathcal{T}}+\varepsilon^{-2}{\mathcal{Q}}$,
this problem can be circumvented by first solving the equation on
$v_{\varepsilon}^{1}$ and then by plugging the result into the equation
on $v_{\varepsilon}^{2}$. However, once the influences of the contributions
$\mathcal{LL}$ or $\mathcal{NL}$ are incorporated, such strong decoupling
is no more available. When dealing with the full system~\eqref{genericeq},
the discussion must necessarily take into account \textit{vectorial
aspects}.

\smallskip{}

- The operator $\mathcal{Q}$ of (\ref{genericeq}) is neither selfadjoint
nor bounded (on $L^{2}$). Up to some extent, it can be viewed as
a non selfadjoint perturbation of the selfadjoint action $-\mu\partial_{\theta\theta}\, I$.
Still, we can compute the point spectrum $\sigma_{P}(\mathcal{Q})$
of $\mathcal{Q}:L^{2}(\T;\R^{2})\longrightarrow H^{-2}(\T;\R^{2})$.
We find that: 
\[
\sigma_{P}(\mathcal{Q})\,:=\,\bigl\{\,\delta\in{\mathbb{C}}\,;\ \mathcal{Q}-\delta\, I\ \,\text{is not injective}\,\bigr\}\,=\,\bigl\{\,\mu\, n^{2}\,;\ n\in\N\,\bigr\}\,.
\]

From the point of view of central variety theorems, the presence of
a (point) spectral gap between the eigenvalue $0$ and the other (positive)
eigenvalues indicates that there is a separation between two types
of behaviours in time, a \textit{slow} one and a \textit{fast} decaying
one, for instance in the spirit of \cite{MR2055317,MR1127004}. Of
course, such a separation is due to the presence of $-\mu\partial_{\theta\theta}\, I$
inside $\mathcal{Q}$. Again, the influence of this dissipation term
is what relates (\ref{eqkinetic}) and (\ref{genericeq}).

\medskip{}

In other respects, singular systems like (\ref{genericeq}) have been
studied in a purely hyperbolic context, that is when $\mu=\lambda_{}=0$.
Then, the discussion is based on tools coming from supercritical nonlinear
geometric optics \cite{MR2533930,MR2182053,MR2032985}.

\medskip{}

The asymptotic analysis of (\ref{genericeq}) under the assumptions
retained here is clearly at the interface of what is done in \cite{MR802803,MR1803225,MR1127004}
and \cite{MR2533930,MR2182053,MR2032985}. Nevertheless, it needs
to develop a specific approach which is the matter of the current
contribution. In the next paragraph, we give a few indications of
our strategy.


\subsubsection{Heuristic description}

Our analysis of $\mathcal{L}^{a}$ is based on a discrete Fourier
decomposition with respect to $\theta\,\in\T$. We can expand $h$
as well as $v=^{t}(v^{1},v^{2})$ into Fourier Series: 
\[
h=\sum_{k\in\Z^{*}}h_{k}\ e^{i\, k\,\theta}\,,\qquad v^{j}=\sum_{k\in\Z}v_{k}^{j}\ e^{ik\theta}\,,\qquad j\in\{1,2\}\,.
\]
 Introduce the following linear map: 
\[
\left.\begin{array}{rcl}
\widetilde{\Pi}\,:L^{2}(\T;\R^{2}) & \longrightarrow & L^{2}(\T;\R^{2})\\
v={}^{t}(v^{1},v^{2}) & \longmapsto & {\displaystyle ^{t}\Bigl(\, v_{0}^{1}\,,\, v_{0}^{2}\,-\, i\,\sum_{k\in\Z^{*}}\, h_{k}\ \mu^{-1}\ k^{-1}\ v_{0}^{1}\ e^{ik\theta}\,\Bigr)\,.}
\end{array}\right.
\]
The application $\widetilde{\Pi}$ is clearly a projector onto the
kernel of $\mathcal{Q}$. Retain that: 
\[
\widetilde{\Pi}\circ\widetilde{\Pi}=\widetilde{\Pi}\,,\qquad\widetilde{\Pi}\, L^{2}=\text{ker}\,\mathcal{Q}\,,\qquad\text{dim}\ (\text{ker}\,\mathcal{Q})=2\,.
\]

$\bullet$ \underline{\bf a} $\bullet$ To understand the action
of the several operators in $\mathcal{L}^{a}$ defined in \eqref{genericeq},
a first approach is to consider the simplified equation: 
\begin{equation}
\partial_{t}\tilde{v}_{\varepsilon}+\varepsilon^{-2}\,\mathcal{Q}\,\tilde{v}_{\varepsilon}=0\,,\qquad\tilde{v}_{\varepsilon}(0,\theta)=\sum_{k\in\Z}\,\tilde{v}_{k}(0)\ e^{ik\theta}\,.\label{modelfac}
\end{equation}
 The corresponding solution $\tilde{v}_{\varepsilon}={}^{t}(\tilde{v}_{\varepsilon}^{1},\tilde{v}_{\varepsilon}^{2})$
involves components $\tilde{v}_{\varepsilon}^{j}$ which can be put
in the form 
\[
\tilde{v}_{\varepsilon}^{j}(t)\,=\,\sum_{k\in\Z}\,\tilde{v}_{k}^{j}\bigl(t,\frac{t}{\varepsilon^{2}}\bigr)\ e^{ik\theta}\,,\qquad\tau:=\frac{t}{\varepsilon^{2}}\,,\qquad j\in\{1,2\}\,.
\]
 For $k=0$, we find that $\tilde{v}_{0}^{1}(t,\tau)=\tilde{v}_{0}^{1}(0)$
and: 
\[
\tilde{v}_{0}^{2}(t,\tau)=\tilde{v}_{0}^{2}(0)-\, i\,\sum_{p\in\Z^{*}}\, h_{p}\,\mu^{-1}\, p^{-1}\ \tilde{v}_{-p}^{1}(0)+\, i\,\sum_{p\in\Z^{*}}\, h_{p}\,\mu^{-1}\, p^{-1}\,\tilde{v}_{-p}^{1}(0)\ e^{-\mu p^{2}\tau}\,.
\]
 The second (constant) term in $\tilde{v}_{0}^{2}(t,\tau)$ is in
general non zero and it comes from contributions inside $\tilde{v}_{\varepsilon}(0,\cdot)$
which are polarized according to $(I-\widetilde{\Pi})\, L^{2}$. Thus,
even if $\tilde{v}_{\varepsilon}(0,\cdot)$ presses only on the positive
point spectrum, the corresponding solution $\tilde{v}_{\varepsilon}$
is not necessarily exponentially decreasing in time. We can see here
a first effect of the nonselfadjoint part inside $\mathcal{Q}$.

\medskip{}

For $k\in\Z^{*}$, noting $\aleph:=\{p\in\Z^{*};p\not=k,p\not=2k\}$,
we have $\tilde{v}_{k}^{1}(t,\tau)=\tilde{v}_{k}^{1}(0)\, e^{-\mu k^{2}\tau}$
and: 
\[
\begin{array}{l}
\tilde{v}_{k}^{2}(t,\tau)=-\, i\, h_{k}\,\mu^{-1}\, k^{-1}\,\tilde{v}_{0}^{1}(0)-2\, i\, k\, h_{2k}\ \tilde{v}_{-k}^{1}(0)\ \tau\ e^{-\mu k^{2}\tau}\\
\qquad{\displaystyle +\, i\,\sum_{p\in\aleph}\, h_{p}\,\mu^{-1}\,(p-2k)^{-1}\ \tilde{v}_{k-p}^{1}(0)\ e^{-\mu(k-p)^{2}\tau}-\, i\,\sum_{p\in\aleph}\, h_{p}\,\mu^{-1}\,(p-2k)^{-1}\ \tilde{v}_{k-p}^{1}(0)\ e^{-\mu k^{2}\tau}}\\
\qquad+\,\bigl\lbrack\tilde{v}_{k}^{2}(0)+i\, h_{k}\,\mu^{-1}\, k^{-1}\ \tilde{v}_{0}^{1}(0)\bigr\rbrack\ e^{-\mu k^{2}\tau}\,.
\end{array}
\]
 By bringing together all constant terms (in $\tau$) inside an expression
$v_{k}^{s}(t,\theta)$ which here does not depend on $t$, these formulas
fit with a decomposition like (\ref{asymplty}).

\medskip{}

$\bullet$ \underline{\bf b} $\bullet$ Next, consider the more elaborated
model: 
\begin{equation}
\partial_{t}\check{v}_{\varepsilon}+\varepsilon^{-1}\,\mathcal{T}\,\check{v}_{\varepsilon}+\varepsilon^{-2}\,\mathcal{Q}\,\check{v}_{\varepsilon}=0\,,\qquad\check{v}_{\varepsilon}(0,y,\theta)=\sum_{k\in\Z}\,\check{v}_{k}(0,y)\ e^{ik\theta}\,.\label{modelelabo}
\end{equation}
 One can expect that the intermediate singular term $\varepsilon^{-1}\,\mathcal{T}$
produces the scaling $t/\varepsilon$. However, such an effect does
not appear here. On the one hand, the contributions polarized according
to $(I-\widetilde{\Pi})\, L^{2}$ are mainly handled as in paragraph
\underline{\bf a}. On the other hand, the $\widetilde{\Pi}\, L^{2}$
parts disappear by a combination of two arguments: 
\begin{description}
\item [{$\quad$-}] Due to the relation $\int_{\T}hh'd\theta=0$, we can
use the following algebraic identity: 
\begin{equation}
\widetilde{\Pi}\circ\mathcal{T}\circ\widetilde{\Pi}\equiv0\,.\label{structuraleq}
\end{equation}
 
\item [{$\quad$-}] We can absorb the extra term $(I-\widetilde{\Pi})\mathcal{T}\widetilde{\Pi}$
through some ellipticity inside $\mathcal{Q}$. Indeed, in what follows,
we seek $\check{v}_{\varepsilon}$ as an expansion of the form $\check{v}_{\varepsilon}=\check{v}_{0}+\varepsilon\,\check{v}_{1}+O(\varepsilon^{2})$.
Assuming that $\tilde{v}_{0}=\widetilde{\Pi}\tilde{v}_{0}\in\text{ker}\,\mathcal{Q}$,
we can observe that: 
\[
(I-\widetilde{\Pi})\,(\varepsilon^{-1}\,\mathcal{T}+\varepsilon^{-2}\,\mathcal{Q})\,(\check{v}_{0}+\varepsilon\,\check{v}_{1})\,=\,\varepsilon^{-1}\ \bigl\lbrack(I-\widetilde{\Pi})\mathcal{T}\widetilde{\Pi}\,\check{v}_{0}+(I-\widetilde{\Pi})\mathcal{Q}(I-\widetilde{\Pi})\,\check{v}_{1}\bigr\rbrack+O(1)\,.
\]
 Now, the idea is to adjust $\check{v}_{1}$ conveniently in order
to remove the $O(\varepsilon^{-1})$ contribution. 
\end{description}
\smallskip{}

In practice, the implementation of these arguments must be done with
care because the different terms which come into play are more tangled
than what is indicated above.

\medskip{}

Note that a normal form approach (in the spirit of \cite{MR2032985}:
meaning to change $\check{v}$ into $(I+\varepsilon M)\dot{v}$ for
some well adjusted operator $M$), can be tried to get rid of $\mathcal{T}$.
However, such a method seems not to succeed. There are always remaining
$O(\varepsilon^{-1})$ terms and, all things considered, to deal with
the actual diagonal form of $\mathcal{T}$ appears to be more suitable.

\medskip{}

$\bullet$ \underline{\bf c} $\bullet$ Finally, consider the full
system (\ref{genericeq}). Our aim is to describe the asymptotic behaviour
of the family $\{v_{\varepsilon}\}_{\varepsilon}$ on a time scale
of the order $t\simeq1$. To this end, we have to understand the $O(1)$
contributions brought by the singularity $\varepsilon^{-1}\,\mathcal{T}+\varepsilon^{-2}\,\mathcal{Q}$.
This singular term is a perturbation of the self adjoint operator
$\mathcal{Q}_{0}:=\varepsilon^{-2}\mu\partial_{\theta\theta}I$. This
perturbation is of two types. 
\begin{description}
\item [{\quad{}-}] The interactions (at order one) between $\mathcal{Q}_{0}$
and $\varepsilon^{-1}\mathcal{T}$ turns out to be the source of some
creation of diffusion. The mechanism is similar to the one met in
the \textit{drift-diffusion} phenomena. Moreover, $\mathcal{Q}_{0}+\varepsilon^{-1}\mathcal{T}$
is a diagonal operator. The components of the velocity are decoupled
and the discussion deals more with scalar arguments than with vectorial
arguments. 
\item [{\quad{}-}] We perturb $\mathcal{Q}_{0}+\varepsilon^{-1}\mathcal{T}$
by $\begin{pmatrix}0 & 0\\
\partial_{\theta}h & 0
\end{pmatrix}$ at order $0$. A first effect is that $\mathcal{Q}$ is not selfadjoint.
It also induced some strong coupling at order $0$ between the two
components of the velocity. One aspect of the construction is to prove
that this strong coupling do not disrupt the production of dissipation.
The discussion has to take into account vectorial aspects and one
issue is to match the initial data between the slow profile $v_{\varepsilon}^{s}$
and the fast profile $v_{\varepsilon}^{f}$. 
\end{description}
\medskip{}

Moreover, we have to determine the effects of $\mathcal{LL}$ and
$\mathcal{NL}$ which are of two types. First, the presence of $\mathcal{LL}$
and $\mathcal{NL}$ reinforces the coupling. Secondly, it induces
nonlinear interactions which are delicate to deal with. In particular,
in the critical case $M=2$, it becomes necessary to exhibit \textit{transparency
phenomena} in order to achieve the analysis.

\medskip{}

In this article, we propose (Proposition \ref{corapproxoperateur})
and we justify (Theorem \ref{propestimationreste}) a complete expansion
for the family $\{v_{\varepsilon}\}_{\varepsilon}$. It is the occasion
to analyze precisely the linear features and the non linear aspects
alluded above. 

\subsubsection{Heuristic arguments for the energy estimates}

The variable $^{t}(q_{\varepsilon}^{R},v_{\varepsilon}^{R})$ can
be interpreted as the solution of the linearized operator $\mathcal{L}$
at $^{t}(q_{\varepsilon}^{a},v_{\varepsilon}^{a})$ perturbed by some
\textit{small} non-linear terms (when $\nu$ and $M$ are large enough).
It can still be interpreted as the solution of the linearized equation of System \eqref{eqrefthetay}
at point $^{t}(0,0,h)+(\varepsilon^{\nu}q_{\varepsilon}^{a},\varepsilon^{M}v_{\varepsilon}^{a})$ (perturbed b some small non-linear terms).
To underline the difficulties to obtain estimates on a strip independent
of $\varepsilon$ we only consider estimates over $(p_{\varepsilon}^{\ell},u_{\varepsilon}^{\ell})$
solution the linearized system \eqref{eqrefthetay} at point $^{t}(0,0,h(\theta))$:
\begin{equation}
\left\{ \begin{array}{llll}
\partial_{t}p_{\varepsilon}^{\ell} & +\varepsilon^{-1}h\,\partial_{y}p_{\varepsilon}^{\ell} &  & =S_{0}\,,\\
\partial_{t}u_{\varepsilon}^{1\ell} & +\varepsilon^{-1}h\,\partial_{y}u_{\varepsilon}^{1\ell} & {-\widetilde{\mathcal{P}}_{\varepsilon}^{1}u_{\varepsilon}^{\ell}} & =S_{1}\,,\\
\partial_{t}u_{\varepsilon}^{2\ell} & +\varepsilon^{-1}h\,\partial_{y}u_{\varepsilon}^{2\ell}+\fbox{\ensuremath{\varepsilon^{-2}\partial_{\theta}h\ u_{\varepsilon}^{1\ell}}} & {-\widetilde{\mathcal{P}}_{\varepsilon}^{2}u_{\varepsilon}^{\ell}} & =S_{2}\,,
\end{array}\right.\label{eqintrolin1}
\end{equation}
 for some sources $S:={}^{t}(S_{0},S_{1},S_{2})$ in $H^{\infty}(\T\times\R\,;\,\R^{3})$.
It is a parabolic-hyperbolic system singular in $\varepsilon$.\\

\paragraph{Purely hyperbolic approach.}

We first consider that $\mu=\lambda=0$ so that the dissipation vanishes.
We perform classical $L^{2}$-estimates on Equation \eqref{eqintrolin1}.
We obtain: 
\[
\left\Vert \left(p_{\varepsilon}^{\ell},u_{\varepsilon}^{\ell}\right)(t,\cdot)\right\Vert _{L^{2}}\lesssim e^{C_{\varepsilon}t}\sup_{t\in[0,T_{\varepsilon}^{l}]}\left\Vert S\right\Vert _{L^{2}},\quad\text{with}\quad C_{\varepsilon}\leq C\left(1+\textcolor{red}{\varepsilon^{-2}\|\partial_{\theta}h\|_{L^{\infty}}}\right).
\]
 Yet, it indicates that classical energy estimates only provide a
control over the solution for time of order $\varepsilon^{2}$. In
particular for bounded time the solution can exponentially increase
with the time~$t$.

Furthermore, let us consider the singular transport equation: 
\[
\partial_{t}v+\varepsilon h(\theta)\partial_{y}v=0,\qquad v_{|_{t=0}}=v_{0}\,.
\]
 The solution is explicit $v(t,\theta,y)=v_{0}(\theta,y-\varepsilon^{-1}th(\theta))$.
In particular each time we differentiate with respect to $\theta$, we
loss a power of $\varepsilon$. Thus, classical Sobolev space are
not well suited for the control of this family of solutions. We have to introduce
anisotropic Sobolev spaces (defined page \pageref{eqinfininorm})
for both the velocity and the pressure.\\

\paragraph{Parabolic-hyperbolic approach.}

To go further in time, we have to consider the dissipation. 
The operator $\mathcal{P}_{\varepsilon,}$ is positive and satisfies
some coercive estimates. There exists a positive constant $c$ such
that for any function $f\in H^{1}(\T\times\R)$, 
\begin{equation}
\forall\,\varepsilon\in]0,1],\quad-\left\langle \widetilde{\mathcal{P}}_{\varepsilon}f,f\right\rangle \geq c\,\left(\left\Vert \varepsilon^{-1}\partial_{\theta}f\right\Vert _{L^{2}(\T\times\R)}^{2}+\left\Vert \partial_{y}f\right\Vert _{L^{2}(\T\times\R)}^{2}\right):=\Phi_{\varepsilon}(\nabla,f).\label{controldissipation0}
\end{equation}
 It has two consequences. 
\begin{description}
\item [{\quad{}-}] At fixed $\varepsilon$, we should obtain a regularization
of the solution. The velocity $u_{\varepsilon}^{\ell}$ is in $L_{t}^{2}H_{\theta,y}^{1}$
(see Inequality \eqref{eqpropestu2}). This is the \textit{regularization
phenomena}. 
\item [{\quad{}-}] Considering the dependency in $\varepsilon$, the dissipation
should also absorb some singular terms. First the squared term can
be estimated as follows: 
\begin{align*}
\varepsilon^{-1}\left|\varepsilon^{-1}\int_{\T\times\R}\partial_{\theta}h\, v_{\varepsilon}^{1l}\, v_{\varepsilon}^{2l}\, d\theta\, dy\right| & \lesssim\varepsilon^{-1}\left(\|h\|_{L^{\infty}}^{2}\|v_{\varepsilon}^{l}\|_{L^{2}}^{2}+\|\varepsilon^{-1}\partial_{\theta}v_{\varepsilon}^{l}\|_{L^{2}}^{2}\right).
\end{align*}
 Thus it only seems to be singular of order one (in $\varepsilon$)
instead of being singular of order two (in $\varepsilon$). Furthermore
it is also desingularizes the singular transport. We can obtain estimates
over the velocity in the classical Sobolev spaces whereas the pressure
is still estimated in the anisotropic Sobolev spaces. It indicates
that the estimates over the velocity and the pressure have to be done
separately. 
\end{description}

Here, the addition of the parabolic aspect in the discussion still
does not allow us to obtain a control over $(p_{\varepsilon}^{\ell},u_{\varepsilon}^{\ell})$
for time of order one. Some additional arguments are required.\\

\paragraph{Singular change of unknowns.}

To keep on desingularizing the term $\varepsilon^{-2}\partial_{\theta}h\, u_{\varepsilon}^{1\ell}$
we consider the change of unknowns: 
\[
\widetilde{q}_{\varepsilon}^{\ell}:=q_{\varepsilon}^{\ell},\quad\widetilde{u}_{\varepsilon}^{1\ell}:=u_{\varepsilon}^{1\ell},\quad\widetilde{u}_{\varepsilon}^{2\ell}:=\textcolor{red}{\varepsilon}u_{\varepsilon}^{2\ell}.
\]
 The system \eqref{eqintrolin1} becomes: 
\begin{equation}
\left\{ \begin{array}{llll}
\partial_{t}\widetilde{p}_{\varepsilon}^{\ell} & +\varepsilon^{-1}h\,\partial_{y}\widetilde{p}_{\varepsilon}^{\ell} &  & =S_{0}\,,\\
\partial_{t}\widetilde{u}_{\varepsilon}^{1\ell} & +\varepsilon^{-1}h\,\partial_{y}\widetilde{u}_{\varepsilon}^{1\ell} & {-{\mathcal{Q}}_{\varepsilon}^{1}}\widetilde{u}_{\varepsilon}^{\ell} & =S_{1}\,,\\
\partial_{t}\widetilde{u}_{\varepsilon}^{2\ell} & +\varepsilon^{-1}h\,\partial_{y}\widetilde{u}_{\varepsilon}^{2\ell}+{\varepsilon^{-1}\partial_{\theta}h\,\widetilde{u}_{\varepsilon}^{1\ell}} & {-{\mathcal{Q}}_{\varepsilon}^{2}}\widetilde{u}_{\varepsilon}^{\ell} & =\varepsilon S_{2}\,.
\end{array}\right.\label{eqintrolin2}
\end{equation}
 where the operator $\mathcal{Q}_{\varepsilon}$ is defined in Equation
\eqref{eqQdiss}. We can notice that: 
\begin{description}
\item [{\quad{}-}] The term $\varepsilon^{-2}\partial_{\theta}h\, u_{\varepsilon}^{1\ell}$
is desingularized into ${\varepsilon^{-1}\partial_{\theta}h\,\widetilde{u}_{\varepsilon}^{1\ell}}$. 
\item [{\quad{}-}] However, the dissipation is turns into the operator
$\mathcal{Q}_{\varepsilon}$. It can no longer satisfies Inequality
\eqref{controldissipation0}. Assuming $\mu_{}$ is large enough,
it is still true (\textit{c.f.} Lemma~\ref{lemestQ}). 
\end{description}
Thus performing the same estimates for system \eqref{eqintrolin2}
as the one done in the previous case should lead to a control over
$(\widetilde{q}_{\varepsilon}^{\ell},\widetilde{u}_{\varepsilon}^{\ell})$
in $L^{2}$-norm for time of order one ($t\approx1$).\\

\paragraph{Conclusion.}

In Section \ref{Energy}, we justify that those heuristic arguments
work for the complete System~\eqref{eqepsilon}. Some technical arguments
must be added to deal with the complete System \eqref{eqepsilon}.
Indeed, it is obviously nonlinear and the pressure and the velocity
are coupled. Nonlinear terms has to be studied carefully.

As indicated, the pressure is expected to be controlled in anisotropic
Sobolev spaces whereas the velocity is estimated in the classical
Sobolev spaces. Thus it indicates that we have to deal with the problem
of the velocity and the pressure separately. However those variables
are coupled by the term 
\[
-\frac{C\varepsilon^{2\nu-M-R-2}}{2}\,^{t}(\partial_{\theta},\varepsilon\partial_{y})\left(q_{\varepsilon}^{a}+\varepsilon^{R}q_{\varepsilon}^{R}\right)^{2},
\]
 in Equation \eqref{eqepsilon}. The constant $\nu$ has to be large
enough so that the pressure does not interfere too much with the velocity.
Of course an other issue is that the pressure is only estimated in
the anisotropic Sobolev spaces. We can go back to the classical Sobolev
spaces using the equivalence of norms \eqref{eqrefnorm}. It has a
cost in power of $\varepsilon$ for each derivatives to estimate.
It explains why we lose $(m+3)$ precision in the definition of $w_{m}$
(see Equation \eqref{eqws}). 


\subsubsection{Contents}

What follows is divided in two main parts: Section \ref{construction}
and Section \ref{Energy}.

\medskip{}

The Section \ref{construction} is devoted to the construction of
the approximated solutions $(q_{\varepsilon}^{a},v_{\varepsilon}^{a})$.
The first step is to show the Proposition \ref{propconstructionvelocity}. 
\begin{description}
\item [{$\quad\ $-}] In this purpose, the paragraph \ref{subsectionconstructvelocity}
deal with the velocity field $v_{\varepsilon}^{a}$, that is with
the equation $\mathcal{L}^{a}(\varepsilon,v_{\varepsilon}^{a})=O(\varepsilon^{N})$.
The limit case $M=2$ is special because in this situation the non
linear terms can interfere at leading order. 
\item [{$\quad\ $-}] In the paragraph \ref{concerningsolutions}, we are
able to exhibit the control (\ref{eqapproxsolv}). 
\end{description}
The pressure component $q_{\varepsilon}^{a}$ is incorporated at the
level of subsection \ref{sectionconstructionpressure}. Then, the
complete construction of $(q_{\varepsilon}^{a},v_{\varepsilon}^{a})$
can be achieved in the form of Proposition \ref{corapproxoperateur}.

\medskip{}

The Section \ref{Energy} is concerned with energy estimates. We first
state the Proposition of control of the velocity and the pressure.
In particular we deduce  estimates  stated in the
Theorem \ref{propestimationreste}. In the subsection \ref{estvelocity},
we look at the equations $\mathcal{L}_{1}$ and $\mathcal{L}_{2}$.
To this end, we crucially need the properties brought by the dissipation.
In the Subsection \ref{estpressure}, we inject the informations which
have been obtained at the level of $\mathcal{L}_{0}$. By this way,
we can deduce controls concerning the pressure component. 


\section{Construction of the approximated solutions}

\label{construction}


We recall that $M$ is assume to be larger than $2$. This section
is dedicated to the proof of Propositions \ref{propconstructionvelocity},
\ref{propconstructionpressure} and \ref{corapproxoperateur}.


\subsection{Approximated velocity}

\label{subsectionconstructvelocity} 
In this Section, we construct an approximated velocity. Since $M\geq2$,
it follows that non linear effects are present. We are forced to work
with the two time scales $t$ and $\varepsilon^{-2}\, t$ together.
We construct expansions, 
\begin{equation}
v_{\varepsilon}^{s}(t,y,\theta)\,=\sum_{k=0}^{N+1}\,\varepsilon^{k}\, v_{k}^{s}(t,y,\theta)\,,\qquad v_{\varepsilon}^{f}\bigl(\frac{t}{\varepsilon^{2}},y,\theta\bigr)\,=\sum_{k=0}^{N+1}\, v_{k}^{f}\bigl(\frac{t}{\varepsilon^{2}},y,\theta\bigr)\,,\label{lesveps}
\end{equation}
 and plug the expression $v_{\varepsilon}^{s}+v_{\varepsilon}^{f}$
into $\mathcal{L}^{a}$ at the level of (\ref{asymplty}). We obtain:
\begin{align}\label{eqtimecouple}
\partial_{t}v_{\varepsilon}^{s}(t,.)+\varepsilon^{-1}\, h\,\partial_{y}v_{\varepsilon}^{s}(t,.) & +\varepsilon^{M-2}\left(v_{\varepsilon}^{1s}(t,.)\partial_{\theta}v_{\varepsilon}^{s}(t,.)+\varepsilon v_{\varepsilon}^{2s}(t,.)\partial_{y}v_{\varepsilon}^{s}(t,.)\right)\nonumber \\
 & +\,\varepsilon^{-2}\ {}^{t}\left(0,\partial_{\theta}h\, v_{\varepsilon}^{1s}(t,.)\right)-\widetilde{\mathcal{P}}_{\varepsilon}v_{\varepsilon}^{s}(t,.)\nonumber \\
+\,\varepsilon^{-2}\,\partial_{\tau}v_{\varepsilon}^{f}(t/\varepsilon^{2},.)+\varepsilon^{-1}h\,\partial_{y}v_{\varepsilon}^{f}(t/\varepsilon^{2},.) & +\,\varepsilon^{M-2}\, v_{\varepsilon}^{1f}(t/\varepsilon^{2},.)\,\partial_{\theta}v_{\varepsilon}^{f}(t/\varepsilon^{2},.)\nonumber \\
 & \hspace{-3cm}+\,\varepsilon^{M-1}\, v_{\varepsilon}^{2f}(t/\varepsilon^{2},.)\,\partial_{y}v_{\varepsilon}^{f}(t/\varepsilon^{2},.)+\,\varepsilon^{-2}\ {}^{t}\left(0,\partial_{\theta}h\, v_{\varepsilon}^{1f}(t/\varepsilon^{2},.)\right)-\widetilde{\mathcal{P}}_{\varepsilon}v_{\varepsilon}^{f}(t/\varepsilon^{2},.)\nonumber \\
 & \hspace{-3cm}+\,\varepsilon^{M-2}\left(v_{\varepsilon}^{1s}(t,.)\partial_{\theta}v_{\varepsilon}^{f}(t/\varepsilon^{2},.)+\varepsilon v_{\varepsilon}^{2s}(t,.)\partial_{y}v_{\varepsilon}^{f}(t/\varepsilon^{2},.)\right)\nonumber \\
 & \hspace{-3cm}+\,\varepsilon^{M-2}\, v_{\varepsilon}^{1f}(t/\varepsilon^{2},.)\,\partial_{\theta}v_{\varepsilon}^{s}(t,.)+\,\varepsilon^{M-1}\, v_{\varepsilon}^{2f}(t/\varepsilon^{2},.)\,\partial_{y}v_{\varepsilon}^{s}(t,.)\,.
\end{align}
 Fix any time $T\in\R_{+}^{*}$. Define 
\begin{equation}
\mathcal{L}^{as}(\varepsilon,v_{\varepsilon}^{s}):=\partial_{t}v_{\varepsilon}^{s}+\varepsilon^{-1}\, h\,\partial_{y}v_{\varepsilon}^{s}+\varepsilon^{M-2}\left(v_{\varepsilon}^{1s}\partial_{\theta}v_{\varepsilon}^{s}+\varepsilon v_{\varepsilon}^{2s}\partial_{y}v_{\varepsilon}^{s}\right)+\varepsilon^{-2}\ {}^{t}\left(0,\partial_{\theta}h\, v_{\varepsilon}^{1s}\right)-\widetilde{\mathcal{P}}_{\varepsilon}v_{\varepsilon}^{s}\,.\label{eqlas}
\end{equation}
 Recall that $v_{\varepsilon}^{f}\in\mathcal{E}_{\delta}^{\infty}$
is assumed to be exponentially decreasing with respect to $\tau\in\R_{+}$.
Since 
\[
e^{-\delta\, t/\varepsilon^{2}}\,=\, O\bigl((\varepsilon^{2}/t)^{\infty}\bigr)\,=\, O(\varepsilon^{N})\,,\qquad\forall\,(t,N)\in\,]0,T]\times\N\,,
\]
 when looking at the Equation (\ref{eqtimecouple}) for times $t\in\,]0,T]$
with in view a precision of the size $O(\varepsilon^{N})$, all terms
involving $v_{\varepsilon}^{f}$ can be neglected. Now, the idea is
simply to extend this (relaxed) smallness requirement on the whole
interval $[0,T]$. Briefly, we seek $v_{\varepsilon}^{s}$ so that
\begin{equation}
\mathcal{L}^{as}(\varepsilon,v_{\varepsilon}^{s})\,=\, O(\varepsilon^{N})\,,\qquad t\in[0,T]\,.\label{eqapproxt}
\end{equation}

The Equation (\ref{eqapproxt}) can be completed with some initial
data 
\begin{equation}
v_{\varepsilon}^{s}(0,\theta,y)\,=\, v_{\varepsilon}^{s0}(\theta,y)\,=\,\sum_{k=0}^{N+1}\,\varepsilon^{k}\ v_{k}^{s0}(\theta,y)\,.\label{someinidt}
\end{equation}
 Clearly, it suffices to specify $v_{\varepsilon}^{s0}$ to determine
what is $v_{\varepsilon}^{s}(t,\cdot)$ for $t\in[0,T]$, by solving
the Cauchy problem (\ref{eqapproxt})-(\ref{someinidt}). Now, in
order to select $v_{\varepsilon}^{s0}$ conveniently, we have to take
into account what happens for small times, in a boundary layer of
size $\varepsilon^{2}$ near $t=0$. To understand why, just come
back to the study of (\ref{eqtimecouple}) for $t\simeq\varepsilon^{2}$
or $\tau\simeq1$. Then, the contributions brought by $v_{\varepsilon}^{f}$
can no more be neglected. Considering (\ref{eqtimecouple}) with the
information (\ref{eqapproxt}) in mind, it seems natural to impose
\begin{equation}
\mathcal{L}^{af}(\varepsilon,v_{\varepsilon}^{f},v_{\varepsilon}^{s})\,=\, O(\varepsilon^{N})\,,\qquad\tau\in[0,1]\label{eqapproxtbiz}
\end{equation}
 where we have introduced 
\begin{align}\label{eqlaf}
\mathcal{L}^{af}(\varepsilon,v_{\varepsilon}^{f},v_{\varepsilon}^{s}):=\varepsilon^{-2}\partial_{\tau}v_{\varepsilon}^{f}(\tau,.) & +\varepsilon^{-1}h\,\partial_{y}v_{\varepsilon}^{f}(\tau,.)-\widetilde{\mathcal{P}}_{\varepsilon}v_{\varepsilon}^{f}(\tau,.)\nonumber \\
 & +\varepsilon^{M-2}\left(v_{\varepsilon}^{1f}(\tau,.)\partial_{\theta}v_{\varepsilon}^{f}(\tau,.)+\varepsilon v_{\varepsilon}^{2f}(\tau,.)\partial_{y}v_{\varepsilon}^{f}(\tau,.)\right)+^{t}\left(0,\partial_{\theta}h\, v_{\varepsilon}^{1f}(\tau,.)\right)\nonumber \\
 & +\varepsilon^{M-2}\left(v_{\varepsilon}^{1s}(\varepsilon^{2}\tau,.)\partial_{\theta}v_{\varepsilon}^{f}(\tau,.)+\varepsilon v_{\varepsilon}^{2s}(\varepsilon^{2}\tau,.)\partial_{y}v_{\varepsilon}^{f}(\tau,.)\right)\nonumber \\
 & +\varepsilon^{M-2}\left(v_{\varepsilon}^{1f}(\tau,.)\partial_{\theta}v_{\varepsilon}^{s}(\varepsilon^{2}\tau,.)+\varepsilon v_{\varepsilon}^{2f}(\tau,.)\partial_{y}v_{\varepsilon}^{s}(\varepsilon^{2}\tau,.)\right)\,.
\end{align}
 Assume that the data $v_{\varepsilon}^{s0}$ is known. The Cauchy
problem (\ref{eqapproxt})-(\ref{someinidt}) furnishes $v_{\varepsilon}^{s}(t,\cdot)$
for $t\in[0,T]$. In particular, it gives access to all derivatives
$(\partial_{\tau})^{l}v_{\varepsilon}^{\ast s}(0,.)$ with $l\in\N$.
Therefore, we can go further in the analysis by replacing in $\mathcal{L}^{af}(\varepsilon,v_{\varepsilon}^{f},v_{\varepsilon}^{s})$
all the expressions $v_{\varepsilon}^{\ast s}(\varepsilon^{2}\tau,.)$
by their corresponding Taylor expansions (say up to the order $N-1$)
near $t=0$. As long as $\tau\in\R$ is fixed, this operation is justified.
From now on, we look at 
\begin{equation}
\mathcal{L}^{aft}(\varepsilon,v_{\varepsilon}^{f})\,=\, O(\varepsilon^{N})\,,\qquad\tau\in\R_{+}\label{eqapproxtau}
\end{equation}
 where the definition of $\mathcal{L}^{aft}$ is 
\begin{align}
\mathcal{L}^{aft}(\varepsilon,v_{\varepsilon}^{f}): & =\,\varepsilon^{-2}\,\partial_{\tau}v_{\varepsilon}(\tau,.)+\varepsilon^{-1}h\,\partial_{y}v_{\varepsilon}(\tau,.)-\widetilde{\mathcal{P}}_{\varepsilon}v_{\varepsilon}(\tau,.)\nonumber \\
 & \hspace{-1cm}+\varepsilon^{M-2}\left(v_{\varepsilon}^{1}(\tau,.)\partial_{\theta}v_{\varepsilon}(\tau,.)+\varepsilon v_{\varepsilon}^{2}(\tau,.)\partial_{y}v_{\varepsilon}(\tau,.)\right)+\varepsilon^{-2}\ {}^{t}\left(0,\partial_{\theta}h\, v_{\varepsilon}^{1}(\tau,.)\right)\nonumber \\
 & \hspace{-1cm}+\varepsilon^{M-2}\left(\sum_{l=0}^{N-1}\frac{(\varepsilon^{2}\tau)^{l}}{l!}(\partial_{t})^{l}v_{\varepsilon}^{1s}(0,.)\partial_{\theta}v_{\varepsilon}(\tau,.)+\varepsilon\sum_{l=0}^{N-1}\frac{(\varepsilon^{2}\tau)^{l}}{l!}(\partial_{t})^{l}v_{\varepsilon}^{2s}(0,.)\partial_{y}v_{\varepsilon}(\tau,.)\right)\nonumber \\
 & \hspace{-1cm}+\varepsilon^{M-2}\left(v_{\varepsilon}^{1}(\tau,.)\sum_{l=0}^{N-1}\frac{(\varepsilon^{2}\tau)^{l}}{l!}(\partial_{t})^{l}\partial_{\theta}v_{\varepsilon}^{s}(0,.)+\varepsilon\, v_{\varepsilon}^{2}(\tau,.)\sum_{l=0}^{N-1}\frac{(\varepsilon^{2}\tau)^{l}}{l!}(\partial_{t})^{l}\partial_{y}v_{\varepsilon}^{s}(0,.)\right)\,.
\end{align}
 To be coherent with (\ref{asymplty}), we have to impose 
\begin{equation}
v_{\varepsilon}^{f}(0,\cdot)\,=\, v_{\varepsilon}^{f0}(\cdot)\,:=\,(v_{\varepsilon}^{a}-v_{\varepsilon}^{s})(0,\cdot)\,,\label{someinidtenc}
\end{equation}
 where $v_{\varepsilon}^{a}(0,\cdot)$ and $v_{\varepsilon}^{s}(0,\cdot)$
are prescribed as indicated in lines (\ref{prescribed}) and (\ref{someinidt}).
Now, we can recover some $v_{\varepsilon}^{f}(\tau,\cdot)$ for $\tau\in\R_{+}$
by solving the Cauchy problem (\ref{eqapproxtau})-(\ref{someinidtenc}).

\medskip{}

The difficulty comes from the condition $v_{\varepsilon}^{f}\in\mathcal{E}_{\delta}^{\infty}$.
Nothing guarantees that the criterion $v_{\varepsilon}^{f}\in\mathcal{E}_{\delta}^{\infty}$
can be verified for some well-chosen $v_{\varepsilon}^{s0}$. To show
the existence and the uniqueness of such a data $v_{\varepsilon}^{s0}$
is in fact what matters. The extraction of an adequate function $v_{\varepsilon}^{s0}$
is clarified in the construction described below.

\medskip{}

For the sake of brevity, for $k\geq M$, introduce the following notations:
\begin{align}
 & \mathcal{J}(M,k):=\left\{ (i,j)\in\llbracket0,N+1\rrbracket^{2}\ ;\ i+j=k-M\right\} \,,\label{mathcalJ(M,k)}\\
 & \mathcal{I}(M,k):=\left\{ (i,j,l)\in\llbracket0,N+1\rrbracket^{2}\times\ \llbracket0,N-1\rrbracket\ ;\ i+j+2l=k-M\right\} \,.\label{mathcalI(M,k)}
\end{align}
 For $k<M$, we set $\mathcal{J}(M,k)=\emptyset$ and $\mathcal{I}(M,k)=\emptyset$.
We also adopt the conventions $v_{k}^{s}\equiv v_{k}^{f}\equiv0$
for $k=-3$, $k=-2$ and $k=-1$. Let us now go into the details of
the BKW calculus.

\medskip{}

The first step is to inject some expansion $v_{\varepsilon}^{s}$
like in (\ref{lesveps}) into the Equation (\ref{eqapproxt}). By
this way, we can obtain a cascade of equations concerning the unknowns
$v_{k}^{s}={}^{t}(v_{1,k}^{s},v_{2,k}^{s})$. More precisely, for
$k\in\llbracket0,N+3\rrbracket$, we have to consider \begin{subequations}
\label{equtjnl} 
\begin{align}
\partial_{t}v_{1,k-2}^{s}+h\,\partial_{y}v_{1,k-1}^{s} & +\sum_{(i,j)\in\mathcal{J}(M,k)}v_{1,i}^{s}\partial_{\theta}v_{1,j}^{s}+\sum_{(i,j)\in\mathcal{J}(M+1,k)}v_{2,i}^{s}\partial_{y}v_{1,j}^{s}\nonumber \\
 & =\mu\partial_{\theta\theta}v_{1,k}^{s}+\mu\partial_{yy}v_{1,k-2}^{s}+\lambda\partial_{\theta\theta}v_{1,k-1}^{s}+\lambda\partial_{\theta y}v_{2,k-2}^{s}\,,\label{equtjnl:1}\\
\  & \ \nonumber \\
\partial_{t}v_{2,k-2}^{s}+h\,\partial_{y}v_{2,k-1}^{s} & +\sum_{(i,j)\in\mathcal{J}(M,k)}v_{1,i}^{s}\partial_{\theta}v_{2,j}^{s}+\sum_{(i,j)\in\mathcal{J}(M+1,k)}v_{2,i}^{s}\partial_{y}v_{2,j}^{s}+\partial_{\theta}h\, v_{1,k}^{s}\nonumber \\
 & =\mu\partial_{\theta\theta}v_{2,k}^{s}+\mu\partial_{yy}v_{2,k-2}^{s}+\lambda\partial_{\theta y}v_{1,k-2}^{s}+\lambda\partial_{yy}v_{2,k-3}^{s}\,.\label{equtjnl:2}
\end{align}
 \end{subequations} The next step is to plug some expansion $v_{\varepsilon}^{f}$
like in (\ref{lesveps}) into the Equation (\ref{eqapproxtau}). By
this way, we can obtain a cascade of equations concerning the unknowns
$v_{k}^{f}={}^{t}(v_{1,k}^{f},v_{2,k}^{f})$. More precisely, for
$k\in\llbracket0,N+1\rrbracket$, we have to consider \begin{subequations}
\label{equtaujnl} 
\begin{align}
\partial_{\tau}v_{1,k}^{f}+h\,\partial_{y}v_{1,k-1}^{f} & +\sum_{(i,j)\in\mathcal{J}(M,k)}v_{1,i}^{f}\partial_{\theta}v_{1,j}^{f}+\sum_{(i,j)\in\mathcal{J}(M+1,k)}v_{2,i}^{f}\partial_{y}v_{1,j}^{f}\nonumber \\
 & +\sum_{(i,j,l)\in\mathcal{I}(M,k)}\left(v_{1,i}^{f}\left(\partial_{t}\right)^{l}(\partial_{\theta}v_{1,j}^{s})(0)+\left(\partial_{t}\right)^{l}(v_{1,i}^{s})(0)\partial_{\theta}v_{1,j}^{f}\right)\frac{\tau^{l}}{l!}\nonumber \\
 & +\sum_{(i,j,l)\in\mathcal{I}(M+1,k)}\left(v_{2,i}^{f}\left(\partial_{t}\right)^{l}(\partial_{y}v_{1,j}^{s})(0)+\left(\partial_{t}\right)^{l}(v_{2,i}^{s})(0)\partial_{y}v_{1,j}^{f}\right)\frac{\tau^{l}}{l!}\nonumber \\
 & =\mu\partial_{\theta\theta}v_{1,k}^{f}+\mu\partial_{yy}v_{1,k-2}^{f}+\lambda\partial_{\theta\theta}v_{1,k-1}^{f}+\lambda\partial_{\theta y}v_{2,k-2}^{f}\,,\label{equtaujnl:1}\\
\  & \ \nonumber \\
\partial_{\tau}v_{2,k}^{f}+h\, v_{2,k-1}^{f}+\partial_{\theta}h\, v_{1,k}^{f} & +\sum_{(i,j)\in\mathcal{J}(M,k)}v_{1,i}^{f}\partial_{\theta}v_{2,j}^{f}+\sum_{(i,j)\in\mathcal{J}(M,k+1)}v_{2,i}^{f}\partial_{\theta}v_{2,j}^{f}\nonumber \\
 & +\sum_{(i,j,l)\in\mathcal{I}(M,k)}\left(v_{1,i}^{f}\left(\partial_{t}\right)^{l}(\partial_{\theta}v_{2,j}^{s})(0)+\left(\partial_{t}\right)^{l}(v_{1,i}^{s})(0)\partial_{\theta}v_{2,j}^{f}\right)\frac{\tau^{l}}{l!}\nonumber \\
 & +\sum_{(i,j,l)\in\mathcal{I}(M+1,k)}\left(v_{2,i}^{f}\left(\partial_{t}\right)^{l}(\partial_{y}v_{2,j}^{s})(0)+\left(\partial_{t}\right)^{l}(v_{2,i}^{s})(0)\partial_{y}v_{2,j}^{f}\right)\frac{\tau^{l}}{l!}\nonumber \\
 & =\mu\partial_{\theta\theta}v_{2,k}^{f}+\mu\partial_{yy}v_{2,k-2}^{f}+\lambda\partial_{\theta y}v_{1,k-2}^{f}+\lambda\partial_{yy}v_{2,k-3}^{f}\,.\label{equtaujnl:2}
\end{align}
 \end{subequations} In view of (\ref{prescribed}), we can associate
(\ref{equtjnl}) and (\ref{equtaujnl}) with initial data $v_{k}^{s0}$
and $v_{k}^{f0}$ satisfying the restriction: 
\begin{align}
v_{k}^{s}(0,\theta,y)+v_{k}^{f}(0,\theta,y)\,=\, v_{k}^{0}(\theta,y)\,,\qquad\forall\, k\in\llbracket0,N+1\rrbracket\,.\label{eqiniu0}
\end{align}
 \begin{proposition}{[}Solving (\ref{equtjnl}) and (\ref{equtaujnl})
together with (\ref{eqiniu0}) and the condition $v_{k}^{f}\in\mathcal{E}_{\delta}^{\infty}${]}
\label{propprofilenonlinear} Fix a time $T\in\R_{+}^{*}$, a number
$\delta\in\,]0,\mu[$ and, for all $k\in\llbracket0,\ldots,N+1\rrbracket$,
functions $v_{k}^{0}\in H^{\infty}(\T\times\R)$. Then, the conditions
(\ref{equtjnl}$\star$), (\ref{equtaujnl}$\star$) and (\ref{eqiniu0})
have a unique solution such that 
\begin{equation}
(v_{k}^{s},v_{k}^{f})\in\mathcal{H}_{T}^{\infty}\times\mathcal{E}_{\delta}^{\infty}\,,\qquad\forall\, k\in\left\llbracket 0,N+1\right\rrbracket \,.\label{solsatistezr}
\end{equation}
 Moreover, the component $v_{k}^{s}$ can be identified through the
homogenized Equation $(\ref{eqhomogenizationg})$. \end{proposition}

The proof relies on some induction on the size of $N$, based on the
following hypothesis of induction: 
\begin{equation}
\mathcal{HN}(N)\,:\ \text{{\it " The Proposition \ref{propprofilenonlinear} is verified up to the integer \ensuremath{N}".}}\label{recurezrnl}
\end{equation}

To go from $N$ up to $N+1$, we will need a succession of lemmas
which are produced below. Before going into the details of the analysis,
we give below a brief description of what happens depending on the
choice of $M$. 
\begin{description}
\item [{$\quad\ $-}] When $M\geq3$, the non linear terms are rather small,
the construction is somehow linear. 
\item [{$\quad\ $-}] When $M=2$, the non linearity becomes critical and
a few arguments must be added. For instance, if we write the Equation
(\ref{equtjnl:1}) for $k=2$, we can notice a Burgers' term 
\begin{equation}
\partial_{t}v_{1,0}^{s}+h\,\partial_{y}v_{1,1}^{s}+\, v_{1,0}^{s}\,\partial_{\theta}v_{1,0}^{s}=\mu\,\partial_{\theta\theta}v_{1,2}^{s}+\mu\,\partial_{yy}v_{1,0}^{s}+\lambda\partial_{\theta\theta}v_{1,1}^{s}+\lambda\partial_{\theta y}v_{2,0}^{s}\label{burgert}
\end{equation}
 and also two contributions $\partial_{\theta\theta}v_{1,2}^{s}$
and $\partial_{\theta\theta}v_{1,1}^{s}$ to be calculated in function
of $v_{1,0}^{s}$, with apparently a non linear dependence with respect
to $v_{1,0}^{s}$. 
\item [{$\quad\ $-}] When $M=2$ again, another effect of the non linear
interactions is the apparition at the level of (\ref{eqlemnonhomogenizationl})
below together with (\ref{eqSPK}), of a source term $S_{k}^{nl\sslash}$
which can depend on $v_{k}^{s}$. 
\item [{$\quad\ $-}] However, there are \textit{transparency phenomena}
at work which come from the initialization procedure. Indeed, knowing
that $v_{k}^{s}\equiv0$ for $k\in\{-3,-2,-1\}$, the Equation (\ref{equtjnl})
in the case $k=0$ and $M\geq2$ reduces to $\mu\,\partial_{\theta\theta}v_{0}^{s}=0$.
In other words, we have to impose 
\begin{equation}
v_{1,0}^{s\perp}\,\equiv\,(I-\Pi)\, v_{1,0}^{s}\,\equiv\,0\,.\label{fondacriterion}
\end{equation}
 It follows that $v_{1,0}^{s}\,\partial_{\theta}v_{1,0}^{s}\equiv0$.
In the same way, all other apparent non linear contributions will
disappear. Therefore, the remaining term $\Pi v_{1,0}^{s}$ can be
determined apart without seeing any non linear effect. 
\end{description}

\subsubsection{Technical Lemmas}


\paragraph{A consequence of the (point) spectral gap.}

The system (\ref{equtaujnl}) is made of two evolution equations of
parabolic type, based on $\partial_{\tau}-\mu\partial_{\theta\theta}$.
This falls under the following framework. \begin{lemma}{[}Fast decreasing
under a polarization condition{]}\label{lemestsol} Let $m\in\N$
and $\delta\in\,]0,\mu[$. Select $w_{0}\in H^{m+2}(\T\times\R)$
and $S_{0}\in(\mathcal{E}_{\delta}^{m+2})^{\perp}(\T\times\R)$, that
is such that $\Pi S_{0}=0$. Consider the initial value problem: 
\begin{equation}
\,\partial_{\tau}w-\mu\,\partial_{\theta\theta}w=S_{0}\,,\qquad w_{|_{\tau=0}}=w_{0}\,.\label{eqwS}
\end{equation}
 For all $T\in\R_{+}$, there is a unique solution $w\in C^{1}\bigl([0,T];H^{m}(\T\times\R)\bigr)$
to the Cauchy problem~(\ref{eqwS}). Moreover, if the initial data
is well prepared in the sense that $\Pi w_{0}=0$, then $\Pi w=0$
for all $t\in[0,T]$ and $w\in\mathcal{E}_{\delta}^{m}(\T\times\R)$.
\end{lemma} The proof of Lemma \ref{lemestsol} is very easy. It
will not be detailed.

\paragraph{Interpretation of the system (\ref{equtjnl}).}

The Lemma below is intended to look at the system (\ref{equtjnl})
otherwise. Indeed, there is a difficulty when dealing with (\ref{equtjnl})
since the knowledge of $v_{k-2}^{s}$ seems to require the determination
of $v_{k-1}^{s}$ and $v_{k}^{s}$, that is the identification of
terms $v_{j}^{s}$ with indices $j$ \textit{greater} than $k-2$.
An important remark is that such a dependence can disappear when the
projector $\Pi$ is applied. The possible dependence in the terms
$v_{j}^{s}$ are recorded in the source terms $S_{k}^{nl\sslash}$
(with explicit formulas). For this occasion, the two cases $M\geq3$
and $M=2$ must be distinguished. This fact can be formulated in the
following way.

\begin{lemma}{[}Non-linear homogenization{]}\label{lemhomogeneisationnonlinear}
Assume that the functions $v_{k}^{s}$ with $k\in\llbracket0,N+3\rrbracket$
are solutions of the system (\ref{equtjnl}). Then, for all $k\in\llbracket0,N+1\rrbracket$,
the part $\Pi v_{k}^{s}={}^{t}(\Pi v_{1,k}^{s},\Pi v_{2,k}^{s})$
is a solution of 
\begin{equation}
\left\{ \begin{array}{l}
\partial_{t}\Pi v_{1,k}^{s}-\left(\mu+\frac{1}{\mu}\,\Pi((\partial_{\theta}^{-1}h)^{2})\right)\partial_{yy}\Pi v_{1,k}^{s}=S_{1,k}^{l\sslash}+S_{1,k}^{nl\sslash}+P_{1}^{l\perp}(I-\Pi)v_{1,k}^{s}\,,\\
\partial_{t}\Pi v_{2,k}^{s}-\left(\mu+\frac{1}{\mu}\,\Pi((\partial_{\theta}^{-1}h)^{2})\right)\partial_{yy}\Pi v_{2,k}^{s}=S_{2,k}^{l\sslash}+S_{2,k}^{nl\sslash}+P_{2}^{l\perp}(I-\Pi)v_{1,k}^{s}\\
\hspace{7,15cm}+\, P_{2}^{l\sslash}\Pi v_{1,k}^{s}+Q_{2}^{l\perp}(I-\Pi)v_{2,k}^{s}\,.
\end{array}\right.\label{eqlemnonhomogenizationl}
\end{equation}
 In the above system (\ref{eqlemnonhomogenizationl}), the four operators
$P_{1}^{l\perp}$, $P_{2}^{l\perp}$, $P_{2}^{l\sslash}$ and $Q_{2}^{l\perp}$,
as well as the source term $S_{k}^{l\sslash}={}^{t}(S_{1,k}^{l\sslash},S_{2,k}^{l\sslash})$,
are defined along lines (\ref{eqS1J})-$\,\cdots\,$-(\ref{c'estpos}).
On the other hand, the contribution $S_{k}^{nl\sslash}={}^{t}(S_{1,k}^{nl\sslash},S_{2,k}^{nl\sslash})$
is given by (\ref{eqSPnlpa1})-(\ref{eqSPnlpa2}). \end{lemma} \smallskip{}
\begin{proof}[Proof of Lemma \ref{lemhomogeneisationnonlinear}]
The matter is to identify the contributions brought by the non linear
terms. Let us project $(\ref{equtjnl})$ according to $\mathcal{V}^{\sslash}$.
Since $\mathcal{J}(M+1,k+2)\equiv\mathcal{J}(M,k+1)$, this yields
\begin{align*}
\partial_{t}\Pi v_{1,k}^{s} & +\Pi\left(h\,\partial_{y}v_{1,k+1}^{s}\right)+\Pi\left(\sum_{(i,j)\in\mathcal{J}(M,k+2)}v_{1,i}^{s}\,\partial_{\theta}v_{1,j}^{s}\right)\\
 & +\Pi\left(\sum_{(i,j)\in\mathcal{J}(M,k+1)}v_{2,i}^{s}\,\partial_{y}v_{1,j}^{s}\right)=\mu\,\partial_{yy}\Pi v_{1,k}^{s}\,,\\
\partial_{t}\Pi v_{2,k}^{s} & +\Pi\left(h\,\partial_{y}v_{2,k+1}^{s}\right)+\Pi\left(\sum_{(i,j)\in\mathcal{J}(M,k+2)}v_{1,i}^{s}\,\partial_{\theta}v_{2,j}^{s}\right)\\
 & +\Pi\left(\sum_{(i,j)\in\mathcal{J}(M,k+1)}v_{2,i}^{s}\,\partial_{y}v_{2,j}^{s}\right)+\Pi\left(\partial_{\theta}h\, v_{1,k+2}^{s}\right)=\mu\,\partial_{yy}\Pi v_{2,k}^{s}+\lambda\,\partial_{yy}\Pi v_{2,k-1}^{s}\,.
\end{align*}
 In what follows, we will use the system (\ref{equtjnl}) and many
integrations by parts in order to interpret $\Pi(h\,\partial_{y}v_{1,k+1}^{s})$,
$\Pi\left(h\,\partial_{y}v_{2,k+1}^{s}\right)$ and $\Pi\left(\partial_{\theta}h\, v_{1,k+2}^{s}\right)$.
The goal is to show that these quantities can be expressed in terms
of the $v_{j}^{s}$ with $j\in\llbracket0,k\rrbracket$.\\

$\circ$ Study of $\Pi(h\,\partial_{y}v_{1,k+1}^{s})$. Exploiting
(\ref{equtjnl:1}) with $k+1$ in place of $k$, we can deduce 
\begin{align*}
\Pi(h\,\partial_{y}v_{1,k+1}^{s}) & =\Pi\bigl(\partial_{\theta\theta}^{-2}(h)\,\partial_{y}\partial_{\theta\theta}v_{1,k+1}^{s}\bigr)\,,\\
 & =\frac{1}{\mu}\,\Pi\bigl(\partial_{\theta\theta}^{-2}(h)\ \partial_{y}(\partial_{t}v_{1,k-1}^{s}-\mu\partial_{yy}v_{1,k-1}^{s}-\lambda\partial_{\theta y}v_{2,k-1}^{s})\bigr)\\
 & +\frac{1}{\mu}\,\Pi\left(\partial_{\theta\theta}^{-2}(h)\ \partial_{y}\left(\sum_{(i,j)\in\mathcal{J}(M,k+1)}v_{1,i}^{s}\,\partial_{\theta}v_{1,j}^{s}+\sum_{(i,j)\in\mathcal{J}(M,k)}v_{2,i}^{s}\,\partial_{y}v_{1,j}^{s}\right)\right)\\
 & +\frac{1}{\mu}\,\Pi\bigl(h\,\partial_{\theta\theta}^{-2}(h)\,\partial_{yy}v_{1,k}^{s}\bigr)-\frac{\lambda}{\mu}\,\Pi\bigl(\partial_{\theta\theta}^{-2}(h)\,\partial_{y\theta\theta}v_{1,k}^{s}\bigr)\,.
\end{align*}
 Recalling \eqref{convenpre} and \eqref{convensec}, we have to deal
with 
\begin{align*}
\Pi(h\,\partial_{y}v_{1,k+1}^{s})= & -\frac{1}{\mu}\,\Pi\left((\partial_{\theta}^{-1}h)^{2}\right)\partial_{yy}\Pi v_{1,k}^{s}-S_{1,k}^{l\sslash}-P_{1}^{l\perp}(I-\Pi)v_{1,k}^{s}\\
 & +\frac{1}{\mu}\,\Pi\left(\partial_{\theta\theta}^{-2}(h)\ \partial_{y}\left(\sum_{(i,j)\in\mathcal{J}(M,k+1)}v_{1,i}^{s}\,\partial_{\theta}v_{1,j}^{s}+\sum_{(i,j)\in\mathcal{J}(M,k)}v_{2,i}^{s}\,\partial_{y}v_{1,j}^{s}\right)\right)\,,
\end{align*}
 with the conventions, 
\begin{align}
S_{1,k}^{l} & :=-\frac{1}{\mu}\,\partial_{\theta\theta}^{-2}(h)\,\partial_{y}\left(\partial_{t}v_{1,k-1}^{s}-\mu\partial_{yy}v_{1,k-1}^{s}-\lambda\partial_{\theta y}v_{2,k-1}^{s}\right)\,,\label{eqS1J}\\
P_{1}^{l}f & :=\frac{\lambda}{\mu}\,\Pi\left(\partial_{\theta\theta}^{-2}(h)\,\partial_{\theta\theta y}f\right)-\frac{1}{\mu}\,\Pi\left(h\,\partial_{\theta\theta}^{-2}(h)\,\partial_{yy}f\right)\,.\label{eqP1GFD}
\end{align}

\smallskip{}

$\circ$ Study of $\Pi(h\,\partial_{y}v_{2,k+1}^{s})$. Exploiting
(\ref{equtjnl:2}) with $k+1$ in place of $k$, we can derive 
\begin{align*}
\Pi\left(h\,\partial_{y}v_{2,k+1}^{s}\right) & =\Pi\bigl(\partial_{\theta\theta}^{-2}(h)\,\partial_{y}\partial_{\theta\theta}v_{2,k+1}^{s}\bigr)\,,\\
 & =\frac{1}{\mu}\,\Pi\left(\partial_{\theta\theta}^{-2}(h)\,\partial_{y}\left(\partial_{t}v_{2,k-1}^{s}-\mu\partial_{yy}v_{2,k-1}^{s}-\lambda\partial_{\theta y}v_{1,k-1}^{s}-\lambda\partial_{yy}v_{2,k-2}^{s}\right)\right)\\
 & +\frac{1}{\mu}\,\Pi\left(\partial_{\theta\theta}^{-2}(h)\,\partial_{y}\left(\sum_{(i,j)\in\mathcal{J}(M,k+1)}v_{1,i}^{s}\,\partial_{\theta}v_{2,j}^{s}+\sum_{(i,j)\in\mathcal{J}(M,k)}v_{2,i}^{s}\,\partial_{y}v_{2,j}^{s}\right)\right)\\
 & +\frac{1}{\mu}\,\Pi\bigl(\partial_{\theta\theta}^{-2}(h)\ \partial_{y}(h\,\partial_{y}v_{2,k}^{s}+\partial_{\theta}h\, v_{1,k+1}^{s})\bigr)\,.
\end{align*}
 We come back to the Equation (\ref{equtjnl:1}) in order to change
the last term in this sum. This yields 
\begin{align*}
\frac{1}{\mu}\, & \Pi\bigl(\partial_{\theta\theta}^{-2}(h)\,\partial_{y}(\partial_{\theta}h\, v_{1,k+1}^{s})\bigr)=\frac{1}{\mu}\Pi\left(\partial_{\theta}h\,\partial_{\theta\theta}^{-2}(h)\,\partial_{y}(I-\Pi)v_{1,k+1}^{s}\right)\,,\\
 & =\frac{1}{\mu^{2}}\,\Pi\left(\partial_{\theta}h\,\partial_{\theta\theta}^{-2}(h)\,\partial_{y}\partial_{\theta\theta}^{-2}(I-\Pi)\left(\partial_{t}v_{1,k-1}^{s}-\mu\partial_{yy}v_{1,k-1}^{s}-\lambda\partial_{\theta y}v_{2,k-1}^{s}\right)\right)\\
 & +\frac{1}{\mu^{2}}\,\Pi\left(\partial_{\theta}h\,\partial_{\theta\theta}^{-2}(h)\,\partial_{y}\partial_{\theta\theta}^{-2}(I-\Pi)\left(h\,\partial_{y}v_{1,k}^{s}-\lambda\partial_{\theta\theta}v_{1,k}^{s}\right)\right)\\
 & +\frac{1}{\mu^{2}}\,\Pi\left(\partial_{\theta}h\,\partial_{\theta\theta}^{-2}(h)\,\partial_{y}\partial_{\theta\theta}^{-2}(I-\Pi)\left(\sum_{(i,j)\in\mathcal{J}(M,k+1)}v_{1,i}^{s}\,\partial_{\theta}v_{1,j}^{s}+\sum_{(i,j)\in\mathcal{J}(M,k)}v_{2,i}^{s}\,\partial_{y}v_{1,j}^{s}\right)\right)\,.\\
\end{align*}
 This amounts to the same thing as 
\begin{align*}
\Pi & (h\,\partial_{y}v_{2,k+1}^{s})=-\frac{1}{\mu}\,\Pi\left((\partial_{\theta}^{-2}h)^{2}\right)\partial_{yy}\Pi v_{2,k}^{s}-S_{2,k}^{l1\sslash}-P_{2}^{l1\perp}\,(I-\Pi)v_{1,k}^{s}-P_{2}^{l1\sslash}\,\Pi v_{1,k}^{s}\\
 & -Q_{2}^{l1\perp}\,(I-\Pi)v_{2,k}^{s}+\frac{1}{\mu}\,\Pi\left(\partial_{\theta\theta}^{-2}(h)\,\partial_{y}\left(\sum_{(i,j)\in\mathcal{J}(M,k+1)}v_{1,i}^{s}\,\partial_{\theta}v_{1,j}^{s}+\sum_{(i,j)\in\mathcal{J}(M,k)}v_{2,i}^{s}\,\partial_{y}v_{1,j}^{s}\right)\right)\\
 & +\frac{1}{\mu^{2}}\,\Pi\left(\partial_{\theta}h\,\partial_{\theta\theta}^{-2}(h)\,\partial_{y}\partial_{\theta\theta}^{-2}(I-\Pi)\left(\sum_{(i,j)\in\mathcal{J}(M,k+1)}v_{1,i}^{s}\,\partial_{\theta}v_{1,j}^{s}+\sum_{(i,j)\in\mathcal{J}(M,k)}v_{2,i}^{s}\,\partial_{y}v_{1,j}^{s}\right)\right).
\end{align*}
 with the notations 
\begin{align}
S_{2,k}^{l1}:= & -\frac{1}{\mu}\,\partial_{\theta\theta}^{-2}(h)\,\partial_{y}\left(\partial_{t}v_{2,k-1}^{s}-\mu\partial_{yy}v_{2,k-1}^{s}-\lambda\partial_{\theta y}v_{1,k-1}^{s}-\lambda\partial_{yy}v_{2,k-2}^{s}\right)\nonumber \\
 & -\frac{1}{\mu^{2}}\,\partial_{\theta}h\,\partial_{\theta\theta}^{-2}(h)\,\partial_{y}\partial_{\theta\theta}^{-2}(I-\Pi)\left(\partial_{t}v_{1,k-1}^{s}-\mu\partial_{yy}v_{1,k-1}^{s}-\lambda\partial_{\theta y}v_{2,k-1}^{s}\right)\,,\label{eqSdebut}\\
P_{2}^{l1}f:= & -\Pi\left(\frac{1}{\mu^{2}}\,\partial_{\theta}h\,\partial_{\theta\theta}^{-2}(h)\,\partial_{y}\partial_{\theta\theta}^{-2}(I-\Pi)\left(h\partial_{y}f-\lambda\partial_{\theta\theta}f\right)\right)\,,\label{eqSajout}\\
Q_{2}^{l1}f:= & -\frac{1}{\mu}\,\Pi\left(h\,\partial_{\theta\theta}^{-2}(h)\,\partial_{yy}f\right)\,.\label{eqQfin}
\end{align}

$\circ$ There remains to compute $\Pi(\partial_{\theta}h\, v_{1,k+2}^{s})$.
This is again (\ref{equtjnl:1}) with this time $k+2$ in place of~$k$.
\begin{align*}
\Pi\left(\partial_{\theta}h\, v_{1,k+2}^{s}\right) & =\Pi\left(\partial_{\theta}^{-1}(h)\,\partial_{\theta\theta}v_{1,k+2}^{s}\right)\,,\\
 & =\frac{1}{\mu}\,\Pi\left(\partial_{\theta}^{-1}(h)\left(\partial_{t}v_{1,k}^{s}-\mu\partial_{yy}v_{1,k}^{s}-\lambda\partial_{\theta y}(I-\Pi)v_{2,k}^{s}\right)\right)\\
 & +\frac{1}{\mu}\,\Pi\bigl(\partial_{\theta}^{-1}(h)\,(h\,\partial_{y}v_{1,k+1}^{s}-\lambda\partial_{\theta\theta}v_{1,k+1}^{s})\bigr)\\
 & +\frac{1}{\mu}\,\Pi\left(\partial_{\theta}^{-1}(h)\left(\sum_{(i,j)\in\mathcal{J}(M,k+2)}v_{1,i}^{s}\,\partial_{\theta}v_{1,j}^{s}+\sum_{(i,j)\in\mathcal{J}(M,k+1)}v_{2,i}^{s}\,\partial_{y}v_{1,j}^{s}\right)\right).
\end{align*}
 We want to remove the presence of $v_{1,k+1}^{s}$. The relation
$\Pi(h\,\partial_{\theta}^{-1}h)=0$ allows to write 
\[
\Pi\bigl(\partial_{\theta}^{-1}(h)\,(h\,\partial_{y}v_{1,k+1}^{s}-\lambda\partial_{\theta\theta}v_{1,k+1}^{s})\bigr)\,=\,\Pi\bigl(\partial_{\theta}^{-1}(h)\,(h\,\partial_{y}-\lambda\partial_{\theta\theta})\,\partial_{\theta\theta}^{-2}\,(I-\Pi)\,\partial_{\theta\theta}v_{1,k+1}^{s})\bigr)\,.
\]
 The part $\partial_{\theta\theta}v_{1,k+1}^{s}$ can be extracted
from the Equation (\ref{equtjnl:1}) with $k+1$ in place of $k$.
We find 
\begin{align*}
\Pi & \left(\partial_{\theta}h\, v_{1,k+2}^{s}\right)=\frac{1}{\mu}\,\Pi\left(\partial_{\theta}^{-1}(h)\,\left(\partial_{t}v_{1,k}^{s}-\mu\partial_{yy}v_{1,k}^{s}-\lambda\partial_{\theta y}(I-\Pi)v_{2,k}^{s}\right)\right)\\
 & +\frac{1}{\mu^{2}}\,\Pi\left(\partial_{\theta}^{-1}(h)\,(h\,\partial_{y}-\lambda\partial_{\theta\theta})\partial_{\theta\theta}^{-2}(I-\Pi)\left(\partial_{t}v_{1,k-1}^{s}-\mu\partial_{yy}v_{1,k-1}^{s}-\lambda\partial_{\theta y}v_{2,k-1}^{s}\right)\right)\\
 & +\frac{1}{\mu^{2}}\,\Pi\left(\partial_{\theta}^{-1}(h)\,(h\,\partial_{y}-\lambda\partial_{\theta\theta})\partial_{\theta\theta}^{-2}(I-\Pi)\left(h\,\partial_{y}v_{1,k}^{s}-\lambda\partial_{\theta\theta}v_{1,k}^{s}\right)\right)\\
 & +\frac{1}{\mu^{2}}\,\Pi\left(\partial_{\theta}^{-1}(h)(h\,\partial_{y}-\lambda\partial_{\theta\theta})\partial_{\theta\theta}^{-1}(I-\Pi)\left(\sum_{(i,j)\in\mathcal{J}(M,k+1)}\!\! v_{1,i}^{s}\,\partial_{\theta}v_{1,j}^{s}+\sum_{(i,j)\in\mathcal{J}(M,k)}\!\! v_{2,i}^{s}\,\partial_{y}v_{1,j}^{s}\right)\right)\\
 & +\frac{1}{\mu}\,\Pi\left(\partial_{\theta}^{-1}(h)\left(\sum_{(i,j)\in\mathcal{J}(M,k+2)}v_{1,i}^{s}\,\partial_{\theta}v_{1,j}^{s}+\sum_{(i,j)\in\mathcal{J}(M,k+1)}v_{2,i}^{s}\,\partial_{y}v_{1,j}^{s}\right)\right).
\end{align*}
 This yields to 
\begin{align*}
\Pi & \left(\partial_{\theta}h\, v_{1,k+2}^{s}\right)=-S_{2,k}^{l2\sslash}-P_{2}^{l2\perp}\,(I-\Pi)v_{1,k}^{s}-P_{2}^{l2\sslash}\,\Pi v_{1,k}^{s}-Q_{2}^{l2\perp}\,(I-\Pi)v_{2,k}^{s}\\
 & +\frac{1}{\mu^{2}}\,\Pi\left(\partial_{\theta}^{-1}(h)\,(h\,\partial_{y}-\lambda\partial_{\theta\theta})\partial_{\theta\theta}^{-2}(I-\Pi)\left(\sum_{(i,j)\in\mathcal{J}(M,k+1)}v_{1,i}^{s}\,\partial_{\theta}v_{1,j}^{s}+\sum_{(i,j)\in\mathcal{J}(M,k)}v_{2,i}^{s}\,\partial_{y}v_{1,j}^{s}\right)\right)\\
 & +\frac{1}{\mu}\,\Pi\left(\partial_{\theta}^{-1}(h)\left(\sum_{(i,j)\in\mathcal{J}(M,k+2)}v_{1,i}^{s}\,\partial_{\theta}v_{1,j}^{s}+\sum_{(i,j)\in\mathcal{J}(M,k+1)}v_{2,i}^{s}\,\partial_{y}v_{1,j}^{s}\right)\right)\,,
\end{align*}
 with 
\begin{align}
S_{2,k}^{l2}:= & \frac{1}{\mu^{2}}\,\partial_{\theta}^{-1}(h)\left(-h\,\partial_{y}+\lambda\partial_{\theta\theta})\partial_{\theta\theta}^{-2}(I-\Pi)(\partial_{t}v_{1,k-1}^{s}-\mu\partial_{yy}v_{1,k-1}^{s}-\lambda\partial_{\theta y}v_{2,k-1}^{s})\right)\,,\label{eqS2debut}\\
P_{2}^{l2}f:= & \frac{1}{\mu^{2}}\,\Pi\left(\partial_{\theta}^{-1}(h)\,(-h\,\partial_{y}+\lambda\partial_{\theta\theta})\,\partial_{\theta\theta}^{-2}\,(I-\Pi)\,(h\partial_{y}f-\lambda\partial_{\theta\theta}f)\right)\nonumber \\
 & -\frac{1}{\mu}\,\Pi\left(\partial_{\theta}^{-1}(h)(\partial_{t}f-\mu\partial_{yy}f)\right)\,,\label{eqS2ajout}\\
Q_{2}^{l2}f:= & \frac{\lambda}{\mu}\,\Pi\left(\partial_{\theta}^{-1}(h)(\partial_{\theta y}f)\right)\,.\label{c'estpos}
\end{align}

Briefly, for all $k\in\N$, we denote 
\[
S_{2,k}^{l}:=S_{2,k}^{l1}+S_{2,k}^{l2}+\lambda\,\partial_{yy}\Pi v_{2,k-1}^{s}\,,\qquad P_{2}^{l}:=P_{2}^{l1}+P_{2}^{l2}\,,\qquad Q_{2}^{l}:=Q_{2}^{l1}+Q_{2}^{l2}\,.
\]

$\circ$ Conclusion. Combining all informations together, we finally
obtain (\ref{eqlemnonhomogenizationl}) with 
\begin{align}
S_{1,k}^{nl\sslash}:= & -\frac{1}{\mu}\,\Pi\left(\partial_{\theta\theta}^{-2}(h)\ \partial_{y}\left(\sum_{(i,j)\in\mathcal{J}(M,k+1)}v_{1,i}^{s}\,\partial_{\theta}v_{1,j}^{s}+\sum_{(i,j)\in\mathcal{J}(M,k)}v_{2,i}^{s}\,\partial_{y}v_{1,j}^{s}\right)\right)\nonumber \\
 & -\Pi\left(\sum_{(i,j)\in\mathcal{J}(M,k+2)}v_{1,i}^{s}\,\partial_{\theta}v_{1,j}^{s}+\sum_{(i,j)\in\mathcal{J}(M,k+1)}v_{2,i}^{s}\,\partial_{y}v_{1,j}^{s}\right),\label{eqSPnlpa1}
\end{align}
 and 
\begin{align}
 & S_{2,k}^{nl\sslash}=-\frac{1}{\mu}\,\Pi\left(\partial_{\theta\theta}^{-2}(h)\ \partial_{y}\left(\sum_{(i,j)\in\mathcal{J}(M,k+1)}v_{1,i}^{s}\,\partial_{\theta}v_{1,j}^{s}+\sum_{(i,j)\in\mathcal{J}(M,k)}v_{2,i}^{s}\,\partial_{y}v_{1,j}^{s}\right)\right)\nonumber \\
 & -\frac{1}{\mu^{2}}\,\Pi\left(\partial_{\theta}h\ \partial_{\theta\theta}^{-2}(h)\ \partial_{y}\partial_{\theta\theta}^{-2}(I-\Pi)\left(\sum_{(i,j)\in\mathcal{J}(M,k+1)}v_{1,i}^{s}\,\partial_{\theta}v_{1,j}^{s}+\sum_{(i,j)\in\mathcal{J}(M,k)}v_{2,i}^{s}\,\partial_{y}v_{1,j}^{s}\right)\right)\nonumber \\
 & -\frac{1}{\mu^{2}}\,\Pi\left(\partial_{\theta}^{-1}(h)\ (h\,\partial_{y}-\lambda\partial_{\theta\theta})\partial_{\theta\theta}^{-2}(I-\Pi)\left(\sum_{(i,j)\in\mathcal{J}(M,k+1)}v_{1,i}^{s}\,\partial_{\theta}v_{1,j}^{s}+\sum_{(i,j)\in\mathcal{J}(M,k)}v_{2,i}^{s}\,\partial_{y}v_{1,j}^{s}\right)\right)\nonumber \\
 & -\frac{1}{\mu}\,\Pi\left(\partial_{\theta}^{-1}(h)\,\left(\sum_{(i,j)\in\mathcal{J}(M,k+2)}v_{1,i}^{s}\,\partial_{\theta}v_{1,j}^{s}+\sum_{(i,j)\in\mathcal{J}(M,k+1)}v_{2,i}^{s}\,\partial_{y}v_{1,j}^{s}\right)\right)\nonumber \\
 & -\Pi\left(\sum_{(i,j)\in\mathcal{J}(M,k+2)}v_{1,i}^{s}\,\partial_{\theta}v_{2,j}^{s}+\sum_{(i,j)\in\mathcal{J}(M,k+1)}v_{2,i}^{s}\,\partial_{y}v_{2,j}^{s}\right).\label{eqSPnlpa2}
\end{align}
 \end{proof}

The definition of $S_{k}^{nl\sslash}$ involves the sets $\mathcal{J}(M,k)$,
$\mathcal{J}(M,k+1)$ and $\mathcal{J}(M,k+2)$. Note that 
\begin{equation}
\bigl\lbrack\, M\geq3\,,\ \ \tilde{k}\in\{k,k+1,k+2\}\,,\ \ (i,j)\in\mathcal{J}(M,\tilde{k})\,\bigr\rbrack\ \ \Longrightarrow\ \ i\leq k-1\text{ and }j\leq k-1\,.\label{implicationres}
\end{equation}
 Thus, the presence of $v_{k}^{s}$ inside $S_{k}^{nl\sslash}$ is
not allowed as long as $M\geq3$. On the contrary, when $M=2$, it
becomes effective. This remark can be formalized through the following
statement.

\begin{lemma} {[}Refined description of the source term of (\ref{eqlemnonhomogenizationl}){]}
\label{lemslowsource} $\,$

\textit{i)} The expressions $S_{k}^{l\sslash}$ only depend on the
$v_{j}^{s}$ with $j\in\llbracket0,k-1\rrbracket$. More precisely,
we have 
\[
S_{k}^{l\sslash}=f_{k}^{l\sslash}(v_{0}^{s},\ldots,v_{k-1}^{s})\,
\]
 where $f_{k}^{l\sslash}:={}^{t}(f_{1,k}^{l\sslash},f_{2,k}^{l\sslash})$
are homogeneous linear functions of their arguments.

\textit{{ii)}} When $M\geq3$, the expressions $S_{k}^{nl\sslash}$
only depend on the $v_{j}^{s}$ with $j\in\llbracket0,k-1\rrbracket$.
More precisely, we can retain that 
\[
S_{k}^{nl\sslash}=f_{k}^{qnl\sslash}(v_{0}^{s},\ldots,v_{k-1}^{s})
\]
 where $f_{k}^{qnl\sslash}:={}^{t}(f_{1,k}^{qnl\sslash},f_{2,k}^{qnl\sslash})$
are quadratic functions of their arguments.

\medskip{}

\textit{{iii)}} When $M=2$, the expressions $S_{k}^{nl\sslash}$
only depend on the $v_{j}^{s}$ with $j\in\llbracket0,k\rrbracket$.
The influence of $v_{k}^{s}$ can be specified through a decomposition
of the form 
\[
S_{k}^{nl\sslash}=f_{k}^{qnl\sslash}(v_{0}^{s},\ldots,v_{k-1}^{s})+SP(v_{0}^{s})\, v_{k}^{s}\,,\qquad f_{k}^{qnl\sslash}:=(f_{1,k}^{qnl\sslash},f_{2,k}^{qnl\sslash})\,.
\]
 Again $f_{k}^{qnl\sslash}$ is a quadratic function of its arguments
whereas $SP(v_{0}^{s}):={}^{t}\bigl(SP_{1}(v_{0}^{s}),SP_{2}(v_{0}^{s})\bigr)$
is the linear differential operator defined according to 
\begin{equation}
\left\{ \begin{array}{l}
SP_{1}(v_{0}^{s})\, v_{k}^{s}:=-\,\Pi\,(v_{1,k}^{s}\,\partial_{\theta}v_{1,0}^{s}+v_{1,0}^{s}\,\partial_{\theta}v_{1,k}^{s})\,,\\
{\displaystyle SP_{2}(v_{0}^{s})\, v_{k}^{s}:=-\,\Pi(v_{1,0}^{s}\,\partial_{\theta}v_{2,k}^{s}+v_{1,k}^{s}\,\partial_{\theta}v_{2,0}^{s})-\frac{1}{\mu}\,\Pi\,\bigl(\partial_{\theta}^{-1}(h)\,(v_{1,0}^{s}\,\partial_{\theta}v_{1,k}^{s}+v_{1,k}^{s}\,\partial_{\theta}v_{1,0}^{s})\bigr)\,.}
\end{array}\right.\label{eqSPK}
\end{equation}
 \end{lemma}

From now on, we simply note $f_{k}^{nl\sslash}:=f_{k}^{l\sslash}+f_{k}^{qnl\sslash}$.

\begin{proof}[Proof of Lemma \ref{lemslowsource}] As already
mentioned, the statement {\em i)} is a direct consequence of (\ref{implicationres}).
The \textit{linear} aspect of $f_{k}^{l\sslash}$ is a consequence
of the formulas obtained in \eqref{eqSdebut} and \eqref{eqS2debut}.
On the other hand, the \textit{quadratic} aspect of $f_{k}^{qnl\sslash}$
is obvious in view of (\ref{eqSPnlpa1}) and (\ref{eqSPnlpa2}). It
proves {\em{ii)}}.

\medskip{}

There remains to consider the situation {\em iii)} where $M=2$.
Note that $\mathcal{J}(2,k+2)=\mathcal{J}(0,k)$. In view of (\ref{implicationres}),
we have to concentrate on the contribution of 
\begin{align*}
\begin{pmatrix}-\,\Pi\left({\displaystyle \sum_{(i,j)\in\mathcal{J}(0,k)}v_{1,i}^{s}\,\partial_{\theta}v_{1,j}^{s}}\right)\\
{\displaystyle -\,\frac{1}{\mu}\,\Pi\left(\partial_{\theta}^{-1}(h)\left(\sum_{(i,j)\in\mathcal{J}(0,k)}v_{1,i}^{s}\,\partial_{\theta}v_{1,j}^{s}\right)\right)-\Pi\left(\sum_{(i,j)\in\mathcal{J}(0,k)}v_{1,i}^{s}\,\partial_{\theta}v_{2,j}^{s}\right)}
\end{pmatrix}.
\end{align*}
 In the sums above, only the extremal indices $(i,j)=(k,0)$ and $(i,j)=(0,k)$
give a contribution to include in $SP(v_{0}^{s})$, leading to (\ref{eqSPK}).
\end{proof}

\subsubsection{Analysis of the system (\ref{eqlemnonhomogenizationl})}

\paragraph{Projection of system \eqref{eqlemnonhomogenizationl}.}

The system (\ref{eqlemnonhomogenizationl}) is clearly an evolution
equation of parabolic type. As such, it can be completed by initial
data. But to solve it, we also need to identify the extra terms $(I-\Pi)v_{\star,k}^{s}$
with $\star\in\{1,2\}$. To this end, it suffices again to exploit
(\ref{equtjnl}). Recall that $\mathcal{V}\in\left\{ W^{m,p},H^{s},\mathcal{W}_{T}^{m,s},\mathcal{H}_{T}^{s},\mathcal{E}_{\delta}^{s}\right\} $
and define the linear continuous isomorphism 
\[
\begin{array}{rcl}
\Phi:\mathcal{V}(\T\times\R)\times\mathcal{V}(\T\times\R) & \longrightarrow & \mathcal{V}^{tot}:=\mathcal{V}^{\perp}(\T\times\R)\times\mathcal{V}^{\sslash}(\R)\times\mathcal{V}^{\perp}(\T\times\R)\times\mathcal{V}^{\sslash}(\R)\\
(f_{1},f_{2}) & \longmapsto & \bigl((I-\Pi)f_{1},\Pi f_{1},(I-\Pi)f_{2},\Pi f_{2}\bigr)\,.
\end{array}
\]
 By construction, the expression $V_{k}^{s}:={}^{t}\Phi\, v_{k}^{s}$
must be solution of the system 
\begin{equation}
\mathcal{A}\, V_{k}^{s}:=\begin{pmatrix}\mu\partial_{\theta\theta} & 0 & 0 & 0\\
-P_{1}^{l\perp} & P_{y} & 0 & 0\\
-T_{s} & -T_{s} & \mu\partial_{\theta\theta} & 0\\
-P_{2}^{l\perp} & -P_{2}^{l\sslash} & -Q_{2}^{l\perp} & P_{y}
\end{pmatrix}V_{k}^{s}=\begin{pmatrix}f_{1,k}^{nl\perp}\\
f_{1,k}^{nl\sslash}\\
f_{2,k}^{nl\perp}\\
f_{2,k}^{nl\sslash}
\end{pmatrix}=:f_{k}^{nl},\qquad V_{k}^{s}=\begin{pmatrix}v_{1,k}^{s\perp}\\
v_{1,k}^{s\sslash}\\
v_{2,k}^{s\perp}\\
v_{2,k}^{s\sslash}
\end{pmatrix}\label{ilfallaitlefaire}
\end{equation}
 where the term $f_{k}^{nl\sslash}$ is given by Lemma \ref{lemslowsource}
whereas $T_{s}$ and $P_{y}$ are the operators defined by 
\[
T_{s}f:=(I-\Pi)\left(\partial_{\theta}h\, f\right)\,,\qquad P_{y}f:=\partial_{t}f-\Bigl(\mu+\frac{1}{\mu}\Pi\left((\partial_{\theta}^{-1}h)^{2}\right)\Bigr)\,\partial_{yy}f.
\]

\paragraph{Description of $f_{k}^{nl\perp}$.}

As for $f_{k}^{nl\sslash}$ we decompose $f_{k}^{nl\perp}$ into $f_{k}^{nl\perp}:=f_{k}^{l\perp}+f_{k}^{qnl\perp}$
where the linear part $f_{k}^{l\perp}$ is defined by 
\begin{align}
f_{1,k}^{l\perp}:= & (I-\Pi)\left(\partial_{t}v_{1,k-2}^{s}-\mu\partial_{yy}v_{1,k-2}^{s}-\lambda\partial_{\theta y}v_{2,k-2}^{s}+h\,\partial_{y}v_{1,k-1}^{s}-\lambda\partial_{\theta\theta}v_{1,k-1}^{s}\right)\,,\label{sourceperplineaire}\\
f_{2,k}^{l\perp}:= & (I-\Pi)(\partial_{t}v_{2,k-2}^{s}-\mu\partial_{yy}v_{2,k-2}^{s}-\lambda\partial_{\theta y}v_{1,k-2}^{s}+h\,\partial_{y}v_{2,k-1}^{s}-\lambda\partial_{yy}v_{2,k-3}^{s})\,,\label{sourceperplineairebis}
\end{align}
 whereas the quadratic part $f_{k}^{qnl\perp}={}^{t}(f_{1,k}^{qnl\perp},f_{2,k}^{qnl\perp})$
is given by 
\begin{equation}
f_{p,k}^{qnl\perp}:=(I-\Pi)\left(\sum_{(i,j)\in\mathcal{J}(M,k)}v_{1,i}^{s}\,\partial_{\theta}v_{p,j}^{s}+\sum_{(i,j)\in\mathcal{J}(M+1,k)}v_{2,i}^{s}\,\partial_{y}v_{p,j}^{s}\right),\qquad p\in\{1,2\}\,.\label{fffffqqqq}
\end{equation}

\begin{lemma} {[}Description of $f_{k}^{nl\perp}${]} \label{lemslowsourcebiss}
$\,$ The functions $f_{k}^{l\perp}$ and $f_{k}^{qnl\perp}$ only
depend on the $v_{j}^{s}$ with $j\in\llbracket0,k-1\rrbracket$.
They are respectively homogeneous linear and homogeneous quadratic
functions of their arguments $(v_{0}^{s},\ldots,v_{k-1}^{s})$. \end{lemma}
\begin{proof}[Proof of Lemma \ref{lemslowsourcebiss}] Just look
at (\ref{sourceperplineaire})-(\ref{sourceperplineairebis})-(\ref{fffffqqqq})
together with (\ref{implicationres}). \end{proof}

\paragraph{ Analysis of the system (\ref{ilfallaitlefaire}).}

To study the system (\ref{ilfallaitlefaire}), we have to distinguish
the general case $M\geq3$ from the critical case $M=2$.

\medskip{}

$\circ$ When $M\geq3$, According to the Lemmas \ref{lemhomogeneisationnonlinear}
and \ref{lemslowsourcebiss}, the above right hand term $f_{k}^{nl}$
can be viewed as a source term.

\bigskip{}

$\circ$ When $M=2$, $f_{k}^{nl\sslash}$ is no longer a source term.
Now, recall that we can exploit the condition~\eqref{fondacriterion}.
This information is essential. It induces many simplifications when
computing $SP(v_{0}^{s})$. We find that $SP_{1}(v_{0}^{s})\equiv0$
whereas $SP_{2}(v_{0}^{s})$ can be reduced to the following linear
(non differential) operator 
\begin{equation}
SP_{2}(v_{0}^{s})\, v_{k}^{s}=-\,\Pi(v_{1,k}^{s\perp}\,\partial_{\theta}v_{2,0}^{s})+\frac{1}{\mu}\ v_{1,0}^{s\sslash}\ \Pi\,(h\, v_{1,k}^{s\perp})\,.\label{eqSPK2bis}
\end{equation}
 At first sight, the expression $SP_{2}(v_{0}^{s})\, v_{0}^{s}$ depends
in a non linear way on $v_{0}^{s}$. However, we can again exploit
the condition (\ref{fondacriterion}) (which says that $v_{1,0}^{s\perp}\equiv0$)
and then apply (\ref{eqSPK2bis}) with $k=0$ in order to obtain further
cancellations. There remains 
\begin{equation}
SP(v_{0}^{s})\, v_{0}^{s}\,\equiv\,0\,.\label{eliminatevs0}
\end{equation}

From now on, since there is no more ambiguity, we can omit to signal
that $SP(v_{0}^{s})$ depends on $v_{0}^{s}$, and in fact only on
$v_{1,0}^{s\sslash}$. We will most often note $SP(v_{0}^{s})\equiv SP(v_{0}^{s\sslash})\equiv SP={}^{t}(SP_{1},SP_{2})$.
Since $SP_{2}\not\equiv0$, the formulation (\ref{ilfallaitlefaire})
must be changed. This time, we have to deal with 
\begin{equation}
\widetilde{\mathcal{A}}\, V_{k}^{s}={}^{t}(f_{1,k}^{nl\perp},f_{1,k}^{nl\sslash},f_{2,k}^{nl\perp},f_{2,k}^{nl\sslash})\,,\qquad V_{k}^{s}:={}^{t}\Phi\, v_{k}^{s}\label{ilfallaitlefairenlc}
\end{equation}
 where 
\begin{equation}
\widetilde{\mathcal{A}}:=\begin{pmatrix}\mu\partial_{\theta\theta} & 0 & 0 & 0\\
-P_{1}^{l\perp} & P_{y} & 0 & 0\\
-T_{s} & -T_{s} & \mu\partial_{\theta\theta} & 0\\
-P_{2}^{l\perp}-SP_{2} & -P_{2}^{l\perp} & -Q_{2}^{l\perp} & P_{y}
\end{pmatrix}.\label{eqtildeAsansk}
\end{equation}

The similarities between $\mathcal{A}$ and $\widetilde{\mathcal{A}}$
are obvious. These two matrix valued operators have both a triangular
structure. The difference, when passing from $\mathcal{A}$ to $\widetilde{\mathcal{A}}$,
concerns only the perturbation in the bottom-left position $(4,1)$.
This particularity plays a crucial part in the discussion below. To
simplify, we present below the result in a smooth setting.

\begin{lemma}{[}Solving the system \eqref{ilfallaitlefaire} or \eqref{ilfallaitlefairenlc}{]}
\label{leminvAk} We assume that the condition \eqref{fondacriterion}
is verified. Select a function $V_{0}^{\sslash}={}^{t}(V_{0}^{1\sslash},V_{0}^{2\sslash})\in H^{\infty}(\R)^{2}$
and a source term $F={}^{t}(F^{1\perp},F^{1\sslash},F^{2\perp},F^{2\sslash})\in(\mathcal{H}_{T}^{\infty})^{tot}$.
Then, for all $T\in\R_{+}^{*}$, the problem 
\begin{equation}
\bigl\{\,\mathcal{P}\, V=F\,,\qquad V={}^{t}(V^{1\perp},V^{1\sslash},V^{2\perp},V^{2\sslash})\,,\qquad{}^{t}(V^{1\sslash},V^{2\sslash})_{|_{t=0}}={}^{t}(V_{0}^{1\sslash},V_{0}^{2\sslash}),\label{moddsupa}
\end{equation}
 where $\mathcal{P}\in\{\mathcal{A},\widetilde{\mathcal{A}}\}$, has
a unique solution $V$ in $(\mathcal{H}_{T}^{\infty})^{tot}$. \end{lemma}
\begin{proof}[Proof of Lemma \ref{leminvAk}] To solve (\ref{moddsupa}),
the strategy is to argue line after line. We start to solve the problem
for the operator $\mathcal{A}$.

\medskip{}

- \textit{First line.} Since the operator $\mu\partial_{\theta\theta}:\left(\mathcal{H}_{T}^{\infty}\right)^{\perp}\longrightarrow\left(\mathcal{H}_{T}^{\infty}\right)^{\perp}$
is invertible, we can define without ambiguity (and with no choice)
\begin{equation}
V^{1\perp}:=\mu^{-1}\partial_{\theta\theta}^{-2}F^{1\perp}\in\left(\mathcal{H}_{T}^{\infty}\right)^{\perp}\,.\label{VGs1}
\end{equation}

- \textit{Second line.} Observe that $P_{1}^{l\perp}:\mathcal{H}_{T}^{\infty}\longrightarrow\left(\mathcal{H}_{T}^{\infty}\right)^{\perp}$.
The next component $V^{1\sslash}$ must be a solution of the heat
equation (in $t$ and $y$) 
\begin{equation}
\bigl\lbrace\, P_{y}V^{1\sslash}=F^{1\sslash}+P_{1}^{l\perp}V^{1\perp}\in(\mathcal{H}_{T}^{\infty})^{\sslash}\,,\qquad(V^{1\sslash})_{|_{t=0}}=V_{0}^{1\sslash}\in(\mathcal{H}_{T}^{\infty})^{\sslash}\,.\label{transVGs}
\end{equation}
 Obviously, there is a unique solution on $[0,T]$ of this initial
value problem. It does not depend on $\theta$. In other words, it
is such that $V(t,\cdot)\in(\mathcal{H}_{T}^{\infty})^{\sslash}$
for all $t\in[0,T]$.

\medskip{}

- \textit{Third line.} Since $T_{s}:\mathcal{H}_{T}^{\infty}\longrightarrow\left(\mathcal{H}_{T}^{\infty}\right)^{\perp}$
the component $V^{2\perp}$ can be obtained through the formula 
\begin{equation}
V^{2\perp}=\mu^{-1}\partial_{\theta\theta}^{-2}\bigl(F^{2\perp}+T_{s}V^{1\perp}+T_{s}V^{1\sslash})\bigr)\in\left(\mathcal{H}_{T}^{\infty}\right)^{\perp}\,.\label{VGs2}
\end{equation}

\medskip{}

- \textit{Fourth line.} We can use the same argument as in the second
line. It suffices to check that by definition $P_{2}^{l}:\mathcal{H}_{T}^{\infty}\longrightarrow\left(\mathcal{H}_{T}^{\infty}\right)^{\perp}$
and $Q_{2}^{l}:\mathcal{H}_{T}^{\infty}\longrightarrow\left(\mathcal{H}_{T}^{\infty}\right)^{\perp}$.
Then, there remains to solve 
\[
\bigl\lbrace\, P_{y}V^{2\sslash}=F^{2\sslash}+P_{2}^{l\perp}V^{1\perp}+P_{2}^{l\sslash}V^{1\sslash}+Q_{2}^{l\perp}V^{2\perp}\in(\mathcal{H}_{T}^{\infty})^{\sslash}\,,\qquad(V^{2\sslash})_{|_{t=0}}=V_{0}^{2\sslash}\in(\mathcal{H}_{T}^{\infty})^{\sslash}\,.
\]
 Note that the triangular structure of $\mathcal{A}$ is crucial in
this procedure. \\

To solve the system \eqref{ilfallaitlefairenlc} that is replacing
the operator $\mathcal{A}$ by $\widetilde{\mathcal{A}}$, the only
change in the above proof is at the level of the fourth line where
the supplementary source term $SP_{2}V^{1\perp}$ must be incorporated.
\end{proof}

\subsubsection{Analysis of the system (\ref{equtaujnl})}

\paragraph{Projection of the system (\ref{equtaujnl}).}

In this paragraph, we consider the parabolic system (\ref{equtaujnl})
which can be associated with some smooth initial data $v_{k}^{f}(0,\cdot)\in\mathcal{H}^{\infty}(\T\times\R)$.
Classical statements (see for instance \cite{MR892024}) say that
the corresponding Cauchy problem has a unique global solution~$v_{k}^{f}$
such that $v_{k}^{f}(t,\cdot)\in\mathcal{H}(\T\times\R)$ for all
$\tau\in\R_{+}$. The difficulty is the following. The variable $\tau$
is aimed to be replaced by $\varepsilon^{-2}\, t$ with $t$ fixed
and $\varepsilon\rightarrow0$ and, since the original equation (\ref{eqrefthetay})
contains nonlinearities, we cannot allow any (uncontrolled) growth
with respect to $\tau$. To get round this problem, we will instead
require a rapid decay when $\tau\rightarrow+\infty$ but this necessitates
$v_{k}^{f}(0,\cdot)$ to be selected conveniently.

\medskip{}

To see how to adjust $v_{k}^{f}(0,\cdot)$, we can interpret (\ref{equtaujnl})
in the form 
\begin{equation}
\mathcal{B}\, V_{k}^{f}:=\begin{pmatrix}P_{\theta} & 0 & 0 & 0\\
0 & \partial_{\tau} & 0 & 0\\
T_{s} & T_{s} & P_{\theta} & 0\\
T_{f} & 0 & 0 & \partial_{\tau}
\end{pmatrix}\begin{pmatrix}g_{1,k}^{nl\perp}\\
g_{1,k}^{nl\sslash}\\
g_{2,k}^{nl\perp}\\
g_{2,k}^{nl\sslash}
\end{pmatrix},\qquad V_{k}^{f}:=\,^{t}\Phi(v_{k}^{f})=\begin{pmatrix}v_{1,k}^{f\perp}\\
v_{1,k}^{f\sslash}\\
v_{2,k}^{f\perp}\\
v_{2,k}^{f\sslash}
\end{pmatrix}\label{ststBrajout}
\end{equation}
 where $T_{f}$ and $P_{\theta}$ are the operators defined by 
\[
T_{f}f:=\Pi(\partial_{\theta}h\, f)\,,\qquad P_{\theta}f:=\partial_{\tau}f-\mu\,\partial_{\theta\theta}f\,.
\]

\paragraph{Description of $g_{k}^{nl\perp}$.}

We decompose $g_{k}^{nl\perp}$ into is linear and quadratic part
$g_{k}^{nl}:=g_{k}^{l}+g_{k}^{qnl}$. The linear part $g_{k}^{l}$
is defined by 
\begin{equation}
\left\{ \begin{array}{l}
g_{1,k}^{l}:=(-h\,\partial_{y}v_{1,k-1}^{f}+\lambda\partial_{\theta\theta}v_{1,k-1}^{f})+(\mu\partial_{yy}v_{1,k-1}^{f}+\lambda\partial_{\theta y}v_{2,k-2}^{f})\,,\\
g_{2,k}^{l}:=-h\,\partial_{y}v_{2,k-1}^{f}+\left(\lambda\partial_{\theta y}v_{1,k-2}^{f}+\mu\partial_{yy}v_{2,k-2}^{f}\right)+\lambda\partial_{yy}v_{2,k-3}^{f}\,,
\end{array}\right.\label{sourcefastlinear}
\end{equation}
 whereas, the quadratic part of $g_{k}^{qnl}={}^{t}(g_{1,k}^{qnl},g_{2,k}^{qnl})$
is given by 
\begin{align*}
g_{1,k}^{qnl}:= & \sum_{(i,j)\in\mathcal{J}(M,k)}v_{1,i}^{f}\,\partial_{\theta}v_{1,j}^{f}+\sum_{(i,j)\in\mathcal{J}(M,k-1)}v_{2,i}^{f}\,\partial_{y}v_{1,j}^{f}\\
 & +\sum_{(i,j,l)\in\mathcal{I}(M,k)}\left(v_{1,i}^{f}\left(\partial_{t}\right)^{l}(\partial_{\theta}v_{1,j}^{s})(0)+\left(\partial_{t}\right)^{l}(v_{1,i}^{s})(0)\,\partial_{\theta}v_{1,j}^{f}\right)\frac{\tau^{l}}{l!}\\
 & +\sum_{(i,j,l)\in\mathcal{I}(M,k-1)}\left(v_{2,i}^{f}\left(\partial_{t}\right)^{l}(\partial_{y}v_{1,j}^{s})(0)+\left(\partial_{t}\right)^{l}(v_{2,i}^{s})(0)\,\partial_{y}v_{1,j}^{f}\right)\frac{\tau^{l}}{l!}\,,
\end{align*}
 
\begin{align*}
g_{2,k}^{qnl}:= & \sum_{(i,j)\in\mathcal{J}(M,k)}v_{1,i}^{f}\,\partial_{\theta}v_{2,j}^{f}+\sum_{(i,j)\in\mathcal{J}(M,k-1)}v_{2,i}^{f}\,\partial_{\theta}v_{2,j}^{f}\\
 & +\sum_{(i,j,l)\in\mathcal{I}(M,k)}\left(v_{1,i}^{f}\left(\partial_{t}\right)^{l}(\partial_{\theta}v_{2,j}^{s})(0)+\left(\partial_{t}\right)^{l}(v_{1,i}^{s})(0)\,\partial_{\theta}v_{2,j}^{f}\right)\frac{\tau^{l}}{l!}\\
 & +\sum_{(i,j,l)\in\mathcal{I}(M,k-1)}\left(v_{2,i}^{f}\left(\partial_{t}\right)^{l}(\partial_{y}v_{2,j}^{s})(0)+\left(\partial_{t}\right)^{l}(v_{2,i}^{s})(0)\,\partial_{y}v_{2,j}^{f}\right)\frac{\tau^{l}}{l!}\,.
\end{align*}
 \begin{lemma}{[}Description of $g_{k}^{nl}${]} \label{lemfastsource}
Assume that $M\geq2$. The functions $g_{k}^{l}$ and $g_{k}^{qnl}$
depend only on the $v_{j}^{f}$ and the $\partial_{t}^{l}v_{j}^{s}(0,\cdot)$
where $j\in\llbracket0,k-1\rrbracket$ and $l\in\bigl\llbracket0,\left\lfloor \frac{k}{2}\right\rfloor \bigr\rrbracket$.
They are respectively homogeneous linear and homogeneous quadratic
functions of their arguments. \end{lemma}

\begin{proof}[Proof of Lemma \ref{lemfastsource}] It suffices
to examine the various terms appearing in the sums involved by the
definition of $g_{k}^{qnl}$. 
\begin{description}
\item [{$\quad\ $-}] The sums based on the symbol $\mathcal{J}$ can be
dealt by observing that 
\begin{equation}
\bigl\lbrack\, M\geq2\,,\ \ \tilde{k}\in\{k-1,k\}\,,\ \ (i,j)\in\mathcal{J}(M,\tilde{k})\,\bigr\rbrack\ \ \Longrightarrow\ \ i\leq k-1\text{ and }j\leq k-1\,.\label{implicationresidu}
\end{equation}
 
\item [{$\quad\ $-}] The sums involving the symbol $\mathcal{I}$ are
of the form $\mathcal{I}(M,k)$ or $\mathcal{I}(M,k-1)$. Coming back
to the definition (\ref{mathcalI(M,k)}), we can easily infer that
\begin{equation}
\bigl\lbrack\, M\geq2\,,\ \ \tilde{k}\in\{k-1,k\}\,,\ \ (i,j,l)\in\mathcal{I}(M,\tilde{k})\,\bigr\rbrack\ \ \Longrightarrow\ \ i\leq k-1\text{ and }j\leq k-1\label{implicationIresidu}
\end{equation}
 as well as $l\leq\left\lfloor \frac{k}{2}\right\rfloor $. 
\end{description}
\end{proof}

Due to Lemma \ref{lemfastsource}, the expression $g_{k}^{nl}$ can
be viewed as a source term in system \eqref{ststBrajout}. In order
to guarantee the fast decaying criterion in $\tau$, we can proceed
as described below.

\begin{lemma}{[}Solving \ref{ststBrajout} in the case of a fast
decay when $\tau$ tends to $+\infty${]}\label{leminvBk} Select
functions 
\[
V_{0}^{\perp}={}^{t}(V_{0}^{1\perp},V_{0}^{2\perp})\in H^{\infty}(\T\times\R)^{2}\,,\qquad G={}^{t}(G^{1\perp},G^{1\sslash},G^{2\perp},G^{2\sslash})\in\left(\mathcal{E}_{\delta}^{\infty}\right)^{tot}\,,\qquad\delta\in\,]0,\mu[\,.
\]
 There is a unique expression $V_{0}^{\sslash}={}^{t}(V_{0}^{1\sslash},V_{0}^{2\sslash})\in(H^{\infty})^{\sslash}(\T\times\R;\R)^{2}$
which can be determined in function of $G$ through formulas (\ref{VGf1})
and (\ref{VGf2}) such that the Cauchy problem 
\begin{equation}
\bigl\{\,\mathcal{B}\, V=G\,,\qquad V={}^{t}(V^{1\perp},V^{1\sslash},V^{2\perp},V^{2\sslash})\,,\qquad V_{|_{\tau=0}}={}^{t}(V_{0}^{1\perp},V_{0}^{1\sslash},V_{0}^{2\perp},V_{0}^{2\sslash})\label{moddsupb}
\end{equation}
 has a global solution $V$ belonging to the space $\left(\mathcal{E}_{\delta}^{\infty}\right)^{tot}$.
\end{lemma} \begin{proof}[Proof of Lemma \ref{leminvBk}] The
strategy is again to argue line after line.

\medskip{}

- \textit{First line.} Just apply the end of Lemma \ref{lemestsol}.

\medskip{}

- \textit{Second line.} It suffices to take 
\begin{equation}
V_{0}^{1\sslash}(.)\,:=\,-\int_{0}^{+\infty}\, G^{1\sslash}(s,.)\ ds\in(H^{\infty})^{\sslash}\label{VGf1}
\end{equation}
 in order to recover after integration that $V^{1\sslash}\in(\mathcal{E}_{\delta}^{\infty})^{\sslash}$
with 
\[
V^{1\sslash}(\tau,.)=V_{0}^{1\sslash}(.)+\int_{0}^{\tau}\, G^{1\sslash}(s,.)\ ds=-\int_{\tau}^{+\infty}\, G^{1\sslash}(s,.)\ ds\,.
\]

\smallskip{}

- \textit{Third line.} For all $m\in\N$, the operator $T_{s}:H^{m}\rightarrow(H^{m})^{\perp}$
is continuous. It follows that $T_{s}$ sends the functional space
$\mathcal{E}_{\delta}^{\infty}$ into $(\mathcal{E}_{\delta}^{\infty})^{\perp}$.
Concerning $V^{2\perp}$, the argument is again Lemma~\ref{lemestsol}
applied this time with the source term $G^{2\perp}-T_{s}V^{1\perp}-T_{s}V^{1\sslash}\in(\mathcal{E}_{\delta}^{\infty})^{\perp}$.

\medskip{}

- \textit{Fourth line.} For all $m\in\N$, the operator $T_{f}:H^{m}\rightarrow(H^{m})^{\sslash}$
is continuous. Therefore, we know that $T_{f}:\mathcal{E}_{\delta}^{\infty}\rightarrow(\mathcal{E}_{\delta}^{\infty})^{\sslash}$.
With this in mind, it suffices to select 
\begin{equation}
V_{0}^{2\sslash}(.)\,:=\,-\int_{0}^{+\infty}\,(G^{2\sslash}-T_{f}G^{1\perp})(s,.)\ ds\in(H^{\infty})^{\sslash}\,.\label{VGf2}
\end{equation}
 \end{proof}

\subsubsection{Proof of Proposition \ref{propprofilenonlinear}}

The matter is to show by induction on $K\in\llbracket0,N+1\rrbracket$
that the property given at the level of line (\ref{recurezrnl}) is
verified.

\medskip{}

\paragraph{Verification of $\mathcal{HN}(0)$.}

By convention, we start with $v_{k}^{s}\equiv0$ and $v_{k}^{f}\equiv0$
for $k\in\{-3,-2,-1\}$. Applying Lemmas \ref{lemslowsource}, \ref{lemslowsourcebiss}
and \ref{lemfastsource} with $k=0$ and exploiting the given (linear
or quadratic) homogeneity properties, we find that $f_{0}^{nl}\equiv0$
and $g_{0}^{nl}\equiv0$.

\medskip{}

Recall that $V_{0}^{s}:={}^{t}\Phi\, v_{0}^{s}$ and $V_{0}^{f}:={}^{t}\Phi\, v_{0}^{f}$.
The matter here is to show the existence of functions $V_{0}^{s}\in(\mathcal{H}_{T}^{\infty})^{tot}$
and $V_{0}^{f}\in(\mathcal{E}_{\delta}^{\infty})^{tot}$ such that:

\medskip{}

- For $M\geq3$: 
\begin{equation}
\mathcal{A}\, V_{0}^{s}=0\,,\qquad\mathcal{B}\, V_{0}^{f}=0\,,\qquad(V_{0}^{s}+V_{0}^{f})(0,\cdot)={}^{t}\Phi\, v_{0}^{0}(\cdot)\,.\label{eqABNL03}
\end{equation}
 - For $M=2$: 
\begin{equation}
\widetilde{\mathcal{A}}\, V_{0}^{s}=0\,,\qquad\mathcal{B}\, V_{0}^{f}=0\,,\qquad(V_{0}^{s}+V_{0}^{f})(0,\cdot)={}^{t}\Phi\, v_{0}^{0}(\cdot)\,.\label{eqABNL02}
\end{equation}
 By construction, the two first lines of (\ref{eqABNL03}) and (\ref{eqABNL02})
amount to the same thing as $\mu\,\partial_{\theta\theta}v_{1,0}^{s\perp}\equiv0$.
We recover here (\ref{fondacriterion}). From (\ref{eliminatevs0}),
we can deduce that $\widetilde{\mathcal{A}}\, V_{0}^{s}\equiv{\mathcal{A}}\, V_{0}^{s}$.
Therefore, the discussion concerning (\ref{eqABNL02}) is the same
as the one related to (\ref{eqABNL03}).\\

The above initial condition can be decomposed into 
\begin{align}
^{t}(V_{0}^{s1\perp},V_{0}^{s2\perp})(0,\cdot)+{}^{t}(V_{0}^{f1\perp},V_{0}^{f2\perp})(0,\cdot) & =(I-\Pi)v_{0}^{0}(\cdot)\,,\label{iniy-pi}\\
^{t}(V_{0}^{s1\sslash},V_{0}^{s2\sslash})(0,\cdot)+{}^{t}(V_{0}^{f1\sslash},V_{0}^{f2\sslash})(0,\cdot) & =\Pi v_{0}^{0}(\cdot)\,.\label{iniypi}
\end{align}
 In view of (\ref{VGf1}) and (\ref{VGf2}), we must have $V_{0}^{f1\sslash}(0,\cdot)\equiv V_{0}^{f2\sslash}(0,\cdot)\equiv0$
whatever $V_{0}^{f\perp}(0,\cdot)$ is. It follows that we can identify
$V_{0}^{s1\sslash}(0,\cdot)$ and $V_{0}^{s2\sslash}(0,\cdot)$ through
(\ref{iniypi}). Now, knowing what is $V_{0}^{s1\sslash}(0,\cdot)$,
formulas (\ref{VGs1}) and (\ref{VGs2}) give access to $V_{0}^{s1\perp}(0,\cdot)$
and $V_{0}^{s2\perp}(0,\cdot)$. There remains to use the condition
(\ref{iniy-pi}) in order to further extract $V_{0}^{f1\perp}(0,\cdot)$
and $V_{0}^{f2\perp}(0,\cdot)$.

\medskip{}

We apply Lemma \ref{leminvAk} and Lemma \ref{leminvBk} in the case
of the initial data $V_{0}^{s\sslash}(0,\cdot)$ and $V_{0}^{f\perp}(0,\cdot)$
which have just been computed. Note that, due to the preceding construction,
there is no contradiction between the expressions $V_{0}^{s\perp}(0,\cdot)$
and $V_{0}^{f\sslash}(0,\cdot)$ thus obtained and the compatibility
conditions required at the level of (\ref{iniy-pi}) and (\ref{iniypi}).
By this way, we can recover functions $V_{0}^{s}\in(\mathcal{H}_{T}^{\infty})^{tot}$
and $V_{0}^{f}\in(\mathcal{E}_{\delta}^{\infty})^{tot}$. Then, to
conclude, it suffices to come back to $v_{0}^{s}\in\mathcal{H}_{T}^{\infty}$
and $v_{0}^{f}\in\mathcal{E}_{\delta}^{\infty}$ through the action
of $\Phi^{-1}$.

\bigskip{}

\paragraph{Assume that the condition $\mathcal{HN}(K)$ is true for some $K\in\llbracket0,N\rrbracket$.} Since the criterion (\ref{fondacriterion}) is satisfied, the problem
can be interpreted as before. The matter is to find two functions
$V_{K+1}^{s}:={}^{t}\Phi\, v_{k}^{s}\in(\mathcal{H}_{T}^{\infty})^{tot}$
and $V_{K+1}^{f}:={}^{t}\Phi\, v_{k}^{f}\in\left(\mathcal{E}_{\delta}^{\infty}\right)^{tot}$
such that:

\medskip{}

- For $M\geq3\,$: 
\begin{align*}
\mathcal{A}\, V_{K+1}^{s}={}^{t}\Phi\, f_{k}^{nl},\qquad\mathcal{B}\, V_{K+1}^{f} & ={}^{t}\Phi\, g_{k}^{nl},\qquad(V_{K+1}^{s}+V_{K+1}^{f})(0,\cdot)={}^{t}\Phi\, v_{K+1}^{0}(\cdot)\,.
\end{align*}
 - For $M=2\,$: 
\begin{align*}
\widetilde{\mathcal{A}}\, V_{K+1}^{s}={}^{t}\Phi f_{k}^{nl}\,,\qquad\mathcal{B}\, V_{K+1}^{f} & ={}^{t}\Phi\, g_{k}^{nl},\qquad(V_{K+1}^{s}+V_{K+1}^{f}\ )(0,\cdot)={}^{t}\Phi\, v_{K+1}^{0}(\cdot)\,.
\end{align*}
 The induction hypothesis applied with the index $K$ together with
Lemmas \ref{lemslowsource}, \ref{lemslowsourcebiss} and \ref{lemfastsource}
say that the functions $^{t}\Phi\, f_{k}^{nl}$ and $^{t}\Phi\, g_{k}^{nl}$
are known source terms with the expected $(\mathcal{H}_{T}^{\infty})^{tot}$
and $(\mathcal{E}_{\delta}^{\infty})^{tot}$ regularities. \\

-When $M\geq3$, the initial condition can be decomposed into 
\begin{align}
^{t}(V_{K+1}^{s1\perp},V_{K+1}^{s2\perp})(0,\cdot)+{}^{t}(V_{K+1}^{f1\perp},V_{K+1}^{f2\perp})(0,\cdot) & =(I-\Pi)v_{K+1}^{0}(\cdot)\,,\label{iniyk-pi}\\
^{t}(V_{K+1}^{s1\sslash},V_{K+1}^{s2\sslash})(0,\cdot)+{}^{t}(V_{K+1}^{f1\sslash},V_{K+1}^{f2\sslash})(0,\cdot) & =\Pi v_{K+1}^{0}(\cdot)\,.\label{iniykpi}
\end{align}
 In view of (\ref{VGf1}) and (\ref{VGf2}), as clearly indicated
in the statement of Lemma \ref{leminvBk}, the expression $V_{K+1}^{f\sslash}(0,\cdot)$
depends only on $^{t}\Phi g_{K+1}^{nl}$. We can determine $V_{K+1}^{s\sslash}(0,\cdot)$
through (\ref{iniykpi}). Knowing what is $V_{K+1}^{s\sslash}(0,\cdot)$,
formulas (\ref{VGs1}) and (\ref{VGs2}) give access to $V_{K+1}^{s\perp}(0,\cdot)$.
There remains to use the condition (\ref{iniyk-pi}) in order to deduce
$V_{K+1}^{f1\perp}(0,\cdot)$ and $V_{K+1}^{f2\perp}(0,\cdot)$. Remark
that the initial data $V_{K+1}^{s}(0,\cdot)$ and $V_{K+1}^{f}(0,\cdot)$
thus obtained inherit the expected $\mathcal{H}^{\infty}(\T\times\R)$
smoothness.

\medskip{}

Again, we apply Lemmas \ref{leminvAk} and \ref{leminvBk} in the
case of the initial data $V_{K+1}^{s\sslash}(0,\cdot)$ and $V_{K+1}^{f\perp}(0,\cdot)$
which have just been computed. As before, the preceding choices concerning
$V_{K+1}^{f\sslash}(0,\cdot)$ and $V_{K+1}^{s\perp}(0,\cdot)$ are
sufficient to guarantee (\ref{iniykpi}) and (\ref{iniyk-pi}). We
find that $V_{K+1}^{s}\in(\mathcal{H}_{T}^{\infty})^{tot}$ and $V_{K+1}^{f}\in(\mathcal{E}_{\delta}^{\infty})^{tot}$.
To conclude, it suffices to come back to $v_{K+1}^{s}\in\mathcal{H}_{T}^{\infty}$
and $v_{K+1}^{f}\in\mathcal{E}_{\delta}^{\infty}$ through $\Phi^{-1}$.\\

-When $M=2$, the same types of arguments prevail. This ends the induction.
\vspace{-0.3cm}
 \hfill{}$\square$\\

\bigskip{}

From the preceding construction, we can also deduce the following
information. \begin{corollary} {[}Nonlinear homogenization{]}\label{corhomogeneisationnonlinear}
For all $k\in\llbracket0,N+1\rrbracket$, the expression $\Pi v_{k}^{s}$
can be determined through the following parabolic equation, 
\begin{align}
\partial_{t}\Pi v_{k}^{s}-\left(\mu+\frac{1}{\mu}\Pi\left((\partial_{\theta}^{-1}h)^{2}\right)\right)\partial_{yy}\Pi v_{k}^{s}=S_{k}^{nl}\,,\label{eqcorhomogenizationnl}
\end{align}
 where the source term $S_{k}^{nl}:={}^{t}(S_{1,k}^{nl},S_{2,k}^{nl})$
depends on the index $j$ with $j<k$. This fact may be formulated
by writing $^{t}(S_{1,k}^{nl},S_{2,k}^{nl})={}^{t}(f_{1,k}^{nl},f_{2,k}^{nl})(v_{0}^{s},\ldots,v_{k-1}^{s})$.
\end{corollary}


\subsubsection{Approximated solutions}

\label{concerningsolutions} 

In this subsection, we prove estimate \eqref{eqapproxsolv}. We assume
that $m\geq2$ so that $H^{m}(\R^{2})$ is an algebra.\\

We explicitly compute the action of the operator $\mathcal{L}^{a}$
on the approximated solution $v_{\varepsilon}^{a}$ built in the previous
subsection and estimate the remainder $R_{\varepsilon}:=\mathcal{L}^{a}(\varepsilon,v_{\varepsilon}^{a})$
in $H^{m}$-norm.

To this ends, we justify that the slow profiles $v_{\varepsilon}^{as}$
and the fast profile $v_{\varepsilon}^{af}$ are good approximation
of operator $\mathcal{L}^{as}$ and $\mathcal{L}^{af}$.

One aspect of the proof is to compute the difference between $\mathcal{L}^{af}$and
$\mathcal{L}^{aft}$ given by some Taylor formula. To control it,
we provide the following Lemma.

\begin{lemma}\label{lemexpdecay}Let $m\geq2$ be an integer and
$\delta\in]0,\mu[$. Let $f\in\mathcal{E}_{\delta}^{m}(\T\times\R)$,
$g\in\mathcal{H}_{T}^{m,0}(\T\times\R)$. On the strip $[0,T]$, consider
the function $h_{exp}^{\varepsilon}(t,\cdot):=f\left(\varepsilon^{-2}t,\cdot\right)\int_{0}^{t}u^{N}g(u,\cdot)du\,.$
Then the family $\left\{ \varepsilon^{-(N+1)}h_{exp}^{\varepsilon}\right\} _{\varepsilon}$
is bounded in $\mathcal{H}_{T}^{m,0}(\T\times\R)$, \emph{i.e.}: 
\[
\sup_{\varepsilon\in]0,1]}\sup_{t\in[0,T]}\left\Vert \varepsilon^{-(N+1)}h_{exp}^{\varepsilon}(t,\cdot)\right\Vert _{H^{m}(\T\times\R)}<\,+\infty.
\]
 \end{lemma} The proof is obvious using some Gagliardo-Niremberg's
estimate.

\begin{proof}[Proof of Proposition \ref{eqapproxsolv}] First,
we decompose the action of $\mathcal{L}^{a}$ on $v_{\varepsilon}^{a}$
into 
\begin{equation}
\mathcal{L}^{a}(\varepsilon,v_{\varepsilon}^{a}(t,\cdot))=\mathcal{L}^{as}(\varepsilon,v_{\varepsilon}^{as}(t,\cdot))+\mathcal{L}^{af}(\varepsilon,v_{\varepsilon}^{as}(t,\cdot),v_{\varepsilon}^{af}(t/\varepsilon^{2},\cdot)),\label{eqapproxav}
\end{equation}
 where the operators $\mathcal{L}^{as}$ and $\mathcal{L}^{af}$ are
defined in \eqref{eqlas} and \eqref{eqlaf}.\\

$\circ$ Writing $v_{\varepsilon}^{as}$ as the sum $v_{\varepsilon}^{as}=\sum_{k=0}^{N+1}\varepsilon^{k}v_{k}^{s}$
and using the cascade of equations \eqref{equtjnl}, we obtain: 
\begin{align}\label{eqRs}
\mathcal{L}^{as}(\varepsilon,v_{\varepsilon}^{as}): & =\sum_{k=N}^{N+1}\varepsilon^{k}\partial_{t}v_{k}^{s}+\varepsilon^{N}\partial_{y}v_{N+1}^{s}\nonumber \\
 & +\sum_{k=N}^{2(N+1)+(M-2)}\varepsilon^{k}\sum_{(i,j)\in\mathcal{J}(M,k+2)}v_{1,i}^{s}\,\partial_{\theta}v_{j}^{s}+\sum_{k=N}^{2(N+1)+(M-1)}\varepsilon^{k}\sum_{(i,j)\in\mathcal{J}(M,k+1)}v_{2,i}^{s}\partial_{y}v_{j}^{s}\nonumber \\
 & -\begin{pmatrix}{\displaystyle \mu\,\sum_{k=N}^{N+1}\varepsilon^{k}\partial_{\theta\theta}v_{1,k}^{s}+\lambda\,\varepsilon^{N}\partial_{\theta\theta}v_{1,N+1}^{s}+\lambda\,\sum_{k=N}^{N+1}\varepsilon^{k}\partial_{\theta y}v_{2,k}^{s}}\\
{\displaystyle \mu\,\sum_{k=N}^{N+1}\varepsilon^{k}\partial_{\theta\theta}v_{2,k}^{s}+\lambda\,\sum_{k=N}^{N+1}\varepsilon^{k}\partial_{\theta y}v_{1,k}^{s}+\lambda\sum_{k=N}^{N+2}\varepsilon^{k}\partial_{yy}v_{2,k-1}^{s}.}
\end{pmatrix}.
\end{align}
 First we can factorize by $\varepsilon^{N}$ in the above expression.
Then, since $v_{k}^{s}$ is in $\mathcal{H}_{T}^{\infty}$ for $k\in\,\llbracket0,N+1\rrbracket$,
we estimate \eqref{eqRs} in $H^{m}$-norm and get for all integer
$m\geq2$: 
\[
\sup_{\varepsilon\in]0,1]}\varepsilon^{-N}\sup_{t\in[0,T]}\left\Vert R_{\varepsilon}^{s}(t,\cdot)\right\Vert _{H^{m}(\T\times\R)}<+\infty.
\]
 Thus $v_{\varepsilon}^{s}$ is an approximated solution for the operators
$\mathcal{L}^{as}$ (up to order $N$).\\

$\circ$ We constructed $v_{\varepsilon}^{f}$ so that it approximates
the operator $\mathcal{L}^{aft}$ instead of $\mathcal{L}^{af}$.
To pass from $\mathcal{L}^{af}$ to $\mathcal{L}^{aft}$, we perform
a Taylor formula (with respect to $t$ up to order $N-1$): 
\[
\mathcal{L}^{af}(\varepsilon,v_{\varepsilon}^{f},v_{\varepsilon}^{s})=\mathcal{L}^{aft}(\varepsilon,v_{\varepsilon}^{f})+R_{\varepsilon}^{tay},
\]
 where the remainder is defined as: 
\begin{align*}
R_{\varepsilon}^{tay}(t,.): & =\varepsilon^{M-2}\left(\int_{0}^{t}\frac{u^{N}}{N!}(\partial_{t})^{N-1}v_{\varepsilon}^{1s}(u,.)\, du\ \partial_{\theta}v_{\varepsilon}^{f}(t/\varepsilon^{2},\cdot)\right.\\
 & \hspace{6,5cm}\left.+\varepsilon\int_{0}^{t}\frac{u^{N}}{N!}(\partial_{t})^{N-1}v_{\varepsilon}^{2s}(u,.)\, du\ \partial_{y}v_{\varepsilon}^{f}(t/\varepsilon^{2},\cdot)\right)\\
 & +\varepsilon^{M-2}\left(v_{\varepsilon}^{f}(t/\varepsilon^{2},\cdot)\ \int_{0}^{t}\frac{u^{N}}{N!}(\partial_{t})^{N-1}\partial_{\theta}v_{\varepsilon}^{s}(u,\cdot)du\right.\\
 & \hspace{6,5cm}+\left.v_{\varepsilon}^{f}(t/\varepsilon^{2},\cdot)\ \int_{0}^{t}\frac{u^{N}}{N!}(\partial_{t})^{N-1}\partial_{y}v_{\varepsilon}^{s}(u,\cdot)du\right).
\end{align*}
 According to Proposition~\ref{propconstructionvelocity}, the hypothesis
of Lemma \ref{lemexpdecay} are satisfied. For all integer $m\geq2$
we have: 
\[
\sup_{\varepsilon\in]0,1]}\sup_{t\in[0,T]}\varepsilon^{-N}\left\Vert R_{\varepsilon}^{tay}(t,\cdot)\right\Vert _{H^{m}(\T\times\R)}<+\infty.
\]

$\circ$ Finally, we plug $v_{\varepsilon}^{s}=~\sum_{k=0}^{N+1}\varepsilon^{k}v_{k}^{s}$
and $v_{\varepsilon}^{f}=\sum_{k=0}^{N+1}\varepsilon^{k}v_{k}^{f}$
into $\mathcal{L}^{aft}$ and use the cascade of equations \eqref{equtaujnl}.
We obtain: 
\begin{align*}
\mathcal{L}^{aft}(\varepsilon,v_{\varepsilon}^{f})(t/\varepsilon^{2})=R_{\varepsilon}^{1f}+R_{\varepsilon}^{2f}.
\end{align*}
 The remainder $R_{\varepsilon}^{1f}$ is defined as 
\begin{align*}
R_{\varepsilon}^{1f}: & =\varepsilon^{N}\partial_{y}v_{N+1}^{f}(t/\varepsilon^{2},\cdot)+\sum_{k=N}^{2(N+1)+(M-2)}\varepsilon^{k}\sum_{(i,j)\in\mathcal{J}(M,k+2)}v_{1,i}^{f}(t/\varepsilon^{2},\cdot)\,\partial_{\theta}v_{j}^{f}(t/\varepsilon^{2},\cdot)\\
 & +\sum_{k=N}^{2(N+1)+(M-1)}\varepsilon^{k}\sum_{(i,j)\in\mathcal{J}(M+1,k+2)}v_{2,i}^{f}(t/\varepsilon^{2},\cdot)\partial_{y}v_{j}^{f}(t/\varepsilon^{2},\cdot)\\
 & -\begin{pmatrix}{\displaystyle \mu\,\sum_{k=N}^{N+1}\varepsilon^{k}\partial_{\theta\theta}v_{1,k}^{f}(t/\varepsilon^{2},\cdot)+\lambda\,\varepsilon^{N}\partial_{\theta\theta}v_{1,N+1}^{f}(t/\varepsilon^{2},\cdot)+\lambda\,\sum_{k=N}^{N+1}\varepsilon^{k}\partial_{\theta y}v_{2,k}^{f}(t/\varepsilon^{2},\cdot)}\\
{\displaystyle \mu\,\sum_{k=N}^{N+1}\varepsilon^{k}\partial_{\theta\theta}v_{2,k}^{f}(t/\varepsilon^{2},\cdot)+\lambda\,\sum_{k=N}^{N+1}\varepsilon^{k}\partial_{\theta y}v_{1,k}^{f}(t/\varepsilon^{2},\cdot)+\lambda\sum_{k=N}^{N+2}\varepsilon^{k}\partial_{yy}v_{2,k-1}^{f}(t/\varepsilon^{2},\cdot)}
\end{pmatrix},
\end{align*}
 whereas $R_{\varepsilon}^{2f}$ consists of the terms corresponding
to the Taylor formula: 
\begin{align*}
R_{\varepsilon}^{2f}: & =\sum_{k=N}^{4N+(M-2)}\varepsilon^{k}\left(\sum_{(i,j,l)\in\mathcal{I}(M,k+2)}\frac{t^{l}}{\varepsilon^{2l}l!}\partial_{t}^{l}v_{1,i}^{s}(0,\cdot)\,\partial_{\theta}v_{j}^{f}(t/\varepsilon^{2},\cdot)\right.\\
 & \hspace{7cm}+\left.\sum_{(i,j,l)\in\mathcal{I}(M,k+2)}\frac{t^{l}}{\varepsilon^{2l}l!}v_{1,i}^{f}(t/\varepsilon^{2},\cdot)\partial_{t}^{l}\partial_{\theta}v_{j}^{s}(0,\cdot)\right)\\
 & \hspace{0cm}+\sum_{k=N}^{4N+(M-1)}\varepsilon^{k}\left(\sum_{(i,j,l)\in\mathcal{I}(M,k+1)}\frac{t^{l}}{\varepsilon^{2l}l!}\partial_{t}^{l}v_{2,i}^{s}(0,\cdot)\,\partial_{y}v_{j}^{f}(t/\varepsilon^{2},\cdot)\right.\\
 & \hspace{7cm}+\left.\sum_{(i,j,l)\in\mathcal{I}(M,k+1)}\frac{t^{l}}{\varepsilon^{2l}l!}v_{1,i}^{f}(t/\varepsilon^{2},\cdot)\partial_{t}^{l}\partial_{y}v_{j}^{s}(0,\cdot)\right).
\end{align*}
 -By construction (see Proposition \ref{propconstructionvelocity}),
the profiles $\{v_{k}^{f}\}_{k\in\llbracket0,N+1\rrbracket}$ lie
in $\mathcal{E}_{\delta}^{\infty}$. Furthermore, we can factorize
$R_{\varepsilon}^{1f}$ by $\varepsilon^{N}$. We deduce that for
all integer $m\geq2$: 
\[
\sup_{\varepsilon\in]0,1]}\sup_{t\in[0,T]}\varepsilon^{-N}\left\Vert R_{\varepsilon}^{1f}(t,\cdot)\right\Vert _{H^{m}(\T\times\R)}<+\infty.
\]
 -There remains to estimate the term $R_{\varepsilon}^{2f}$. \textit{A
priori} this term can be dangerous because it contains some polynomials
in the variable $t/\varepsilon^{2}$. Nevertheless we can use the
fast decreasing behaviours of the profile $v_{\varepsilon}^{f}$.
Indeed if $f\in\mathcal{E}_{\delta}^{\infty}$ then for all $l\in\N$
the function $\tau^{l}\, f$ is also quickly decreasing: $\tau^{l}\, f\in\mathcal{E}_{\delta'}^{\infty}$
for some $0<\delta'<\delta$. Furthermore, noticing than we can factorize
by $\varepsilon^{N}$ in $R_{\varepsilon}^{2f}$, we obtain for all
integer~$m\geq2$: 
\[
\sup_{\varepsilon\in]0,1]}\sup_{t\in[0,T]}\varepsilon^{-N}\left\Vert R_{\varepsilon}^{2f}(t,\cdot)\right\Vert _{H^{m}(\T\times\R)}<+\infty.
\]
 \end{proof}

\subsection{The case of the pressure - Consequence}

\label{sectionconstructionpressure}

Here, we still assume that $M\geq2$. First of all we quickly prove
Proposition $\ref{propconstructionpressure}$. Then we take advantage
of the control obtained on $\{q_{\varepsilon}^{a}\}_{\varepsilon}$
to prove that the approximated solution $v_{\varepsilon}^{a}$ is
a good approximation for operator $(\mathcal{L}^{1},\mathcal{L}^{2})$
assuming $\nu$ is large enough (Proposition \ref{corapproxoperateur}).\\

\paragraph{Approximated pressure.}

Consider the approximated velocity $v_{\varepsilon}^{a}$ built according
to Proposition \ref{propconstructionvelocity}. Since the operator
$\mathcal{L}_{0}$ is linear with respect to the pressure variable,
we build the profile $q_{k}^{\varepsilon}$ as the solution of the
following problem: 
\begin{multline}\label{eqpressureqk}
\mathcal{L}_{0}(\varepsilon,q_{k}^{\varepsilon},v_{\varepsilon}^{a})=\partial_{t}q_{k}^{\varepsilon}+\varepsilon^{-1}h\,\partial_{y}q_{k}^{\varepsilon}\\
+\varepsilon^{M-2}(v_{\varepsilon}^{a1}\partial_{\theta}q_{k}^{\varepsilon}+\varepsilon\, v_{\varepsilon}^{2a}\partial_{y}q_{k}^{\varepsilon})+C\,\varepsilon^{M-2}q_{k}^{\varepsilon}(\partial_{\theta}v_{\varepsilon}^{a1}+\varepsilon\partial_{y}v_{\varepsilon}^{2a})=0\,
\end{multline}
 with initial data satisfying $q_{k}^{\varepsilon}(0,\cdot)=q_{k}^{0}(\cdot)$.

At fixed $\varepsilon$, for any positive time $T$, \eqref{eqpressureqk}
has an unique solution $q_{k}^{\varepsilon}$ in $\mathcal{H}_{T}^{m}(\T\times\R)$.
We recover the approximated solution $q_{\varepsilon}^{a}$ (on the
strip $[0,T]$) by summing over all the multi-indices $k\in\llbracket0,N+1\rrbracket$:
\[
q_{\varepsilon}^{a}:=\sum_{k=0}^{N+1}\varepsilon^{k}q_{k}^{\varepsilon}.
\]

However, since the transport is singular the family of solution $\{q_{\varepsilon}^{a}\}_{\varepsilon}$
is not bounded in $\mathcal{H}_{T}^{m}(\T\times\R)$. Yet, it is in
the anisotropic Sobolev spaces. In other words, Inequality \eqref{eqapproxsolq2}
is satisfied.

\paragraph{Approximated solution for the operator $\mathcal{L}$.}

A direct consequence of Inequality \eqref{eqapproxsolq2} is that
$\{(q_{\varepsilon}^{a},v_{\varepsilon}^{a})\}_{\varepsilon}$ is
an approximated solution for the operator $\mathcal{L}$, \textit{i.e.}
Proposition \ref{corapproxoperateur} is satisfied. \\

First, we have $\mathcal{L}_{0}(\varepsilon,q_{\varepsilon}^{a},v_{\varepsilon}^{a})=0$
and: 
\begin{align*}
^{t}\left(\mathcal{L}_{1},\mathcal{L}_{2}\right)(\varepsilon,q_{\varepsilon}^{a},v_{\varepsilon}^{a})=\mathcal{L}^{a}(\varepsilon,v_{\varepsilon}^{a})+C\varepsilon^{2\nu-M-2}\ {}^{t}\left(q_{\varepsilon}^{a}\partial_{\theta}q_{\varepsilon}^{a},\varepsilon\, q_{\varepsilon}^{a}\partial_{y}q_{\varepsilon}^{a}\right)\,.
\end{align*}
 The quantity $\mathcal{L}^{a}(\varepsilon,v_{\varepsilon}^{a})$
can be estimated thanks to \eqref{eqapproxsolv}. The pressure term
$\varepsilon^{2\mathcal{V}-M-2}\,{}^{t}(q_{\varepsilon}^{a}\partial_{\theta}q_{\varepsilon}^{a},\varepsilon\, q_{\varepsilon}^{a}$
$\partial_{y}q_{\varepsilon}^{a})$ can be estimated thanks to the
Gagliardo-Niremberg's Inequality. Let $\alpha\in\N^{2}$, $|\alpha|\leq m$,
then 
\begin{align*}
\left\Vert \partial^{\alpha}\left(q_{\varepsilon}^{a}\partial_{\theta}q_{\varepsilon}^{a}\right)\right\Vert _{L^{2}} & =\left\Vert \partial^{\alpha}\partial_{\theta}\left(q_{\varepsilon}^{a}\right)^{2}\right\Vert _{L^{2}}\leq2C_{g}\left\Vert q_{\varepsilon}^{a}\right\Vert _{L^{\infty}}\left\Vert q_{\varepsilon}^{a}\right\Vert _{\overset{\circ}{H}{}^{|\alpha|+1}}\,.
\end{align*}
 To recover an estimate in the anisotropic version of the Sobolev
spaces, we use \eqref{eqinfininorm} together with the following equivalence
of norms: 
\begin{equation}
\left\Vert \cdot\right\Vert _{H^{m}}\leq\varepsilon^{-m}\left\Vert \cdot\right\Vert _{H_{(1,\varepsilon)}^{m}}.\label{eqrefnorm}
\end{equation}
 Thus, we obtain 
\[
\varepsilon^{-N}\left\Vert \varepsilon^{2\nu-M-2}\,{}^{t}(q_{\varepsilon}^{a}\partial_{\theta}q_{\varepsilon}^{a},\varepsilon\, q_{\varepsilon}^{a}\partial_{y}q_{\varepsilon}^{a})\right\Vert _{H^{m}}\lesssim\varepsilon^{2\nu-M-5/2-(m+1)-N}\left\Vert q_{\varepsilon}^{a}\right\Vert _{H_{(1,\varepsilon)}^{m+1}}^{2}\lesssim\varepsilon^{2\nu-M-5/2-(m+1)-N}\,.
\]
 Assuming \eqref{eqwsN}, it completes the proof. \vspace{-0.3cm}
 \hfill{}$\blacksquare$


\section{Energy estimates}

\label{Energy} 

In this section we prove Theorem \ref{propestimationreste}. The integers
$m,\,\nu,\, M,\, N$ and $R$ satisfy property \eqref{eqws}. As ever
mentioned at the level of the introduction, page \pageref{eqintrolin2},
the main issue is to get a control over the singular term $\varepsilon^{-2}\partial_{\theta}h\, v_{\varepsilon}^{1R}$.
To desingularize it, we consider the new unknowns: 
\begin{equation}
\widetilde{q}_{\varepsilon}^{R}:=q_{\varepsilon}^{R},\qquad\widetilde{v}_{\varepsilon}^{1R}:=v_{\varepsilon}^{1R},\qquad\widetilde{v}_{\varepsilon}^{2R}:=\varepsilon v_{\varepsilon}^{2R}.\label{eqnewunknown}
\end{equation}
 It satisfies a hyperbolic-parabolic system (singular in $\varepsilon$):
\begin{subequations} \label{eqepsilon} 
\begin{flalign}
 & \partial_{t}\widetilde{q}_{\varepsilon}^{R}+\varepsilon^{-1}h\,\partial_{y}\widetilde{q}_{\varepsilon}^{R}+\varepsilon^{M-2}\left(v_{\varepsilon}^{1a}\,\partial_{\theta}\widetilde{q}_{\varepsilon}^{R}+\varepsilon v_{\varepsilon}^{2a}\,\partial_{y}\widetilde{q}_{\varepsilon}^{R}\right)+C\,\varepsilon^{M-2}\widetilde{q}_{\varepsilon}^{R}\left(\partial_{\theta}v_{\varepsilon}^{1a}+\varepsilon\,\partial_{y}v_{\varepsilon}^{2a}\right)\nonumber \\
 & \hspace{1.5cm}+\varepsilon^{M-2}\left(\widetilde{v}_{\varepsilon}^{1R}\,\partial_{\theta}q_{\varepsilon}^{a}+\widetilde{v}_{\varepsilon}^{2R}\,\partial_{y}q_{\varepsilon}^{a}\right)+C\,\varepsilon^{M-2}q_{\varepsilon}^{a}\left(\partial_{\theta}\widetilde{v}_{\varepsilon}^{1R}+\partial_{y}\widetilde{v}_{\varepsilon}^{2R}\right)=\widetilde{S}_{\varepsilon}^{0,R,N},\label{eqqepsilon}\\
 & \partial_{t}\widetilde{v}_{\varepsilon}^{R}+\varepsilon^{-1}h\,\partial_{y}\widetilde{v}_{\varepsilon}^{R}+\fbox{\ensuremath{^{t}\left(0,\varepsilon^{-1}\,\partial_{\theta}h\,\widetilde{v}_{\varepsilon}^{1R}\right)}}+\varepsilon^{M-2}\left(v_{\varepsilon}^{1a}\,\partial_{\theta}\widetilde{v}_{\varepsilon}^{R}+\varepsilon v_{\varepsilon}^{2a}\,\partial_{y}\widetilde{v}_{\varepsilon}^{R}\right)\nonumber \\
 & \hspace{3.5cm}+\varepsilon^{M-2}\left(\widetilde{v}_{\varepsilon}^{1R}\partial_{\theta}+\widetilde{v}_{\varepsilon}^{2R}\partial_{y}\right){}^{t}(v_{\varepsilon}^{1a},\varepsilon\, v_{\varepsilon}^{2a})-\mathcal{Q}_{\varepsilon}(\widetilde{v}_{\varepsilon}^{R})=\widetilde{S}_{\varepsilon}^{R,N},\label{equepsilon}
\end{flalign}
 \end{subequations} together with the initial data $\left(q_{\varepsilon}^{R}(0,\cdot),v_{\varepsilon}^{R}(0,\cdot)\right)\equiv0$.
The right-hand side of Equation \eqref{eqepsilon} is defined as:
\begin{align}
 & \label{source0}\widetilde{S}_{\varepsilon}^{0,R,N}:=-\varepsilon^{N-R}(\varepsilon^{-N}\mathcal{L}_{0}(\varepsilon,q_{\varepsilon}^{a},v_{\varepsilon}^{a}))-\varepsilon^{R+M-2}\left(\widetilde{v}_{\varepsilon}^{1R}\partial_{\theta}\widetilde{q}_{\varepsilon}^{R}+\widetilde{v}_{\varepsilon}^{2R}\partial_{y}\widetilde{q}_{\varepsilon}^{R}\right)\nonumber \\
 & \hspace{8,5cm}-C\varepsilon^{R+M-2}\widetilde{q}_{\varepsilon}^{R}\left(\partial_{\theta}\widetilde{v}_{\varepsilon}^{1R}+\partial_{y}\widetilde{v}_{\varepsilon}^{2R}\right),\\
 & \widetilde{S}_{\varepsilon}^{R,N}:=-\varepsilon^{N-R}\,(\varepsilon^{-N}\,{}^{t}\left(\mathcal{L}_{1}^{a}(\varepsilon,v_{\varepsilon}),\varepsilon\mathcal{L}_{2}^{a}(\varepsilon,v_{\varepsilon})\right))-\varepsilon^{R+M-2}\left(\widetilde{v}_{\varepsilon}^{1R}\partial_{\theta}\widetilde{v}_{\varepsilon}^{R}+\widetilde{v}_{\varepsilon}^{2R}\partial_{y}\widetilde{v}_{\varepsilon}^{R}\right)\nonumber \\
 & \hspace{7cm}-\frac{C\varepsilon^{2\nu-M-R-2}}{2}\,^{t}\left(\partial_{\theta},\varepsilon^{2}\partial_{y}\right)\left(q_{\varepsilon}^{a}+\varepsilon^{R}\widetilde{q}_{\varepsilon}^{R}\right)^{2}.\label{source12}
\end{align}
 This is clearly a non-linear problem and the variable of pressure
$\widetilde{q}_{\varepsilon}^{R}$ and the variables of velocity $\widetilde{v}_{\varepsilon}^{R}$
are coupled. The dissipation is turned into 
\begin{equation}
\mathcal{Q}_{\varepsilon}\widetilde{v}_{\varepsilon}^{R}:=\begin{pmatrix}\mathcal{Q}_{\varepsilon}^{1}\widetilde{v}_{\varepsilon}^{R}\\
\mathcal{Q}_{\varepsilon}^{2}\widetilde{v}_{\varepsilon}^{R}
\end{pmatrix}=\frac{1}{\varepsilon^{2}}\begin{pmatrix}\lambda\left(\partial_{\theta\theta}\widetilde{v}_{\varepsilon}^{1R}+\varepsilon^{2}\partial_{yy}\widetilde{v}_{\varepsilon}^{1R}\right)+\mu{\varepsilon}\left(\partial_{\theta\theta}\widetilde{v}_{\varepsilon}^{1R}+\partial_{y\theta}\widetilde{v}_{\varepsilon}^{2R}\right)\\
\lambda\left(\partial_{\theta\theta}\widetilde{v}_{\varepsilon}^{2R}+\varepsilon^{2}\partial_{yy}\widetilde{v}_{\varepsilon}^{2R}\right)+\varepsilon^{2}\mu{\varepsilon}\left(\partial_{\theta y}\widetilde{v}_{\varepsilon}^{1R}+\partial_{yy}\widetilde{v}_{\varepsilon}^{2R}\right)
\end{pmatrix}\,.\label{eqQdiss}
\end{equation}

We clearly desingularize the hyperbolic part to a cost on the parabolic
part, $\mathcal{Q}_{\varepsilon}$. It can no longer satisfy an estimate
such as \eqref{controldissipation0} necessary to keep control over
singular terms in \eqref{eqepsilon}. One important aspect of the
proof is to obtain Inequality \eqref{controldissipationQ}. If it
is not satisfied, the mechanism of the proof fails.\\

Thus, in this section, we look for accurate energy estimates for the
unknown $(\widetilde{q}_{\varepsilon}^{R},\widetilde{v}_{\varepsilon}^{R})$.
We easily go back to the initial variables according to \eqref{eqnewunknown}.\\

\paragraph{Coupling and nonlinear aspects.}

In Equation \eqref{equepsilon}, the variables of pressure and velocity
are only coupled threw the term: 
\begin{equation}
\frac{C\varepsilon^{2\nu-M-R-2}}{2}\,^{t}(\partial_{\theta},\varepsilon^{2}\,\partial_{y})\left(q_{\varepsilon}^{a}+\varepsilon^{R}\widetilde{q}_{\varepsilon}^{R}\right)^{2}.\label{eqcoupled}
\end{equation}
 If $\nu$ is large enough, the Equation \eqref{equepsilon} is somehow
independent of the pressure. It seems that we can deal with the velocity
and then deal with the pressure. This is reinforced by the fact that
the pressure and the velocity are estimated in different Sobolev Spaces.

Yet, the term \eqref{eqcoupled} has to be estimated carefully. We
have to link the anisotropic Sobolev norms with the classical Sobolev
norms \eqref{eqrefnorm}. It explains the loss of precision on $w_{m}$
with respect to the regularity $m$ (see Equation \eqref{eqws}).

In practice, we perform energy estimates on Equation \eqref{equepsilon}.
Then we plug the estimates obtained on the velocity into Equation
\eqref{eqqepsilon}. We underline here the fact that we need more
regularity on the velocity to estimate the term $\widetilde{S}_{\varepsilon}^{0,R,N}$
(defined by Equation \eqref{source0}). Thus, the \textit{regularization}
of the velocity (thanks to the dissipation) plays again a crucial
role.\\

Finally, the problem is nonlinear (for example term \eqref{eqcoupled}).
Classically, nonlinear terms are estimated thanks to the Gagliardo-Niremberg's
estimate which required a $L^{\infty}$ control over the unknowns.
To get this control, we introduce the characterized time $T_{\varepsilon}^{*}$:
\[
T_{\varepsilon}^{*}:=\text{min}\left(1,\sup_{T\in[0,T_{\varepsilon}]}\left\{ \forall\, t\in[0,T],\ \left\Vert \widetilde{q}_{\varepsilon}^{R}(t,\cdot)\right\Vert _{H_{(1,\varepsilon)}^{m+3}(\T\times\R)}\leq2,\ \left\Vert \widetilde{v}_{\varepsilon}^{R}(t,\cdot)\right\Vert _{H^{m+3}(\T\times\R)}\leq2\right\} \right).
\]
 It provides $L^{\infty}$-estimates on the strip $[0,T_{\varepsilon}^{*}]$,
for all $\varepsilon\in]0,1]$ (see \eqref{eqinfininorm}): 
\begin{equation}
\forall\, t\in[0,T_{\varepsilon}^{*}],\quad\left\Vert \sqrt{\varepsilon}\,\widetilde{q}_{\varepsilon}^{R}(t,\cdot)\right\Vert _{W_{(1,\varepsilon)}^{m+1,\infty}(\T\times\R)}\leq2,\qquad\left\Vert \widetilde{v}_{\varepsilon}^{R}(t,\cdot)\right\Vert _{W^{m+1,\infty}(\T\times\R)}\leq2.\label{eqnormsup}
\end{equation}

\subsubsection{Results and Consequences}

As indicated, we first prove an estimate over the velocity on the
interval $[0,T_{\varepsilon}^{*}]$, thanks to Equation \eqref{equepsilon}:

\begin{proposition}\label{propestu} Let $\widetilde{v}_{\varepsilon}^{R}$
be a solution of (\ref{eqepsilon}) on $[0,T_{\varepsilon}^{*}]$.
There exist a positive constant $\varepsilon_{crit}$ and two positive
constants $K_{m}^{1}$ and $K_{m}^{2}$ (independent of $\varepsilon$)
such that 
\begin{equation}
\forall\,\varepsilon\in]0,\varepsilon_{crit}],\quad\forall\, t\in[0,T_{\varepsilon}^{*}],\qquad\left\Vert \widetilde{v}_{\varepsilon}^{R}(t,\cdot)\right\Vert _{H^{m+3}(\T\times\R)}^{2}\leq\varepsilon^{2w_{m}}\, K_{m}^{1}\left(e^{K_{m}^{2}\, t}-1\right).\label{eqpropestu1}
\end{equation}
 Furthermore, one can prove a \em{regularization} property. Select
a multi-index $\alpha\in\N^{2}$ of length $|\alpha|$ smaller than
$m+4$; then 
\begin{equation}
\forall\,\varepsilon\in]0,\varepsilon_{crit}],\ \forall\, t\in[0,T_{\varepsilon}^{*}],\quad\int_{0}^{t}\left\Vert \varepsilon^{-(1-\delta_{\alpha_{1},0})}\partial^{\alpha}\widetilde{v}_{\varepsilon}^{R}(s,\cdot)\right\Vert _{L^{2}(\T\times\R)}^{2}ds\leq K_{m}^{1}\,\varepsilon^{2w_{m}}t\,,\label{eqpropestu2}
\end{equation}
 where $\delta_{i,j}$ denotes the Kronecker symbol (of
two integers): $\delta_{i,j}=0$ if $i\neq j$ and $\delta_{i,i}=1$.
\end{proposition}

By plugging \eqref{eqpropestu1}-\eqref{eqpropestu2} in \eqref{eqqepsilon}
we obtain:

\begin{proposition}\label{lemestq}Let $\widetilde{q}_{\varepsilon}^{R}$
be a solution of Equation (\ref{eqepsilon}) on $[0,T_{\varepsilon}^{*}]$.
There exist a positive constant $\varepsilon_{crit}$ and a positive
constant $K_{m}^{1}$ (independent of $\varepsilon$) such that 
\[
\forall\,\varepsilon\in]0,\varepsilon_{crit}],\quad\forall\, t\in[0,T_{\varepsilon}^{*}],\qquad\left\Vert \widetilde{q}_{\varepsilon}^{R}(t,\cdot)\right\Vert _{H_{(1,\varepsilon)}^{m+3}(\T\times\R)}^{2}\leq\varepsilon^{2w_{m}}K_{m}^{1}\, t.
\]
 \end{proposition}

Those two lemmas allow to prove an accurate estimate of the solution
on a strip independent of~$\varepsilon$.

\begin{corollary}\label{propestglobal} Consider integers $m$, $\nu$,
$M$, $N$ and $R$ satisfying the condition~\eqref{eqws}. Then,
there exist two positive constants $\varepsilon_{c}$ and $T_{c}$
(independent of $\varepsilon$), and a positive constant $c_{err}>0$
such that, 
\[
\forall\, t\in[0,T_{c}],\quad\forall\,\varepsilon\in]0,\varepsilon_{c}],\qquad\left\Vert \widetilde{q}_{\varepsilon}^{R}(t,\cdot)\right\Vert _{H_{(1,\varepsilon)}^{m+3}(\T\times\R)}\leq c_{err},\qquad\left\Vert \widetilde{v}_{\varepsilon}^{R}(t,\cdot)\right\Vert _{H^{m+3}(\T\times\R)}\leq c_{err}.
\]
 \end{corollary}

\begin{proof}[Proof of Corollary \ref{propestglobal}] We argue
by contradiction: 
\[
\forall\,(\widetilde{\varepsilon},T)\in]0,1]\times[0,1],\quad\exists\,\varepsilon\in]0,\widetilde{\varepsilon}],\ T_{\varepsilon}^{*}<T.
\]

We recall that there exist $\varepsilon_{d}$ and $C_{m}$ positive
constants such that, for all $\varepsilon\in]0,\varepsilon_{d}]$
and for all time $t\in[0,T_{\varepsilon}^{*}]$: 
\[
\left\Vert \widetilde{q}_{\varepsilon}^{R}(t,\cdot)\right\Vert _{H_{(1,\varepsilon)}^{m+3}(\T\times\R)}^{2}\leq\varepsilon^{2w_{m}}\, C_{m}\, t,\qquad\left\Vert \widetilde{v}_{\varepsilon}^{R}(t,\cdot)\right\Vert _{H^{m+3}(\T\times\R)}^{2}\leq\varepsilon^{2w_{m}}C_{m}\, t\,.
\]
 We choose for instance $T=\min(\frac{1}{2\, C_{m}},\frac{1}{2})$
and $\widetilde{\varepsilon}=\varepsilon_{d}<1$. In particular, $T<1$.
From assumption, there exists $\varepsilon_{0}\in]0,\varepsilon_{d}]$
such that $T_{\varepsilon_{0}}^{*}<T$. Furthermore, since $w_{m}\geq0$
from condition \eqref{eqws} we get: 
\[
\forall\, t\in[0,T_{\varepsilon_{0}}^{*}],\qquad\left\Vert \widetilde{q}_{\varepsilon_{0}}^{R}(t,\cdot)\right\Vert _{H_{(1,\varepsilon_{0})}^{m+3}(\T\times\R)}\leq\frac{1}{2}<2,\qquad\left\Vert \widetilde{v}_{\varepsilon_{0}}^{R}(t,\cdot)\right\Vert _{H^{s+3}(\T\times\R)}\leq\frac{1}{2}<2.
\]

Now consider the applications 
\[
t\in[0,T_{\varepsilon_{0}}^{*}[\longmapsto\left\Vert \widetilde{q}_{\varepsilon_{0}}^{R}(t,\cdot)\right\Vert _{H_{(1,\varepsilon_{0})}^{m+3}(\T\times\R)}\quad\text{and}\quad t\in[0,T_{\varepsilon_{0}}^{*}[\longmapsto\left\Vert \widetilde{v}_{\varepsilon_{0}}^{R}(t,\cdot)\right\Vert _{H^{m+3}(\T\times\R)}.
\]
 They are continuous. $\widetilde{q}_{\varepsilon_{0}}^{R}$ (respectively
$\widetilde{v}_{\varepsilon_{0}}^{R}$) can be extended in time as
long as the quantity~$\left\Vert \widetilde{q}_{\varepsilon_{0}}^{R}(t,\cdot)\right\Vert _{H_{(1,\varepsilon_{0})}^{m+3}}$
( respectively $\left\Vert \widetilde{v}_{\varepsilon_{0}}^{R}(t,\cdot)\right\Vert _{H^{m+3}}$)
remains bounded.

It follows that we can find $T\in]T_{\varepsilon_{0}}^{*},T_{\varepsilon_{0}}[$
such that for all time $t\in[0,T]$, $\left\Vert \widetilde{q}_{\varepsilon_{0}}^{R}(t,.)\right\Vert _{H_{(1,\varepsilon_{0})}^{m+3}}<2$
(respectively $\left\Vert \widetilde{v}_{\varepsilon_{0}}^{R}(t,.)\right\Vert _{H^{m+3}}<2$).
This is in contradiction with the definition of $T_{\varepsilon_{0}}^{*}$.
\end{proof}

\bigskip{}

\medskip{}

Assuming Propositions \ref{propestu} and \ref{lemestq} are satisfied,
Corollary \ref{propestglobal} holds. We go back to the initial unknowns
thanks to the Equation \eqref{eqnewunknown} and obtain Theorem \ref{propestimationreste}.

Thus in Subsection \ref{estvelocity}, we prove Proposition \ref{propestu}.
We take advantage of Inequalities \eqref{eqpropestu1} and~\eqref{eqpropestu2}
to obtain Proposition \ref{lemestq}.

\subsubsection{Notations}

Here we introduce some notations required for the proof of Propositions
\ref{propestu} and \ref{lemestq}.\\

Let $(q_{\varepsilon}^{a},v_{\varepsilon}^{a})$ an approximated solution
of order $N$ constructed on the interval $[0,1]$ according to Proposition
\ref{propconstructionvelocity} and Proposition \ref{propconstructionpressure}.
The family $\{(q_{\varepsilon}^{a},v_{\varepsilon}^{a})\}_{\varepsilon}$
lies in $\mathcal{H}_{1,(1,\varepsilon)}^{m+6,0}\times\mathcal{H}_{1}^{m+6,0}$
and satisfies Inequalities (\ref{eqapproxsolv}) and (\ref{eqapproxsolq2}).
We denote by $C_{a}$, $C_{\mathcal{L}^{a}}$ and $C_{\mathcal{L}}$
positive constants such that: 
\begin{equation}
\sup_{\varepsilon\in]0,1]}\quad\sup_{t\in[0,1]}\left\Vert v_{\varepsilon}^{a}(t,\cdot)\right\Vert _{H^{m+6}}<C_{a},\qquad\sup_{\varepsilon\in]0,1]}\quad\sup_{t\in[0,1]}\left\Vert q_{\varepsilon}^{a}(t,\cdot)\right\Vert _{H_{(1,\varepsilon)}^{m+6}}<C_{a},\label{eqhupvq}
\end{equation}
 and 
\begin{align}
\sup_{\varepsilon\in\,]0,1]}\sup_{t\in[0,1]}\quad\left\Vert \varepsilon^{-N}\,\mathcal{L}^{a}(\varepsilon,v_{\varepsilon}^{a})\right\Vert _{H^{m+6}(\T\times\R)}<C_{\mathcal{L}^{a}}\,,\label{eqlah}\\
\sup_{\varepsilon\in\,]0,1]}\sup_{t\in[0,1]}\quad\left\Vert \varepsilon^{-N}\,\mathcal{L}_{0}(\varepsilon,q_{\varepsilon}^{a},v_{\varepsilon}^{a})\right\Vert _{H_{(1,\varepsilon)}^{m+6}(\T\times\R)}<C_{\mathcal{L}}\,.\label{eql0h}
\end{align}

In what follows we adopt the conventions: 
\[
\forall\, m\in\N^{*},\qquad\left\Vert f\right\Vert _{H^{-m}}\equiv0,
\]
 to simplify all the statement.


\subsection{Energy estimates for the velocity}

\label{estvelocity}

In this section, we prove Proposition \ref{propestu} by induction
on the size of $m$ settings: 
\begin{equation}
\mathcal{P}(m)\,:\ \text{{\it " Proposition \ref{propestu} holds up to the integer \ensuremath{m}".}}\label{recurezrnl}
\end{equation}
 To go from $m$ to $m+1$, we prove an \textit{energy inequality}
for the velocity performing an energy method on Equation (\ref{equepsilon})
in the homogeneous Sobolev space $\overset{\circ}{H}{}^{m}$ (defined
page \pageref{asymplty}). \begin{lemma}\label{lemall}There exist
$c_{1}$ and $\varepsilon_{d}$ two positive constants (which only
depend on $m$) such that for any $J\in\llbracket0,m+3\rrbracket$,
there exist four positive constants $C_{p}$, $C_{J}^{1}$, $C_{J}^{2}$
and $C_{J}^{3}$ such that for all $\varepsilon\in]0,\varepsilon_{d}]$,
and for all time $t\in[0,T_{\varepsilon}^{*}]$, 
\begin{multline*}
\frac{1}{2}\partial_{t}\left\Vert \widetilde{v}_{\varepsilon}^{R}(t,\cdot)\right\Vert _{\overset{\circ}{H}{}^{J}(\T\times\R)}^{2}+c_{1}\sum_{|\alpha|=J}\Phi_{\varepsilon}(\nabla,\partial^{\alpha}\widetilde{v}_{\varepsilon}^{R})(t)\leq C_{J}^{1}\left\Vert \widetilde{v}_{\varepsilon}^{R}(t,\cdot)\right\Vert _{\overset{\circ}{H}{}^{J}(\T\times\R)}^{2}+(J+1)C_{p}\varepsilon^{2w_{m}}\\
+C_{J}^{2}\left\Vert \widetilde{v}_{\varepsilon}^{R}(t,\cdot)\right\Vert _{H^{J-1}(\T\times\R)}^{2}+C_{J}^{3}\sum_{|\alpha|=J}\left\Vert \varepsilon^{-(1-\delta_{\alpha_{1},0})}\partial^{\alpha}\widetilde{v}_{\varepsilon}^{R}(t,\cdot)\right\Vert _{L^{2}(\T\times\R)}^{2}.
\end{multline*}
 \end{lemma} Subsections \ref{substepone}-$\ldots$-\ref{sublaststep}
are dedicated to the proof of Lemma \ref{lemall}. Then in Subsection
\ref{subsectionpropestu}, we finally prove Proposition \ref{propestu}.

First, consider a multi-index $\alpha\in\N^{2}$ and differentiate $\alpha$
times Equation $(\ref{equepsilon})$. Then we multiply by $\partial^{\alpha}\widetilde{v}_{\varepsilon}^{R}$
and integrate (with respect to the space variables $\theta$ and $y$),
\begin{multline}\label{eqenergyplouche}
\frac{1}{2}\partial_{t}\left\Vert \partial^{\alpha}\widetilde{v}_{\varepsilon}^{R}\right\Vert _{L^{2}}^{2}-\left\langle \partial^{\alpha}(\mathcal{Q}_{\varepsilon}\widetilde{v}_{\varepsilon}^{R}),\partial^{\alpha}\widetilde{v}_{\varepsilon}^{R}\right\rangle =\left\langle \partial^{\alpha}S_{\varepsilon}^{R,N},\partial^{\alpha}\widetilde{v}_{\varepsilon}^{R}\right\rangle -\left\langle \partial^{\alpha}\,(\mathcal{A}\,\widetilde{v}_{\varepsilon}^{R}),\partial^{\alpha}\widetilde{v}_{\varepsilon}^{R}\right\rangle \\
-\left\langle \partial^{\alpha}\,(\mathcal{B}\widetilde{v}_{\varepsilon}^{R}),\partial^{\alpha}\widetilde{v}_{\varepsilon}^{R}\right\rangle -\left\langle \partial^{\alpha}\,(\mathcal{C}\widetilde{v}_{\varepsilon}^{R}),\partial^{\alpha}\widetilde{v}_{\varepsilon}^{R}\right\rangle -\left\langle \partial^{\alpha}\,(\mathcal{H}\widetilde{v}_{\varepsilon}^{R}),\partial^{\alpha}\widetilde{v}_{\varepsilon}^{R}\right\rangle ,
\end{multline}
 where the operators $\mathcal{H}$, $\mathcal{A}$, $\mathcal{B}$
and $\mathcal{C}$ are defined as follows: 
\begin{align*}
 & \mathcal{H}\,\widetilde{v}_{\varepsilon}^{R}:=\varepsilon^{-1}h\,\partial_{y}\widetilde{v}_{\varepsilon}^{R},\qquad\quad\ \mathcal{A}\,\widetilde{v}_{\varepsilon}^{R}:=\varepsilon^{M-2}\left(\widetilde{v}_{\varepsilon}^{1R}\partial_{\theta}+\widetilde{v}_{\varepsilon}^{2R}\partial_{y}\right)\,{}^{t}(v_{\varepsilon}^{1a},\varepsilon\, v_{\varepsilon}^{2a}),\\
 & \mathcal{C}\,\widetilde{v}_{\varepsilon}^{R}:=^{t}\left(0,\varepsilon^{-1}\partial_{\theta}h\,\widetilde{v}_{\varepsilon}^{1R}\right),\quad\mathcal{B}\,\widetilde{v}_{\varepsilon}^{R}:=\varepsilon^{M-2}\left(v_{\varepsilon}^{1a}\partial_{\theta}\widetilde{v}_{\varepsilon}^{R}+\varepsilon v_{\varepsilon}^{2a}\partial_{y}\widetilde{v}_{\varepsilon}^{R}\right).
\end{align*}

The strategy of the estimates is the following. We first prove the
coercive estimate over $\mathcal{Q}_{\varepsilon}$ (see Lemma \ref{lemestQ}
in Subsection \ref{substepone}). Then we compute all the singular
contribution (with respect to the regularity and to $\varepsilon$)
in Lemmas \ref{controlG}-$\ldots$-\ref{lemestC} that we absorb
thanks to $\mathcal{Q}_{\varepsilon}$ (see Subsection \ref{substeptwo}).
Finally, we simply bound the remaining terms thanks to Lemma \ref{lemA}
in Subsection \ref{sublaststep}.

\subsubsection{Step one : Coercive estimate for the dissipation}

\label{substepone}

We start by estimating the term involving the dissipation $\mathcal{Q}_{\varepsilon}$.
\begin{lemma} \label{lemestQ} We recall that $\mu>0$. Select constant
$\lambda_{}$ and $\mu$ which satisfy 
\begin{equation}
\lambda_{}<4\,\mu\,.\label{eqmueps}
\end{equation}
 There exists $c_{0}>0$, there exists $\varepsilon_{d}\in]0,1]$,
such that for any function $f\in H^{1}(\T\times\R,\R^{2})$: 
\begin{equation}
\forall\,\varepsilon\in]0,\varepsilon_{d}],\qquad-\left\langle \mathcal{Q}_{\varepsilon}f,f\right\rangle \geq c_{0}\ \Phi_{\varepsilon}(\nabla,f)\,.\label{controldissipationQ}
\end{equation}
 where $\Phi_{\varepsilon}$ is defined in Equation \eqref{controldissipation0}.
\end{lemma}

\begin{proof}[Proof of Lemma \ref{lemestQ}] Let $f\in\, H^{1}(\T\times\R,\R^{2})$.
We decompose $f$ in Fourier series (in $\theta$) into: 
\[
f(\theta,y):=\sum_{k\in\Z}f_{k}(y)e^{ik\theta}\,,
\]
 then we use a Fourier transform in the variable~$y$. By the Parseval
equality, up to a constant (that we forget here), 
\[
-\left\langle \mathcal{Q}_{\varepsilon}f,f\right\rangle =\sum_{k\in\Z}\int_{\R}\widehat{f_{k}}(\xi)Q_{\varepsilon}(k,\xi)\widehat{f_{k}}(\xi)d\xi\,,\qquad\widehat{f_{k}}(\xi):=\int_{\R}e^{i\, y.\xi}f_{k}(y)\, dy\,,
\]
 where ${Q}_{\varepsilon}$ is defined as: \vspace{-0.2cm}
 
\[
{Q}_{\varepsilon}(k,\xi):=\begin{pmatrix}(\mu+\lambda{\varepsilon})\frac{k^{2}}{\varepsilon^{2}}+\xi^{2} & \frac{\lambda{\varepsilon}}{2}\left(1+\frac{1}{\varepsilon^{2}}\right)k\,\xi\\
\frac{\lambda{\varepsilon}}{2}\left(1+\frac{1}{\varepsilon^{2}}\right)k\,\xi & \mu\frac{k^{2}}{\varepsilon^{2}}+(\mu+\lambda{\varepsilon})\xi^{2}
\end{pmatrix},\quad k\in\Z,\quad\xi\in\R.
\]
 We interpret $\left\langle \mathcal{Q}_{\varepsilon}f,f\right\rangle $
as a quadratic form in the variables $\widehat{f}_{k}^{1}$ and $\widehat{f}_{k}^{2}$.
In this way, to prove Inequality (\ref{controldissipationQ}) we show
that there exist $\varepsilon_{d}$ and $c_{0}$ two positive constants
such that for all $g\in H^{1}(\R,\R^{2})$: 
\begin{equation}
\forall\, k\in\Z,\quad\forall\,\xi\in\R,\qquad\widehat{g}(\xi){Q}_{\varepsilon}(k,\xi)\widehat{g}(\xi)\geq c_{0}\left(\varepsilon^{-2}\, k^{2}+\xi^{2}\right)\ \left(\widehat{g}_{1}^{2}(\xi)+\widehat{g}_{2}^{2}(\xi)\right)\,.\label{inegQkxi}
\end{equation}
 At fixed $(k,\xi)\in\Z\times\R$, $Q_{\varepsilon}(k,\xi)$ is a
diagonalizable matrix since real and symmetric. We compute its eigenvalues:
\begin{align*}
\mu_{\varepsilon}^{1}(k,\xi):=\frac{2\mu+\lambda{\varepsilon}}{2}((\varepsilon^{-1}k)^{2}+\xi^{2})+\frac{\lambda{\varepsilon}}{2}\sqrt{((\varepsilon^{-1}k)^{4}+\xi^{4})+(\varepsilon(\varepsilon^{-1}k)\xi)^{2}+\varepsilon^{-2}\left(\varepsilon^{-1}k\xi\right)^{2}},\\
\mu_{\varepsilon}^{2}(k,\xi):=\frac{2\mu+\lambda{\varepsilon}}{2}((\varepsilon^{-1}k)^{2}+\xi^{2})-\frac{\lambda{\varepsilon}}{2}\sqrt{((\varepsilon^{-1}k)^{4}+\xi^{4})+(\varepsilon(\varepsilon^{-1}k)\xi)^{2}+\varepsilon^{-2}\left(\varepsilon^{-1}k\xi\right)^{2}}\,.
\end{align*}
 In the end of the proof, we show that for any $(k,\xi)\in\Z\times\R$
we have 
\[
\mu_{\varepsilon}^{2}(k,\xi)\geq c_{0}\left(\varepsilon^{-2}\, k^{2}+\xi^{2}\right)\ \left(\widehat{g}_{1}^{2}(\xi)+\widehat{g}_{2}^{2}(\xi)\right).
\]
 Since $\mu_{\varepsilon}^{1}\geq\mu_{\varepsilon}^{2}$, we clearly
obtain Inequality \eqref{inegQkxi}.\\

$\circ$ We define the function $\mu_{\varepsilon}:\R^{2}\longrightarrow\R$
by: 
\[
\mu_{\varepsilon}(x,y):=\frac{2\mu+\lambda{\varepsilon}}{2}(x^{2}+y^{2})-\frac{\lambda{\varepsilon}}{2}\sqrt{(x^{4}+y^{4})+(\varepsilon\, x\, y)^{2}+\varepsilon^{-2}\left(x\, y\right)^{2}}.
\]
 Then there exist $\varepsilon_{d}$ and $c_{0}$ (independent of
$\varepsilon$) such that 
\begin{equation}
\forall(x,y)\in\R^{2},\qquad\mu_{\varepsilon}(x,y)\geq c_{0}(x^{2}+y^{2}).\label{eqmachin}
\end{equation}

-First of all the function $\mu_{\varepsilon}$ is homogeneous of
order 2 in the sense that it satisfies for all $\alpha\in\R$: 
\[
\forall\,(x,y)\in\R\times\R,\qquad\mu_{\varepsilon}(\alpha\, x,\alpha\, y)=\alpha^{2}\mu_{\varepsilon}(x,y)\,.
\]
 Thus, we prove \eqref{eqmachin} on the restricted set $(x,y)\in\Sp^{1}$
the sphere of center $0$ and radius $1$.

-We expand $\mu_{\varepsilon}$ as $\varepsilon$ goes to $0^{+}$:
\[
\mu_{\varepsilon}(x,y)=(x^{2}+y^{2})-\frac{\lambda_{}}{2}|x||y|+O(\varepsilon),
\]
 where $O(\varepsilon)$ is uniform in $(x,y)\in\Sp^{1}$. Let us
recall that we have: 
\[
\forall(x,y)\in\R^{2},\qquad(x^{2}+y^{2})-c|x||y|\geq0\qquad\Longleftrightarrow\qquad c<2.
\]
 That is the case if and only if $\lambda_{}<4\mu$, \textit{i.e.}
assumption \eqref{eqmueps} is satisfied. We get an uniform bound
of $\mu_{\varepsilon}$ (in $\varepsilon$). Finally, there exist
$\varepsilon_{d}$ and $c_{0}$ two positive constants such that Inequality
(\ref{eqmachin}) is satisfied.

$\circ$ Plugging $x=\varepsilon^{-1}k$ and $y=\xi$ in Inequality
(\ref{eqmachin}), we obtain for all $\varepsilon\in]0,\varepsilon_{d}]$:
\[
\forall\,(k,\xi)\in\Z\times\R\,,\qquad\mu_{\varepsilon}^{2}(k,\xi)\geq c_{0}\,((\varepsilon^{-1}k)^{2}+\xi^{2}).
\]
 We deduce that 
\[
\widehat{f_{k}}(\xi)Q_{\varepsilon}(k,\xi)\widehat{f_{k}}(\xi)\geq c_{0}\,((\varepsilon^{-1}k)^{2}+\xi^{2})\|\widehat{f_{k}}(\xi)\|^{2}.
\]
 Finally, summing over $k\in\Z$ and integrating with respect to $\xi$,
we have: 
\begin{align*}
-\left\langle \mathcal{Q}_{\varepsilon}f,f\right\rangle \geq c_{0}\sum_{k\in\Z}\int_{\R}((\varepsilon^{-1}k)^{2}+\xi^{2})\|\widehat{f_{k}}(\xi)\|^{2}d\xi=c_{0}\,\Phi_{\varepsilon}(\nabla,f).
\end{align*}
 \end{proof}

\bigskip{}

\begin{remark} The singular change of unknowns desingularizes the
hyperbolic part of system \eqref{eqepsilon} to a cost on the parabolic
part $\mathcal{P}_{\varepsilon}$. $\mathcal{Q}_{\varepsilon}$ is
dissipative enough assuming the $\Delta$-part of the dissipation
is strong enough with respect to the $\nabla\, div$-part ($\lambda<4\mu$).
\end{remark} \begin{remark} The Inequality \eqref{controldissipationQ}
can be proved for a more general family of dissipation. We can replace
$\varepsilon\lambda$ in the dissipation $\mathcal{Q}_{\varepsilon}$
by some coefficient $\lambda_{\varepsilon}$ going to $0$ when $\varepsilon$
approaches $0$. Then replacing assumption \eqref{eqmueps} by: 
\[
\limsup_{\varepsilon\rightarrow0^{+}}\frac{\lambda_{\varepsilon}}{\varepsilon}<4\mu.
\]
 Inequality \eqref{controldissipationQ} still holds. \end{remark}
From Equation \eqref{eqenergyplouche} and the preceding lemma we
deduce that for all $\varepsilon\in]0,\varepsilon_{d}]$ and for all
time $t\in[0,T_{\varepsilon}^{*}]$: 
\begin{multline}\label{eqenergyplouche2}
\frac{1}{2}\partial_{t}\left\Vert \partial^{\alpha}\widetilde{v}_{\varepsilon}^{R}\right\Vert _{L^{2}}^{2}+c_{0}\Phi_{\varepsilon}(\nabla,\widetilde{v}_{\varepsilon}^{R})\leq\left\langle \partial^{\alpha}S_{\varepsilon}^{R,N},\partial^{\alpha}\widetilde{v}_{\varepsilon}^{R}\right\rangle -\left\langle \partial^{\alpha}\,(\mathcal{A}\,\widetilde{v}_{\varepsilon}^{R}),\partial^{\alpha}\widetilde{v}_{\varepsilon}^{R}\right\rangle \\
-\left\langle \partial^{\alpha}\,(\mathcal{B}\widetilde{v}_{\varepsilon}^{R}),\partial^{\alpha}\widetilde{v}_{\varepsilon}^{R}\right\rangle -\left\langle \partial^{\alpha}\,(\mathcal{C}\widetilde{v}_{\varepsilon}^{R}),\partial^{\alpha}\widetilde{v}_{\varepsilon}^{R}\right\rangle -\left\langle \partial^{\alpha}\,(\mathcal{H}\widetilde{v}_{\varepsilon}^{R}),\partial^{\alpha}\widetilde{v}_{\varepsilon}^{R}\right\rangle .
\end{multline}

\subsubsection{Step two : estimates over the singular terms}

\label{substeptwo}

We now try to absorb the singular terms in \eqref{eqenergyplouche},
thanks to $\Phi_{\varepsilon}$. Singular terms are of two types: 
\begin{description}
\item [{\quad{}}] -the operator only singular with respect to the regularity,
$S_{\varepsilon}^{R,N}$, 
\item [{\quad{}}] -the operators at least singular with respect of $\varepsilon$,
$\mathcal{H}$ and $\mathcal{C}$. 
\end{description}

\paragraph{Contribution of $S_{\varepsilon}^{R,N}$.}

We start by estimating the term $S_{\varepsilon}^{R,N}$. It contains
nonlinear terms together with the coupling terms \eqref{eqcoupled}.
What can be underlined is that this term \eqref{eqcoupled} is nonlinear
(only) singular with respect to the number of derivatives acting on
the pressure. This is a problem since the dissipation only regularizes
the velocity. When possible, we pass those extra-derivatives onto
the velocity (thanks to an integration by parts). One can prove:

\begin{lemma}\label{controlG}{[}Control over $S_{\varepsilon}^{R,N}${]}
Select a multi-index $\alpha\in\N^{2}$ with length smaller than~$m+~3$
and a positive constant $C_{S}$. There exist $C_{S}^{1},\ C_{p}$
two positive constants such that for any~$\varepsilon\in~]0,1]$,
for any time $t\in~[0,T_{\varepsilon}^{*}]$, 
\begin{align*}
\left|\left\langle \partial^{\alpha}S_{\varepsilon}^{R,N},\partial^{\alpha}v_{\varepsilon}^{R}\right\rangle \right|(t,\cdot)\leq C_{S}\left\Vert \widetilde{v}_{\varepsilon}^{R}(t,\cdot)\right\Vert _{\overset{\circ}{H}{}^{|\alpha|+1}(\T\times\R)}^{2}+C_{S}^{1}\left\Vert \widetilde{v}_{\varepsilon}^{R}(t,\cdot)\right\Vert _{\overset{\circ}{H}{}^{|\alpha|}(\T\times\R)}^{2}+C_{p}\ \varepsilon^{2w_{m}}\,.
\end{align*}
 \end{lemma}

\begin{proof}[Proof of Lemma \ref{controlG}] Consider $\alpha\in\N^{2}$
satisfying $|\alpha|\leq m+3$. We decompose the source into 
\begin{multline*}
\left\langle \partial^{\alpha}S_{\varepsilon}^{R,N},\partial^{\alpha}v_{\varepsilon}^{R}\right\rangle =-\varepsilon^{N-R}\left\langle \partial^{\alpha}\left(\varepsilon^{-N}\mathcal{L}^{a}(\varepsilon,v_{\varepsilon}^{a})\right),\partial^{\alpha}\widetilde{v}_{\varepsilon}^{R}\right\rangle \\
-\varepsilon^{R+M-2}\left\langle \partial^{\alpha}(\widetilde{v}_{\varepsilon}^{1R}\,\partial_{\theta}\widetilde{v}_{\varepsilon}^{R}+\varepsilon\,\widetilde{v}_{\varepsilon}^{2R}\,\partial_{y}\widetilde{v}_{\varepsilon}^{R}),\partial^{\alpha}\widetilde{v}_{\varepsilon}^{R}\right\rangle \\
-C\,2^{-1}\varepsilon^{2\nu-M-R-2}\ \left\langle ^{t}(\partial_{\theta},\varepsilon\,\partial_{y})\partial^{\alpha}\left(q_{\varepsilon}^{a}+\varepsilon^{R}\,\widetilde{q}_{\varepsilon}^{R}\right)^{2},\partial^{\alpha}\widetilde{v}_{\varepsilon}^{R}\right\rangle .
\end{multline*}
 We estimate the three above contributions.

$\circ$ \textit{The first term: contribution of the approximated
solution.} $v_{\varepsilon}^{a}$ is an approximated solution for
the operator $\mathcal{L}^{a}$. According to Proposition \ref{propconstructionvelocity},
it satisfies Inequality (\ref{eqapproxsolv}). That is to say, Inequality~\eqref{eqlah}
holds and: 
\begin{align}\label{eqsourcefine}
\left|\varepsilon^{N-R}\left\langle \partial^{\alpha}\left(\varepsilon^{-N}\mathcal{L}^{a}(\varepsilon,v_{\varepsilon}^{a})\right),\partial^{\alpha}\widetilde{v}_{\varepsilon}^{R}\right\rangle \right| & \leq\frac{1}{2}\left(\varepsilon^{2(N-R)}\left\Vert \partial^{\alpha}\left(\varepsilon^{-N}\mathcal{L}^{a}(\varepsilon,v_{\varepsilon}^{a})\right)\right\Vert _{L^{2}}^{2}+\left\Vert \partial^{\alpha}v_{\varepsilon}^{R}\right\Vert _{L^{2}}^{2}\right)\,,\nonumber \\
 & \leq\frac{1}{2}\left(C_{\mathcal{L}^{a}}\,\varepsilon^{2(N-R)}+\left\Vert v_{\varepsilon}^{R}\right\Vert _{\overset{\circ}{H}{}^{|\alpha|}}^{2}\right)\,.
\end{align}

$\circ$ \textit{The second term} is the contribution corresponding
to the non linear part \textquotedbl{}$u\cdot\nabla u$\textquotedbl{}
(in Navier-Stokes Equations). Select a positive constant~$c_{1}$.
Further, we choose it so that the contribution of~$\nabla\widetilde{v}_{\varepsilon}^{R}$
in $H^{m}$-norm is small with respect to $\Phi_{\varepsilon}$. 
\begin{align}
 & \varepsilon^{R+M-2}\left|\left\langle \partial^{\alpha}(\widetilde{v}_{\varepsilon}^{1R}\partial_{\theta}\widetilde{v}_{\varepsilon}^{R}+\varepsilon\,\widetilde{v}_{\varepsilon}^{2R}\,\partial_{y}\widetilde{v}_{\varepsilon}^{R}),\partial^{\alpha}\widetilde{v}_{\varepsilon}^{R}\right\rangle \right|\nonumber \\
 & \hspace{1.5cm}\leq\frac{1}{2}\left(c_{1}\varepsilon^{2(R+M-2)}\left\Vert \partial^{\alpha}\left(\widetilde{v}_{\varepsilon}^{1R}\partial_{\theta}\widetilde{v}_{\varepsilon}^{R}+\widetilde{v}_{\varepsilon}^{2R}\partial_{y}\widetilde{v}_{\varepsilon}^{R}\right)\right\Vert _{L^{2}}^{2}+\frac{1}{c_{1}}\left\Vert \partial^{\alpha}v_{\varepsilon}^{R}\right\Vert _{L^{2}}^{2}\right)\,,\nonumber \\
 & \hspace{1.5cm}\leq c_{1}\varepsilon^{2(R+M-2)}\left(\left\Vert \partial^{\alpha}\left(\widetilde{v}_{\varepsilon}^{1R}\partial_{\theta}\widetilde{v}_{\varepsilon}^{R}\right)\right\Vert _{L^{2}}^{2}+\left\Vert \partial^{\alpha}\left(\widetilde{v}_{\varepsilon}^{2R}\partial_{y}\widetilde{v}_{\varepsilon}^{R}\right)\right\Vert _{L^{2}}^{2}\right)+\frac{1}{2c_{1}}\left\Vert \partial^{\alpha}\widetilde{v}_{\varepsilon}^{1R}\right\Vert _{L^{2}}^{2}.
\end{align}
 We deal with the nonlinear term $\partial^{\alpha}(\widetilde{v}_{\varepsilon}^{1R}\partial_{\theta}\widetilde{v}_{\varepsilon}^{R})$
applying the Gagliardo-Niremberg's estimate together with \eqref{eqnormsup}.
There exists a positive constant $C_{g}$ which only depends on $m$
such that for all $\varepsilon\in]0,1]$ and for all time $t\in[0,T_{\varepsilon}^{*}]$:
\begin{align*}
\left\Vert \partial^{\alpha}\left(\widetilde{v}_{\varepsilon}^{1R}\,\partial_{\theta}\widetilde{v}_{\varepsilon}^{R}\right)\right\Vert _{L^{2}}^{2} & \leq C_{g}^{2}\left(\left\Vert \widetilde{v}_{\varepsilon}^{1R}\right\Vert _{L^{\infty}}\left\Vert \partial_{\theta}\widetilde{v}_{\varepsilon}^{R}\right\Vert _{\overset{\circ}{H}{}^{|\alpha|}}+\left\Vert \widetilde{v}_{\varepsilon}^{1R}\right\Vert _{\overset{\circ}{H}{}^{|\alpha|}}\left\Vert \partial_{\theta}\widetilde{v}_{\varepsilon}^{R}\right\Vert _{L^{\infty}}\right)^{2}\,,\\
 & \leq4\, C_{g}^{2}\left(\left\Vert v_{\varepsilon}^{R}\right\Vert _{\overset{\circ}{H}{}^{|\alpha|+1}}+\left\Vert \widetilde{v}_{\varepsilon}^{R}\right\Vert _{\overset{\circ}{H}{}^{|\alpha|}}\right)^{2}\leq8\, C_{g}^{2}\left(\left\Vert \widetilde{v}_{\varepsilon}^{R}\right\Vert _{\overset{\circ}{H}{}^{|\alpha|+1}}^{2}+\left\Vert \widetilde{v}_{\varepsilon}^{R}\right\Vert _{\overset{\circ}{H}{}^{|\alpha|}}^{2}\right).
\end{align*}
 The same holds for $\widetilde{v}_{\varepsilon}^{2R}\,\partial_{y}\widetilde{v}_{\varepsilon}^{R}$.
Finally, we get that for all $\varepsilon\in]0,1]$ and for all time
$t\in[0,T_{\varepsilon}^{*}]$ 
\begin{multline}\label{eqvelocityfine}
\left|\varepsilon^{R+M-2}\left\langle \partial^{\alpha}(\widetilde{v}_{\varepsilon}^{1R}\,\partial_{\theta}\widetilde{v}_{\varepsilon}^{R}+\varepsilon\,\widetilde{v}_{\varepsilon}^{2R}\,\partial_{y}\widetilde{v}_{\varepsilon}^{R}),\partial^{\alpha}\widetilde{v}_{\varepsilon}^{R}\right\rangle \right|\\
\leq16\, c_{1}\, C_{g}^{2}\left\Vert \widetilde{v}_{\varepsilon}^{R}\right\Vert _{\overset{\circ}{H}{}^{|\alpha|+1}}^{2}+(16\, c_{1}\, C_{g}^{2}+\frac{1}{2c_{1}})\left\Vert \widetilde{v}_{\varepsilon}^{R}\right\Vert _{\overset{\circ}{H}{}^{|\alpha|}}^{2}.
\end{multline}

$\circ$ \textit{The last term.} This term has to be studied with
care. As mentioned, two obstacles have to be overcome: 
\begin{description}
\item [{\quad{}}] -It is singular with respect to the number of derivatives
acting on $q$. However the term can formally be written under the
form: 
\[
\frac{1}{2}\left\langle \nabla(q^{2}),v\right\rangle .
\]
 A derivative can be passed onto the velocity with an integration
by parts. 
\item [{\quad{}}] -The pressure is only estimated in anisotropic Sobolev
spaces $H_{(1,\varepsilon)}^{m}$. We pass to the anisotropic Sobolev
spaces $H_{(1,\varepsilon)}^{m}$ thanks to Inequality \eqref{eqrefnorm}
at a cost on the precision $w_{m}$ (see Equation~\eqref{eqws}). 
\end{description}
First we integrate by parts: 
\begin{multline*}
\left|-C\,2^{-1}\varepsilon^{2\nu-M-R-2}\ \left\langle ^{t}(\partial_{\theta},\varepsilon\,\partial_{y})\partial^{\alpha}\left(q_{\varepsilon}^{a}+\varepsilon^{R}\,\widetilde{q}_{\varepsilon}^{R}\right)^{2},\partial^{\alpha}\widetilde{v}_{\varepsilon}^{R}\right\rangle \right|=\\
\left|C\,2^{-1}\varepsilon^{2\nu-M-R-2}\ \left\langle \partial^{\alpha}\left(q_{\varepsilon}^{a}+\varepsilon^{R}\,\widetilde{q}_{\varepsilon}^{R}\right)^{2},\partial^{\alpha}\ {}^{t}(\partial_{\theta},\varepsilon\,\partial_{y})\widetilde{v}_{\varepsilon}^{R}\right\rangle \right|\,.
\end{multline*}
 We select a positive constant $c_{2}$, then 
\begin{multline}\label{eqpressurebiz}
\left|-C\,2^{-1}\varepsilon^{2\nu-M-R-2}\ \left\langle ^{t}(\partial_{\theta},\varepsilon\,\partial_{y})\partial^{\alpha}\left(q_{\varepsilon}^{a}+\varepsilon^{R}\,\widetilde{q}_{\varepsilon}^{R}\right)^{2},\partial^{\alpha}\widetilde{v}_{\varepsilon}^{R}\right\rangle \right|\\
\leq\frac{C^{2}\varepsilon^{2(2\nu-R-M-2)}}{8c_{2}}\left\Vert \partial^{\alpha}\left(q_{\varepsilon}^{a}+\varepsilon^{R}\widetilde{q}_{\varepsilon}^{R}\right)^{2}\right\Vert _{L^{2}}^{2}+\frac{c_{2}}{2}\left\Vert \partial^{\alpha}\ {}^{t}(\partial_{\theta},\varepsilon\,\partial_{y})v_{\varepsilon}^{R}\right\Vert _{L^{2}}^{2}.
\end{multline}
 We apply the Gagliardo-Niremberg inequality together with the equivalence
of norms (\ref{eqrefnorm}), 
\begin{align*}
\left\Vert \partial^{\alpha}\left(q_{\varepsilon}^{a}+\varepsilon^{R}\widetilde{q}_{\varepsilon}^{R}\right)^{2}\right\Vert _{L^{2}} & \leq2C_{g}\left\Vert q_{\varepsilon}^{a}+\varepsilon^{R}\widetilde{q}_{\varepsilon}^{R}\right\Vert _{L^{\infty}}\left\Vert q_{\varepsilon}^{a}+\varepsilon^{R}\widetilde{q}_{\varepsilon}^{R}\right\Vert _{\overset{\circ}{H}{}^{|\alpha|}}\,,\\
 & \leq2C_{g}\varepsilon^{-|\alpha|}\left\Vert q_{\varepsilon}^{a}+\varepsilon^{R}\widetilde{q}_{\varepsilon}^{R}\right\Vert _{L^{\infty}}\left\Vert q_{\varepsilon}^{a}+\varepsilon^{R}\widetilde{q}_{\varepsilon}^{R}\right\Vert _{\overset{\circ}{H}_{(1,\varepsilon)}^{|\alpha|}}.
\end{align*}
 With regards to the construction of the characterized time $T_{\varepsilon}^{*}$,
\eqref{eqnormsup} and \eqref{eqhupvq} holds. It results in: 
\[
\forall\,\varepsilon\in]0,1],\qquad\forall\, t\in[0,T_{\varepsilon}^{*}],\qquad\left\Vert \sqrt{\varepsilon}\,\widetilde{q}_{\varepsilon}^{R}(t,\cdot)\right\Vert _{L^{\infty}(\T\times\R)}\leq2,\quad\left\Vert \sqrt{\varepsilon}\, q_{\varepsilon}^{a}(t,\cdot)\right\Vert _{L^{\infty}(\T\times\R)}\leq C_{a}\,.
\]
 We obtain 
\begin{equation}
\left\Vert \partial^{\alpha}\left(q_{\varepsilon}^{a}+\varepsilon^{R}\widetilde{q}_{\varepsilon}^{R}\right)^{2}\right\Vert _{L^{2}}\leq\,2\, C_{g}\,\varepsilon^{-1/2-|\alpha|}(C_{a}+2)^{2}.\label{ineqpressurereste}
\end{equation}
 Finally plugging Inequality \eqref{ineqpressurereste} in \eqref{eqpressurebiz},
we deduce: 
\begin{multline}\label{eqpressurefine}
\left|-C\,2^{-1}\varepsilon^{2\nu-M-R-2}\ \left\langle ^{t}(\partial_{\theta},\varepsilon\,\partial_{y})\partial^{\alpha}\left(q_{\varepsilon}^{a}+\varepsilon^{R}\,\widetilde{q}_{\varepsilon}^{R}\right)^{2},\partial^{\alpha}\widetilde{v}_{\varepsilon}^{R}\right\rangle \right|\\
\leq\frac{C^{2}\, C_{g}^{2}\,\varepsilon^{2(2\nu-R-M-5/2-|\alpha|)}}{2c_{2}}(C_{a}+2)^{4}+\frac{c_{2}}{2}\left\Vert v_{\varepsilon}^{R}\right\Vert _{\overset{\circ}{H}{}^{|\alpha|+1}}^{2}.
\end{multline}

$\circ$ \textit{To finish}, we put estimates (\ref{eqsourcefine}),
(\ref{eqvelocityfine}) and (\ref{eqpressurefine}) together. Let
$c_{1}$ and $c_{2}$ be two positive constants for all $\varepsilon\in]0,1]$
and for all time $t\in[0,T_{\varepsilon}^{*}]$: 
\begin{align*}
\left|\left\langle \partial^{\alpha}S_{\varepsilon}^{R,N},\partial^{\alpha}v_{\varepsilon}^{R}\right\rangle \right| & \leq\left(\frac{c_{2}}{2}+16\, c_{1}\, C_{g}^{2}\right)\left\Vert v_{\varepsilon}^{R}\right\Vert _{\overset{\circ}{H}^{|\alpha|+1}}^{2}+\left(\frac{1}{2}+16\, c_{1}\, C_{g}^{2}+\frac{1}{2c_{1}}\right)\left\Vert v_{\varepsilon}^{R}\right\Vert _{\overset{\circ}{H}{}^{|\alpha|}}^{2}\\
 & \hspace{3,1cm}+\left(C_{\mathcal{L}^{a}}\,\varepsilon^{2(N-R)}+\frac{C^{2}\, C_{g}^{2}}{2c_{2}}(C_{a}+2)^{4}\,\varepsilon^{2(2\mathcal{V}-R-M-5/2-|\alpha|)}\right)\,,\\
 & \leq\left(\frac{c_{2}}{2}+16\, c_{1}\, C_{g}^{2}\right)\left\Vert v_{\varepsilon}^{R}\right\Vert _{\overset{\circ}{H}{}^{|\alpha|+1}}^{2}+\left(\frac{1}{2}+16\, c_{1}\, C_{g}^{2}+\frac{1}{2c_{1}}\right)\left\Vert v_{\varepsilon}^{R}\right\Vert _{\overset{\circ}{H}{}^{|\alpha|}}^{2}\\
 & \hspace{6,9cm}+\left(C_{\mathcal{L}^{a}}+\frac{C^{2}\, C_{g}^{2}}{2c_{2}}(C_{a}+2)^{4}\right)\varepsilon^{w_{m}}\,.
\end{align*}
 Select $C_{S}$ a positive constant. We choose $c_{1}$ and $c_{2}$
two positive constants such that $C_{S}:=\frac{c_{2}}{2}+16\, c_{1}\, C_{g}^{2}$.
Then, it requires $C_{S}^{1}=\frac{1}{2}+16\, c_{1}\, C_{g}^{2}+\frac{1}{2c_{1}}$
and $C_{p}=C_{\mathcal{L}^{a}}+\frac{C\, C_{g}^{2}}{2\, c_{2}}(C_{a}+2)^{4}$.
\end{proof}

\paragraph{Contribution of operators $\mathcal{H}$ and $\mathcal{C}$.}

Presently, we study the singular operators $\mathcal{H}$ and $\mathcal{C}$.
We decompose their action into: 
\begin{align*}
\left\langle \partial^{\alpha}\left(\mathcal{H}\widetilde{v}_{\varepsilon}^{R}\right),\partial^{\alpha}\widetilde{v}_{\varepsilon}^{R}\right\rangle  & =\left\langle \mathcal{H}\,(\partial^{\alpha}\widetilde{v}_{\varepsilon}^{R}),\widetilde{v}_{\varepsilon}^{R}\right\rangle +\left\langle [\partial^{\alpha},\mathcal{H}]\widetilde{v}_{\varepsilon}^{R},\partial^{\alpha}\widetilde{v}_{\varepsilon}^{R}\right\rangle \,,\\
\left\langle \partial^{\alpha}\left(\mathcal{C}\widetilde{v}_{\varepsilon}^{R}\right),\partial^{\alpha}\widetilde{v}_{\varepsilon}^{R}\right\rangle  & =\left\langle \mathcal{C}\,(\partial^{\alpha}\widetilde{v}_{\varepsilon}^{R}),\widetilde{v}_{\varepsilon}^{R}\right\rangle +\left\langle [\partial^{\alpha},\mathcal{C}]\widetilde{v}_{\varepsilon}^{R},\partial^{\alpha}\widetilde{v}_{\varepsilon}^{R}\right\rangle \,.
\end{align*}
 by making some commutators appear. Since the commutator of two operators
respectively of order $m$ and $n$ is of order $m+n-1$, the contribution
of the commutators are less singular than expected: \begin{lemma}\label{lemH}Select
a multi-index $\alpha\in\N^{2}$ with length satisfying $1\leq|\alpha|\leq m+3$.
Select $\mathcal{V}\in\{\mathcal{H},\mathcal{C}\}$.There exist $C_{\mathcal{V}}^{1}$,
$C_{\mathcal{V}}^{2}$ and $C_{\mathcal{V}}^{3}$ three positive constants
such that for all $\varepsilon\in]0,1]$, for all time $t\in[0,T_{\varepsilon}^{*}]$,
\begin{multline*}
\left|\left\langle \left[\partial^{\alpha},\mathcal{V}\right]\widetilde{v}_{\varepsilon}^{R},\partial^{\alpha}\widetilde{v}_{\varepsilon}^{R}\right\rangle \right|(t)\leq C_{\mathcal{V}}^{1}\left\Vert \widetilde{v}_{\varepsilon}^{R}(t,\cdot)\right\Vert _{\overset{\circ}{H}{}^{|\alpha|}(\T\times\R)}^{2}+C_{\mathcal{V}}^{2}\left\Vert \widetilde{v}_{\varepsilon}^{R}(t,\cdot)\right\Vert _{H^{|\alpha|-1}(\T\times\R)}^{2}\\
+C_{\mathcal{V}}^{3}\left\Vert \varepsilon^{-(1-\delta_{\alpha_{1},0})}\partial^{\alpha}\widetilde{v}_{\varepsilon}^{R}(t,\cdot)\right\Vert _{L^{2}(\T\times\R)}^{2}.
\end{multline*}
 \end{lemma} The proof is rather easy. We do not write it. Then,
since $\left\langle \mathcal{H}\,\partial^{\alpha}\widetilde{v}_{\varepsilon}^{R},\widetilde{v}_{\varepsilon}^{R}\right\rangle =0$,
there remains to deal with $\left\langle \mathcal{C}(\partial^{\alpha}\widetilde{v}_{\varepsilon}^{R}),\partial^{\alpha}\widetilde{v}_{\varepsilon}^{R}\right\rangle $.

\begin{lemma}{[}Absorption of $\mathcal{C}${]} \label{lemestC}
Select a positive constant $c_{1}$ and a multi-index $\alpha\in\N^{2}$
of length less than~$m+3$. Then for all $\varepsilon\in]0,1]$ for
all time $t\in[0,T_{\varepsilon}^{*}]$: 
\[
\langle\mathcal{C}\partial^{\alpha}\widetilde{v}_{\varepsilon}^{R}(t,\cdot),\partial^{\alpha}\widetilde{v}_{\varepsilon}^{R}(t,\cdot)\rangle\leq\frac{\left\Vert h\right\Vert _{L^{\infty}(\T)}}{2c_{1}}\|\partial^{\alpha}\widetilde{v}_{\varepsilon}^{R}(t,\cdot)\|_{L^{2}(\T\times\R)}^{2}+\frac{c_{1}}{2}\left\Vert h\right\Vert _{L^{\infty}}\left\Vert \varepsilon^{-1}\partial_{\theta}\partial^{\alpha}\widetilde{v}_{\varepsilon}^{R}\right\Vert _{L^{2}}^{2}.
\]
 \end{lemma} \begin{proof}[Proof of Lemma \ref{lemestC}] We
prove the result for $\alpha=(0,0)$. Replacing $\widetilde{v}_{\varepsilon}^{R}$
by $\partial^{\alpha}\widetilde{v}_{\varepsilon}^{R}$ in the above
estimates proves the result. We select $c_{1}$ a positive constant.
The idea is to integrate by parts with respect to the variable $\theta$
to make the weighted derivative of the velocity $\varepsilon^{-1}\partial_{\theta}\widetilde{v}_{\varepsilon}^{R}$
appear. 
\begin{align*}
|\langle\mathcal{C}\widetilde{v}_{\varepsilon}^{R},\widetilde{v}_{\varepsilon}^{R}\rangle|= & \left|\int h\,\varepsilon^{-1}\partial_{\theta}\widetilde{v}_{\varepsilon}^{1R}\,\widetilde{v}_{\varepsilon}^{2R}d\theta\, dy+\int h\,\widetilde{v}_{\varepsilon}^{1R}\,\varepsilon^{-1}\partial_{\theta}\widetilde{v}_{\varepsilon}^{2R}d\theta\, dy\right|\,,\\
\leq & \frac{\left\Vert h\right\Vert _{L^{\infty}}}{2}\left(\frac{1}{c_{1}}\left\Vert \widetilde{v}_{\varepsilon}^{R}\right\Vert _{L^{2}}^{2}+{c_{1}}{}\left\Vert \varepsilon^{-1}\partial_{\theta}\widetilde{v}_{\varepsilon}^{R}\right\Vert _{L^{2}}^{2}\right).
\end{align*}
 \end{proof}

\paragraph{Absorption of singular terms.}

Select $c_{1}$ and $C_{s}$ two positive constants to be chosen (small)
later. We get a bound for the right-hand-side of \eqref{eqenergyplouche2}
applying Lemmas \ref{controlG}-$\ldots$-\ref{lemestC}. We obtain
that for all $\varepsilon\in]0,\varepsilon_{d}]$ and for all time
$t\in[0,T_{\varepsilon}^{*}]$: 
\begin{multline}
\frac{1}{2}\partial_{t}\left\Vert \partial^{\alpha}\widetilde{v}_{\varepsilon}^{R}(t,\cdot)\right\Vert _{L^{2}}^{2}+\fbox{\textcolor{marin}{\ensuremath{c_{0}\,\Phi_{\varepsilon}\left(\nabla,\partial^{\alpha}\widetilde{v}_{\varepsilon}^{R}\right)-\frac{c_{1}}{2}\left\Vert h\right\Vert _{L^{\infty}}\left\Vert \varepsilon^{-1}\partial_{\theta}\partial^{\alpha}\widetilde{v}_{\varepsilon}^{R}\right\Vert _{L^{2}}^{2}-C_{S}\left\Vert \widetilde{v}_{\varepsilon}^{R}(t,\cdot)\right\Vert _{\overset{\circ}{H}{}^{J+1}(\T\times\R)}^{2}}}}\\
\leq\left(\frac{\left\Vert h\right\Vert _{L^{\infty}}}{2c_{1}}+C_{S}^{1}+C_{\mathcal{C}}^{1}+C_{\mathcal{H}}^{1}\right)\left\Vert \widetilde{v}_{\varepsilon}^{R}(t,\cdot)\right\Vert _{\overset{\circ}{H}{}^{J}(\T\times\R)}^{2}+C_{p}\,\varepsilon^{2w_{m}}\\
+(C_{\mathcal{C}}^{2}+C_{\mathcal{H}}^{2})\left\Vert \widetilde{v}_{\varepsilon}^{R}(t,\cdot)\right\Vert _{H^{J-1}(\T\times\R)}^{2}+(C_{\mathcal{C}}^{3}+C_{\mathcal{H}}^{3})\left\Vert \varepsilon^{-(1-\delta_{\alpha_{1},0})}\partial^{\alpha}\widetilde{v}_{\varepsilon}^{R}(t,\cdot)\right\Vert _{L^{2}(\T\times\R)}^{2}\\
+\left|\left\langle \partial^{\alpha}\left(\mathcal{A}\widetilde{v}_{\varepsilon}^{R}\right),\partial^{\alpha}\widetilde{v}_{\varepsilon}^{R}\right\rangle \right|+\left|\left\langle \partial^{\alpha}\left(\mathcal{B}\widetilde{v}_{\varepsilon}^{R}\right),\partial^{\alpha}\widetilde{v}_{\varepsilon}^{R}\right\rangle \right|.
\end{multline}
 At this stage, the quadratic form $\Phi_{\varepsilon}$ can not absorb
the norm $\left\Vert \widetilde{v}_{\varepsilon}^{R}(t,\cdot)\right\Vert _{\overset{\circ}{H}{}^{J+1}(\T\times\R)}^{2}$.
However summing over all $\alpha\in\N^{2}$ of length $J$ the quadratic
form $\sum_{|\alpha|=J}\Phi_{\varepsilon}(\nabla,\partial^{\alpha}\widetilde{v}_{\varepsilon}^{R})$
does.

Choosing $c_{1}$ and $C_{S}^{1}$ such that the quantity $\widetilde{c}_{0}:=c_{0}-\frac{c_{1}}{2}\left\Vert h\right\Vert _{L^{\infty}}-(m+4)C_{S}^{1}$
is positive, we have: 
\begin{align*}
 & \sum_{|\alpha|=J}\left(c_{0}\,\Phi_{\varepsilon}\left(\nabla,\partial^{\alpha}\widetilde{v}_{\varepsilon}^{R}\right)-\frac{c_{1}}{2}\left\Vert h\right\Vert _{L^{\infty}}\left\Vert \varepsilon^{-1}\partial_{\theta}\partial^{\alpha}\widetilde{v}_{\varepsilon}^{R}\right\Vert _{L^{2}}^{2}-C_{S}\left\Vert \widetilde{v}_{\varepsilon}^{R}(t,\cdot)\right\Vert _{\overset{\circ}{H}{}^{J+1}(\T\times\R)}^{2}\right)\\
 & \hspace{1cm}\geq c_{0}\sum_{|\alpha|=J}\Phi_{\varepsilon}(\nabla,\partial^{\alpha}\widetilde{v}_{\varepsilon}^{R})-\frac{c_{1}}{2}\left\Vert h\right\Vert _{L^{\infty}}\sum_{|\alpha|=J}\left\Vert \varepsilon^{-1}\partial_{\theta}\partial^{\alpha}\widetilde{v}_{\varepsilon}^{R}\right\Vert _{L^{2}}^{2}-(J+1)C_{S}\left\Vert \widetilde{v}_{\varepsilon}^{R}(t,\cdot)\right\Vert _{\overset{\circ}{H}{}^{J+1}(\T\times\R)}^{2}\\
 & \hspace{1cm}\geq\left(c_{0}-\frac{c_{1}}{2}\left\Vert h\right\Vert _{L^{\infty}}-(m+4)C_{S}\right)\sum_{|\alpha|=J}\Phi_{\varepsilon}(\nabla,\partial^{\alpha}\widetilde{v}_{\varepsilon}^{R})=\widetilde{c}_{0}\sum_{|\alpha|=J}\Phi_{\varepsilon}(\nabla,\partial^{\alpha}\widetilde{v}_{\varepsilon}^{R}).
\end{align*}
 Hence, we obtain for all $\varepsilon\in]0,\varepsilon_{d}]$ and
for all time $t\in[0,T_{\varepsilon}^{*}]$: 
\begin{multline}\label{eqenergyplouche3}
\frac{1}{2}\partial_{t}\left\Vert \widetilde{v}_{\varepsilon}^{R}(t,\cdot)\right\Vert _{\overset{\circ}{H}{}^{J}(\T\times\R)}^{2}+\textcolor{marin}{ \widetilde{c}_{0}\, \sum_{|\alpha|=J} \Phi_{\varepsilon}\left(\nabla,\partial^{\alpha}\widetilde{v}_{\varepsilon}^{R}\right)}\\
\leq(J+1)\left(\frac{\left\Vert h\right\Vert _{L^{\infty}}}{2c_{1}}+C_{S}^{1}+C_{\mathcal{C}}^{1}+C_{\mathcal{H}}^{1}\right)\left\Vert \widetilde{v}_{\varepsilon}^{R}(t,\cdot)\right\Vert _{\overset{\circ}{H}{}^{J}(\T\times\R)}^{2}+C_{p}(J+1)\,\varepsilon^{2w_{m}}\\
+(J+1)(C_{\mathcal{C}}^{2}+C_{\mathcal{H}}^{2})\left\Vert \widetilde{v}_{\varepsilon}^{R}(t,\cdot)\right\Vert _{H^{J-1}(\T\times\R)}^{2}+(C_{\mathcal{C}}^{3}+C_{\mathcal{H}}^{3})\sum_{|\alpha|=J}\left\Vert \varepsilon^{-(1-\delta_{\alpha_{1},0})}\partial^{\alpha}\widetilde{v}_{\varepsilon}^{R}(t,\cdot)\right\Vert _{L^{2}(\T\times\R)}^{2}\\
+\sum_{|\alpha|=J}\left|\left\langle \partial^{\alpha}\left(\mathcal{A}\widetilde{v}_{\varepsilon}^{R}\right),\partial^{\alpha}\widetilde{v}_{\varepsilon}^{R}\right\rangle \right|+\sum_{|\alpha|=J}\left|\left\langle \partial^{\alpha}\left(\mathcal{B}\widetilde{v}_{\varepsilon}^{R}\right),\partial^{\alpha}\widetilde{v}_{\varepsilon}^{R}\right\rangle \right|.
\end{multline}
 \begin{remark} The term $\sum_{|\alpha|=J}\left\Vert \varepsilon^{-(1-\delta_{\alpha_{1},0})}\partial^{\alpha}\widetilde{v}_{\varepsilon}^{R}(t,\cdot)\right\Vert _{L^{2}(\T\times\R)}^{2}$
seems to be singular with respect to $\varepsilon$. With regards
to the Inequality \eqref{eqpropestu2}, we will interpret it as a
non-degenerate source term in the proof of Proposition \eqref{propestu}.
\end{remark}

\subsubsection{Last step: estimates over non-singular terms}

\label{sublaststep}

The two last terms (contribution of operator $\mathcal{A}$ and operator
$\mathcal{B}$) can be easily estimated. They satisfy:

\begin{lemma}\label{lemA}Select a multi-index $\alpha\in\N^{2}$
of length smaller than $m+3$. Select $\mathcal{V}\in\{\mathcal{A},\mathcal{B}\}$.
There exist $C_{\mathcal{V}}^{1}$ and $C_{\mathcal{V}}^{2}$ two
positive constants such that for all $\varepsilon\in]0,1]$ and for
all time $t\in[0,T_{\varepsilon}^{*}]$: 
\begin{equation}
\left|\left\langle \partial^{\alpha}\left(\mathcal{V}\widetilde{v}_{\varepsilon}^{R}\right),\partial^{\alpha}\widetilde{v}_{\varepsilon}^{R}\right\rangle \right|(t)\leq C_{\mathcal{V}}^{1}\left\Vert \widetilde{v}_{\varepsilon}^{R}(t,\cdot)\right\Vert _{\overset{\circ}{H}{}^{|\alpha|}(\T\times\R)}^{2}+C_{\mathcal{V}}^{2}\left\Vert \widetilde{v}_{\varepsilon}^{R}(t,\cdot)\right\Vert _{H^{|\alpha|-1}(\T\times\R)}^{2}.\label{eqestimvapprox}
\end{equation}
 \end{lemma} The proof is very classical. We do not write it. Applying
Lemma \ref{lemA}, we obtain from Inequality~\eqref{eqenergyplouche3}
that for all $\varepsilon\in]0,\varepsilon_{d}]$ and for all time
$t\in[0,T_{\varepsilon}^{*}]$: 
\begin{multline}
\frac{1}{2}\partial_{t}\left\Vert \widetilde{v}_{\varepsilon}^{R}(t,\cdot)\right\Vert _{\overset{\circ}{H}{}^{J}(\T\times\R)}^{2}+\widetilde{c}_{0}\,\sum_{|\alpha|=J}\Phi_{\varepsilon}\left(\nabla,\partial^{\alpha}\widetilde{v}_{\varepsilon}^{R}\right)\leq C_{J}^{1}\left\Vert \widetilde{v}_{\varepsilon}^{R}(t,\cdot)\right\Vert _{\overset{\circ}{H}{}^{J}(\T\times\R)}^{2}+C_{p}(J+1)\,\varepsilon^{2w_{m}}\\
+C_{J}^{2}\left\Vert \widetilde{v}_{\varepsilon}^{R}(t,\cdot)\right\Vert _{H^{J-1}(\T\times\R)}^{2}+C_{J}^{3}\sum_{|\alpha|=J}\left\Vert \varepsilon^{-(1-\delta_{\alpha_{1},0})}\partial^{\alpha}\widetilde{v}_{\varepsilon}^{R}(t,\cdot)\right\Vert _{L^{2}(\T\times\R)}^{2},
\end{multline}
 with $C_{J}^{1}:=(J+1)\left(\frac{\left\Vert h\right\Vert _{L^{\infty}}}{2c_{1}}+C_{S}^{1}+C_{\mathcal{A}}^{1}+C_{\mathcal{B}}^{1}+C_{\mathcal{C}}^{1}+C_{\mathcal{H}}^{1}\right)$,
$C_{J}^{2}:=(J+1)(C_{\mathcal{A}}^{2}+C_{\mathcal{B}}^{2}+C_{\mathcal{C}}^{2}+C_{\mathcal{H}}^{2})$
and $C_{J}^{3}:=(C_{\mathcal{C}}^{3}+C_{\mathcal{H}}^{3})$.\\

That achieves the proof of Lemma \ref{lemall}.

\vspace{-0.5cm}
 \hfill{}$\square$\\

\subsubsection{Proof of Proposition \ref{propestu}}

\label{subsectionpropestu} This subsection is dedicated to the proof
of Proposition \ref{propestu}. We prove by induction that property~$\mathcal{P}(J)$,
defined at the level of (\ref{recurezrnl}), is satisfied for $J\in\llbracket0,m+3\rrbracket$.
We proceed in two steps. First we obtain an estimate for the velocity
(Inequality \eqref{eqpropestu1}). Then, we take advantage of the
quadratic form $\sum_{|\alpha|=J}\Phi_{\varepsilon}\left(\nabla,\partial^{\alpha}\widetilde{v}_{\varepsilon}^{R}\right)$
to obtain the regularization Inequality \eqref{eqpropestu2}.\\

Since the proof is exactly the same for the case $J=0$ and the increment
of the induction, we do not write the details to prove that property
$\mathcal{P}(0)$ is true.

\bigskip{}

\paragraph{We assume that $\mathcal{P}(J)$ is true for $J\in\llbracket0,m+2\rrbracket$.} We apply Lemma~$\ref{lemall}$
in the case $J+1\ (\geq1)$. There exist $\varepsilon_{d}$ and~$c_{1}$
two positive constants, there exist $C_{J+1}^{1}$, $C_{J+1}^{2}$,
$C_{J+1}^{3}$ and $C_{p}$ four positive constants such that for
all $\varepsilon\in]0,\varepsilon_{d}]$ and for all time $t\in[0,T_{\varepsilon}^{*}]$:
\begin{multline}\label{inegenergyJ+1}
\frac{1}{2}\partial_{t}\left\Vert \widetilde{v}_{\varepsilon}^{R}(t,\cdot)\right\Vert _{\overset{\circ}{H}{}^{J+1}(\T\times\R)}^{2}+c_{1}\textcolor{marin}{\sum_{|\alpha|=J+1}\Phi_{\varepsilon}(\nabla,\partial^{\alpha}\widetilde{v}_{\varepsilon}^{R})(t)}\\
\leq C_{J+1}^{1}\left\Vert \widetilde{v}_{\varepsilon}^{R}(t,\cdot)\right\Vert _{H^{J}(\T\times\R)}^{2}+C_{J+1}^{2}\left\Vert \widetilde{v}_{\varepsilon}^{R}(t,\cdot)\right\Vert _{\overset{\circ}{H}{}^{J+1}(\T\times\R)}^{2}\\
+C_{J+1}^{3}\sum_{|\alpha|=J+1}\left\Vert \varepsilon^{-(1-\delta_{\alpha_{1},0})}\partial^{\alpha}\widetilde{v}_{\varepsilon}^{R}(t,\cdot)\right\Vert _{L^{2}(\T\times\R)}^{2}+(J+2)\, C_{p}\,\varepsilon^{2w_{m}}.
\end{multline}

$\circ$ \textit{First}, we look for the $H^{J+1}$-estimate. We neglect
$\sum_{|\alpha|=J+1}\Phi_{\varepsilon}(\nabla,\partial^{\alpha}\widetilde{v}_{\varepsilon}^{R})$
since it is positive as a sum of positive quadratic forms. Hence,
for all $\varepsilon\in]0,\varepsilon_{d}]$ and for all time $t\in[0,T_{\varepsilon}^{*}]$:
\[
\frac{1}{2}\partial_{t}\left\Vert \widetilde{v}_{\varepsilon}^{R}(t,\cdot)\right\Vert _{\overset{\circ}{H}{}^{J+1}(\T\times\R)}^{2}\leq C_{J+1}^{2}\left\Vert \widetilde{v}_{\varepsilon}^{R}(t,\cdot)\right\Vert _{\overset{\circ}{H}{}^{J+1}(\T\times\R)}^{2}+S_{J+1}(t)\,,
\]
 where $S_{J+1}$ is defined as: 
\begin{multline*}
S_{J+1}(t):=C_{J+1}^{1}\left\Vert \widetilde{v}_{\varepsilon}^{R}(t,\cdot)\right\Vert _{H^{J}(\T\times\R)}^{2}\\
+C_{J+1}^{3}\sum_{|\alpha|=J+1}\left\Vert \varepsilon^{-(1-\delta_{\alpha_{1},0})}\partial^{\alpha}\widetilde{v}_{\varepsilon}^{R}(t,\cdot)\right\Vert _{L^{2}(\T\times\R)}^{2}+(J+2)\, C_{p}\,\varepsilon^{2w_{m}}.
\end{multline*}
 From the assumption $\mathcal{P}(J)$, Inequalities \eqref{eqpropestu1}
and \eqref{eqpropestu2} hold. The functions $t\longmapsto\left\Vert \widetilde{v}_{\varepsilon}^{R}(t,\cdot)\right\Vert _{H^{j}}^{2}$
and ${\displaystyle t\longmapsto\sum_{|\alpha|=J+1}\left\Vert \varepsilon^{-(1-\delta_{\alpha_{1},0})}\partial^{\alpha}\widetilde{v}_{\varepsilon}^{R}(t,\cdot)\right\Vert _{L^{2}(\T\times\R)}^{2}}$
are in~$L^{1}([0,T_{\varepsilon}^{*}])$. In addition there exists
$M_{m}^{1}$ a positive constant such that for all $\varepsilon\in]0,\varepsilon_{d}]$
and for all time $t\in[0,T_{\varepsilon}^{*}]$: 
\[
\left\Vert \widetilde{v}_{\varepsilon}^{R}(t,\cdot)\right\Vert _{H^{j}}^{2}\leq M_{m}^{1}\,\varepsilon^{2w_{m}}\qquad\text{and}\qquad\sum_{|\alpha|=J+1}\left\Vert \varepsilon^{-(1-\delta_{\alpha_{1},0})}\partial^{\alpha}\widetilde{v}_{\varepsilon}^{R}(t,\cdot)\right\Vert _{L^{2}(\T\times\R)}^{2}\leq M_{m}^{1}\,\varepsilon^{2w_{m}}.
\]
 Thus, the function $t\longmapsto S_{J+1}(t)$ is in $L^{1}([0,T_{\varepsilon}^{*}])$
and satisfies for all $\varepsilon\in]0,\varepsilon_{d}]$ and for
all time $t\in[0,T_{\varepsilon}^{*}]$: 
\[
S_{J+1}(t)\leq\,(2M_{m}^{1}+(J+2)C_{p})\varepsilon^{2w_{m}}.
\]
 We can apply the Gronwall's lemma, for all $\varepsilon\in]0,\varepsilon_{d}]$
and for all time $t\in[0,T_{\varepsilon}^{*}]$, 
\begin{align}
\hspace{-0.1cm}\left\Vert \widetilde{v}_{\varepsilon}^{R}(t,\cdot)\right\Vert _{\overset{\circ}{H}{}^{J+1}(\T\times\R)}^{2}\leq\int_{0}^{t}e^{2C_{J+1}^{2}(t-s)}S_{J+1}(s)ds\leq\frac{2M_{m}^{1}+(J+2)C_{p}}{2C_{J+1}^{2}}\varepsilon^{2w_{m}}\left(e^{2C_{J+1}^{2}t}-1\right).\label{ineqvelo1}
\end{align}
 Setting $K_{J+1}^{1}:=(2M_{m}^{1}+(J+2)C_{p})/(2C_{J+1}^{2})$ and
$K_{j+1}^{2}:=2C_{J+1}^{2}$, the first inequality of $\mathcal{P}(J+1)$
is proved.\\

\bigskip{}

$\circ$ \textit{Then}, to obtain the estimation $L^{2}([0,T_{\varepsilon}^{*}],H^{J+2})$,
we go back to Equation (\ref{inegenergyJ+1}). We integrate it with
respect to the time $t$, 
\begin{multline}\label{ineqenergyhyper}
\frac{1}{2}\left\Vert \widetilde{v}_{\varepsilon}^{R}(t,.)\right\Vert _{\overset{\circ}{H}^{J+1}}^{2}+\int_{0}^{t}{\displaystyle \sum_{|\alpha|=J+1}\Phi_{\varepsilon}(\nabla,\partial^{\alpha}\widetilde{v}_{\varepsilon}^{R})(s)ds}\\
\leq C_{J+1}^{2}\int_{0}^{t}\left(\left\Vert \widetilde{v}_{\varepsilon}^{R}(s,\cdot)\right\Vert _{\overset{\circ}{H}{}^{J+1}}^{2}+S_{J+1}(s)\right)ds.
\end{multline}
 Since $\left\Vert \widetilde{v}_{\varepsilon}^{R}(t,.)\right\Vert _{\overset{\circ}{H}^{J+1}}^{2}$
is positive, we neglect it in the left-hand side of the Inequality
\eqref{ineqenergyhyper}. According to \eqref{ineqvelo1}, the functions
$t\longmapsto\left\Vert \widetilde{v}_{\varepsilon}^{R}(t,\cdot)\right\Vert _{\overset{\circ}{H}{}^{J+1}}^{2}$
and $S_{J+1}$ are in $L^{1}([0,T_{\varepsilon}^{*}])$. Furthermore,
there exists $K_{J+1}^{2}$ such that for all $\varepsilon\in]0,\varepsilon_{d}]$
and for all time $t\in[0,T_{\varepsilon}^{*}]$: 
\[
\left\Vert \widetilde{v}_{\varepsilon}^{R}(t,\cdot)\right\Vert _{\overset{\circ}{H}{}^{J+1}}^{2}+S(t)\leq K_{J+1}^{2}\varepsilon^{2w_{m}}\, t\leq K_{J+1}^{2}\varepsilon^{2w_{m}}.
\]
 So to say, there exists~$K_{J+1}^{2}(>0)$ such that for all $\varepsilon\in]0,\varepsilon_{d}]$
and for all time $t\in[0,T_{\varepsilon}^{*}]$: 
\[
\int_{0}^{t}{\displaystyle \sum_{|\alpha|=J+1}\Phi_{\varepsilon}(\nabla,\partial^{\alpha}\widetilde{v}_{\varepsilon}^{R})(s)ds\leq\varepsilon^{2w_{s}}\, K_{J+1}^{2}\, t.}
\]
 We end the proof choosing $K_{m}^{i}:={\displaystyle \max_{j\in\llbracket0,m+4\rrbracket}K_{j}^{i}}$
($i\in\{1,2\}$). \vspace{-0.5cm}
 \hfill{}$\square$\\

\bigskip{}

\begin{remark} In the above discussion, we only use the assumption
$M\geq2$ whereas we assume~$M\geq~7/2$. It becomes crucial when
estimating the pressure. \end{remark} 

\subsection{Control over the pressure}

\label{estpressure} 

In Subsection \ref{estpressure} we prove Proposition \ref{lemestq}.
The result is again proved by induction on the size $m$ setting $\mathcal{Q}(m)$:
\begin{equation}
\mathcal{Q}(m)\,:=\ \text{{\it " The Proposition \ref{lemestq} is satisfied up to the integer \ensuremath{m}".}}\label{recurezqnl}
\end{equation}
 To go from $m$ to $m+1$, the proof is once more based on an energy
method for Equation~(\ref{eqqepsilon}) in the anisotropic Sobolev
spaces $H_{(1,\varepsilon)}^{m}$ (defined page \pageref{eqinfininorm}):
\begin{lemma} \label{lemallq}There exits a positive constant $\varepsilon_{d}>0$
such that for any $J\in\llbracket0,m+3\rrbracket$, there exist three
positive constants $C_{J}^{1}$, $C_{J}^{2}$ and $C_{J}^{3}$ such
that for all $\varepsilon\in]0,\varepsilon_{d}]$ and for all time
$t\in[0,T_{\varepsilon}^{*}]$: 
\begin{multline}\label{ineqergJq}
\frac{1}{2}\partial_{t}\left\Vert \widetilde{q}_{\varepsilon}^{R}(t,\cdot)\right\Vert _{\overset{\circ}{H}{}_{(1,\varepsilon)}^{J}(\T\times\R)}^{2}\leq C_{J}^{1}\left(1+\left\Vert \widetilde{v}_{\varepsilon}^{R}(t,\cdot)\right\Vert _{H^{m+4}(\T\times\R)}^{2}\right)\left\Vert \widetilde{q}_{\varepsilon}^{R}(t,\cdot)\right\Vert _{\overset{\circ}{H}{}_{(1,\varepsilon)}^{J}(\T\times\R)}^{2}\\
+C_{J}^{2}\left(1+\left\Vert \widetilde{v}_{\varepsilon}^{R}(t,\cdot)\right\Vert _{H^{m+4}}^{2}\right)\left\Vert \widetilde{q}_{\varepsilon}^{R}(t,\cdot)\right\Vert _{H_{(1,\varepsilon)}^{J-1}(\T\times\R)}^{2}\\
+C_{J}^{3}\left(\left\Vert \widetilde{v}_{\varepsilon}^{R}(t,\cdot)\right\Vert _{H^{m+4}(\T\times\R)}^{2}+\varepsilon^{2(N-R)}\right).
\end{multline}
 \end{lemma} The lack of dissipation has two consequences. 
\begin{description}
\item [{\quad{}}] -There is no \textit{absorption phenomena}. It makes
us consider the anisotropic Sobolev spaces instead of the classical
Sobolev spaces. An unexpected effect of considering the anisotropic
Sobolev norms is that the family $\{\partial_{\theta}\widetilde{q}_{\varepsilon}^{R}\}_{\varepsilon}$
becomes singular with respect to $\varepsilon$ in $L^{2}$. However~$\{\varepsilon\partial_{\theta}\widetilde{q}_{\varepsilon}^{R}\}_{\varepsilon}$
is. We introduce a power of $\varepsilon$ when necessary thanks to
the integer $M$. It explains why $M$ is assumed at least larger
than $3$. 
\item [{\quad{}}] -There is no \textit{regularization phenomena} over
the pressure. A difficulty appears when we want to estimate terms
such as $\nabla p$. Each time it appears, we integrate by parts (if
possible) to pass the derivative on the velocity. It explains the
appearance of the $H^{m+4}$ norm of the velocity. The regularization
of the velocity plays again a crucial role in this process. 
\end{description}
The energy method in $H_{(1,\varepsilon)}^{m}$ consists in differentiating
Equation (\ref{equepsilon}) by $\varepsilon^{\alpha_{1}}\partial^{\alpha}$.
Then we multiply it by $\varepsilon^{\alpha_{1}}\partial^{\alpha}\widetilde{q}_{\varepsilon}^{R}$
and integrate (with respect to the space variables $(\theta,y)$),
\begin{multline}
\frac{1}{2}\partial_{t}\left\Vert \varepsilon^{\alpha_{1}}\partial^{\alpha}\widetilde{q}_{\varepsilon}^{R}\right\Vert _{L^{2}}^{2}+\left\langle \varepsilon^{\alpha_{1}}\partial^{\alpha}(\mathcal{H}\widetilde{q}_{\varepsilon}^{R}),\varepsilon^{\alpha_{1}}\partial^{\alpha}\widetilde{q}_{\varepsilon}^{R}\right\rangle \\
+\left\langle \varepsilon^{\alpha_{1}}\partial^{\alpha}(\mathcal{B}\widetilde{q}_{\varepsilon}^{R}),\varepsilon^{\alpha_{1}}\partial^{\alpha}\widetilde{q}_{\varepsilon}^{R}\right\rangle +\left\langle \varepsilon^{\alpha_{1}}\partial^{\alpha}(\mathcal{D}\widetilde{v}_{\varepsilon}^{R}),\varepsilon^{\alpha_{1}}\partial^{\alpha}\widetilde{q}_{\varepsilon}^{R}\right\rangle \\
=\left\langle \varepsilon^{\alpha_{1}}\partial^{\alpha}S_{\varepsilon}^{0,R,N},\varepsilon^{\alpha_{1}}\partial^{\alpha}\widetilde{q}_{\varepsilon}^{R}\right\rangle -\left\langle \varepsilon^{\alpha_{1}}\partial^{\alpha}(\mathcal{F}\widetilde{v}_{\varepsilon}^{R}),\varepsilon^{\alpha_{1}}\partial^{\alpha}\widetilde{q}_{\varepsilon}^{R}\right\rangle ,
\end{multline}
 where operators $\mathcal{H}$, $\mathcal{B}$, $\mathcal{D}$ and
$\mathcal{F}$ are defined as 
\begin{align*}
 & \mathcal{H}\widetilde{q}_{\varepsilon}^{R}:=\varepsilon^{-1}h\,\partial_{y}\widetilde{q}_{\varepsilon}^{R},\quad\mathcal{F}\widetilde{v}_{\varepsilon}^{R}:=\varepsilon^{M-2}\left(\widetilde{v}_{\varepsilon}^{1R}\,\partial_{\theta}q_{\varepsilon}^{a}+\varepsilon\,\widetilde{v}_{\varepsilon}^{2R}\,\partial_{y}q_{\varepsilon}^{a}\right)+C\,\varepsilon^{M-2}q_{\varepsilon}^{a}\left(\partial_{\theta}\widetilde{v}_{\varepsilon}^{1R}+\varepsilon\,\partial_{y}\widetilde{v}_{\varepsilon}^{2R}\right)\,,\\
 & \mathcal{B}\widetilde{q}_{\varepsilon}^{R}:=\varepsilon^{M-2}\left(v_{\varepsilon}^{1a}\,\partial_{\theta}\widetilde{q}_{\varepsilon}^{R}+\varepsilon\, v_{\varepsilon}^{2a}\,\partial_{y}\widetilde{q}_{\varepsilon}^{R}\right),\quad\mathcal{D}\widetilde{q}_{\varepsilon}^{R}:=C\,\varepsilon^{M-2}\widetilde{q}_{\varepsilon}^{R}\left(\partial_{\theta}v_{\varepsilon}^{1a}+\varepsilon\,\partial_{y}v_{\varepsilon}^{2a}\right)\,.
\end{align*}

In the sequel we prove several lemmas where we estimate each contributions.
We sort them contingent on how much they are singular with respect
to $\varepsilon$ in the anisotropic Sobolev spaces. In Subsection
\ref{subsectionnonsingular}, we study the contribution of operators
$\mathcal{H}$, $\mathcal{B}$ and $\mathcal{D}$. Then in Subsection
\ref{subsectionsingular}, due to technical computations, we have
to assume that $M\geq7/2$ and deal with the terms~$\mathcal{F}$
and~$S_{\varepsilon}^{0,R,N}$. \\
 Finally in Subsection \ref{subsubsectionprooflemma}, once Lemma
\ref{lemallq} is proved, we present the induction to prove Proposition
\ref{lemestq}.

\subsubsection{Non-singular terms: $M\geq3$ }

\label{subsectionnonsingular}

First we deal with the contributions of operators $\mathcal{H}$,
$\mathcal{B}$ and $\mathcal{D}$. Since we introduce the anisotropic
Sobolev spaces to get ride of the singular transport $\mathcal{H}$,
it is no longer singular:

\begin{lemma}\label{lemcomh} Assume that $M\geq2$. Select a multi-index
$\alpha\in\N^{2}$ of length smaller than $m+3$. Then, there exist
two positive constants $C_{\mathcal{H}}^{1}$ and $C_{\mathcal{H}}^{2}$
such that for all $\varepsilon\in]0,1]$ and for all time $t\in[0,T_{\varepsilon}^{*}]$:
\begin{align*}
\left|\left\langle \varepsilon^{\alpha_{1}}\partial^{\alpha}\left(\varepsilon^{-1}h\,\partial_{y}\widetilde{q}_{\varepsilon}^{R}\right),\varepsilon^{\alpha_{1}}\partial^{\alpha}\widetilde{q}_{\varepsilon}^{R}\right\rangle \right|(t)\leq & C_{\mathcal{H}}^{1}\left\Vert \widetilde{q}_{\varepsilon}^{R}(t,\cdot)\right\Vert _{\overset{\circ}{H}{}_{(1,\varepsilon)}^{|\alpha|}(\T\times\R)}^{2}+C_{\mathcal{H}}^{2}\left\Vert \widetilde{q}_{\varepsilon}^{R}(t,\cdot)\right\Vert _{H_{(1,\varepsilon)}^{|\alpha|-1}(\T\times\R)}^{2}.
\end{align*}
 \end{lemma} When $|\alpha|=0$, the contribution of $\mathcal{H}$
is even vanishing. Since it is very classical, we do not write the
proof.\\

Presently we compute the contribution of operators $\mathcal{B}$
and $\mathcal{D}$. We have to estimate the term~$\partial_{\theta}\widetilde{q}_{\varepsilon}^{R}$
singular in $\varepsilon$ in the anisotropic spaces. It thus requires:

\begin{lemma}\label{lemm3} Assume $M\geq3$. Select a multi-index
$\alpha\in\N^{2}$ with length smaller than $m+3$. Select $\mathcal{V}\in\{\mathcal{B},\mathcal{D}\}$.
There exist two positive constants $C_{\mathcal{V}}^{1}$ and $C_{\mathcal{V}}^{2}$
such that for all $\varepsilon\in]0,1]$ and for all time $t\in[0,T_{\varepsilon}^{*}]$,
\begin{align*}
\left|\left\langle \varepsilon^{\alpha_{1}}\partial^{\alpha}\left(\mathcal{V}\widetilde{q}_{\varepsilon}^{R}\right),\varepsilon^{\alpha_{1}}\partial^{\alpha}\widetilde{q}_{\varepsilon}^{R}\right\rangle \right|(t)\leq C_{\mathcal{V}}^{1}\left\Vert \widetilde{q}_{\varepsilon}^{R}(t,\cdot)\right\Vert _{\overset{\circ}{H}{}_{(1,\varepsilon)}^{|\alpha|}(\T\times\R)}^{2}+C_{\mathcal{V}}^{2}\left\Vert \widetilde{q}_{\varepsilon}^{R}(t,\cdot)\right\Vert _{H_{(1,\varepsilon)}^{|\alpha|-1}(\T\times\R)}^{2}.
\end{align*}
 \end{lemma} Of course the contribution of $\mathcal{D}$ only requires
$M\geq2$ since $\partial_{y}q_{\varepsilon}^{R}$ is bounded in the
anisotropic Sobolev spaces.

\begin{proof}[Proof of Lemma \ref{lemm3}] We prove the result
for operator $\mathcal{B}$. We consider a multi-index $\alpha\in\N^{2}$
such that $|\alpha|\geq1$. We have: 
\begin{multline*}
\left\langle \varepsilon^{\alpha_{1}}\partial^{\alpha}\left(\mathcal{B}\,\widetilde{q}_{\varepsilon}^{R}\right),\varepsilon^{\alpha_{1}}\partial^{\alpha}\widetilde{q}_{\varepsilon}^{R}\right\rangle =\varepsilon^{M-2}\left\langle \varepsilon^{\alpha_{1}}\partial^{\alpha}\left(v_{\varepsilon}^{a1}\,\partial_{\theta}\widetilde{q}_{\varepsilon}^{R}\right),\varepsilon^{\alpha_{1}}\partial^{\alpha}\widetilde{q}_{\varepsilon}^{R}\right\rangle \\
+\varepsilon^{M-2}\left\langle \varepsilon^{\alpha_{1}}\partial^{\alpha}\left(\varepsilon\, v_{\varepsilon}^{a2}\,\partial_{y}\widetilde{q}_{\varepsilon}^{R}\right),\varepsilon^{\alpha_{1}}\partial^{\alpha}\widetilde{q}_{\varepsilon}^{R}\right\rangle \,.
\end{multline*}
 The two terms are estimated performing the same proof. Thus we only
write the proof for the term $\varepsilon^{M-2}\left\langle \varepsilon^{\alpha_{1}}\partial^{\alpha}\left(v_{\varepsilon}^{a1}\,\partial_{\theta}\widetilde{q}_{\varepsilon}^{R}),\varepsilon^{\alpha_{1}}\partial^{\alpha}\widetilde{q}_{\varepsilon}^{R}\right)\right\rangle $.
Of course, the estimates for the second term only requires $M\geq2$
since it appears the derivative $\partial_{y}$ instead of $\partial_{\theta}$.\\

We apply the Leibniz formula. Then to diminish the number of derivatives
acting on the pressure (for the extremal term), we consider an integration
by parts. There exists a family $\{C_{\alpha,\beta}\}$ of positive
constants such that, 
\begin{align}\label{eqnose}
 & \varepsilon^{M-2}\left|\left\langle \varepsilon^{\alpha_{1}}\partial^{\alpha}\left(v_{\varepsilon}^{a1}\,\partial_{\theta}\widetilde{q}_{\varepsilon}^{R}\right),\varepsilon^{\alpha_{1}}\partial^{\alpha}\widetilde{q}_{\varepsilon}^{R}\right\rangle \right|\nonumber \\
 & =\varepsilon^{M-2}\sum_{\beta<\alpha}C_{\alpha,\beta}\int_{\T\times\R}\varepsilon^{2\alpha_{1}}\partial^{\alpha-\beta}v_{\varepsilon}^{a1}\,\partial^{\beta}\left(\partial_{\theta}\widetilde{q}_{\varepsilon}^{R}\right)\,\partial^{\alpha}\widetilde{q}_{\varepsilon}^{R}d\theta\, dy\nonumber \\
 & \hspace{6,5cm}-\varepsilon^{(M-2)}/2\int_{\T\times\R}\partial_{\theta}v_{\varepsilon}^{a1}(\varepsilon^{\alpha_{1}}\partial^{\alpha}\widetilde{q}_{\varepsilon}^{R})^{2}\, d\theta\, dy\,.
\end{align}

$\circ$ \textit{First}, we get a bound for the last term in the right-hand-side
of Equation \eqref{eqnose}. For all $\varepsilon\in]0,1]$ and for
all time $t\in[0,T_{\varepsilon}^{*}]$: 
\begin{align}
\left|-\varepsilon^{(M-2)}/2\int_{\T\times\R}\partial_{\theta}v_{\varepsilon}^{a1}(t,\cdot)(\varepsilon^{\alpha_{1}}\partial^{\alpha}\widetilde{q}_{\varepsilon}^{R}(t,\cdot))^{2}\, d\theta\, dy\right| & \leq\left\Vert v_{\varepsilon}^{a}(t,\cdot)\right\Vert _{W^{1,\infty}}/2\left\Vert q_{\varepsilon}(t,\cdot)\right\Vert _{\overset{\circ}{H}{}_{(1,\varepsilon)}^{|\alpha|}}^{2}\,,\nonumber \\
 & \leq C_{a}/2\left\Vert q_{\varepsilon}(t,\cdot)\right\Vert _{\overset{\circ}{H}{}_{(1,\varepsilon)}^{|\alpha|}}^{2}.\label{ineqipp}
\end{align}

$\circ$ \textit{Then}, to control the other terms in Equation \eqref{eqnose},
we make appear the $\varepsilon$-derivative $\varepsilon\,\partial_{\theta}$
paying a loss of precision on~$M$. For all $\varepsilon\in]0,1]$
and for all time $t\in[0,T_{\varepsilon}^{*}]$, 
\begin{align*}
 & \left|\varepsilon^{M-2}\int_{\T\times\R}\varepsilon^{2\alpha_{1}}\ \partial^{\alpha-\beta}v_{\varepsilon}^{a1}(t,\cdot)\ \partial^{\beta}\partial_{\theta}\widetilde{q}_{\varepsilon}^{R}(t,\cdot)\ \partial^{\alpha}\widetilde{q}_{\varepsilon}^{R}(t,\cdot)d\theta\, dy\right|\\
 & \hspace{4cm}=\left|\fbox{\ensuremath{\textcolor{marin}{\varepsilon^{M-3}}}}\int_{\T\times\R}\varepsilon^{2\alpha_{1}+\textcolor{marin}{1}}\ \partial^{\alpha-\beta}v_{\varepsilon}^{a1}(t,\cdot)\ \partial^{\beta}\partial_{\theta}\widetilde{q}_{\varepsilon}^{R}(t,\cdot)\ \partial^{\alpha}\widetilde{q}_{\varepsilon}^{R}(t,\cdot)\, d\theta\, dy\right|\,,\\
 & \hspace{4cm}\leq\frac{\left\Vert v_{\varepsilon}^{a1}(t,\cdot)\right\Vert _{W^{m+4,\infty}}}{2}\left(\left\Vert \varepsilon^{\alpha_{1}+1}\partial^{\beta}\partial_{\theta}\widetilde{q}_{\varepsilon}^{R}(t,\cdot)\right\Vert _{L^{2}}^{2}+\left\Vert \varepsilon^{\alpha_{1}}\partial^{\alpha}\widetilde{q}_{\varepsilon}^{R}(t,\cdot)\right\Vert _{L^{2}}^{2}\right)\,,\\
 & \hspace{4cm}\leq\frac{C_{a}}{2}\left(\left\Vert \widetilde{q}_{\varepsilon}^{R}(t,\cdot)\right\Vert _{H_{(1,\varepsilon)}^{|\alpha|-1}}^{2}+2\left\Vert \widetilde{q}_{\varepsilon}^{R}(t,\cdot)\right\Vert _{\overset{\circ}{H}{}_{(1,\varepsilon)}^{|\alpha|}}^{2}\right).
\end{align*}
 Choosing $A:=\sum_{\beta<\alpha}C_{\alpha,\beta}C_{a}/2>0$ and $B:=(\sum_{\beta<\alpha}C_{\alpha,\beta}+1/2)C_{a}>0$
we obtain that for all $\varepsilon\in]0,1]$ and for all time $t\in[0,T_{\varepsilon}^{*}]$,
\begin{equation}
\varepsilon^{M-2}\left|\left\langle \varepsilon^{\alpha_{1}}\partial^{\alpha}\left(v_{\varepsilon}^{a1}\,\partial_{\theta}\widetilde{q}_{\varepsilon}^{R}\right),\varepsilon^{\alpha_{1}}\partial^{\alpha}\widetilde{q}_{\varepsilon}^{R}\right\rangle \right|(t)\leq A\left\Vert \widetilde{q}_{\varepsilon}^{R}(t,\cdot)\right\Vert _{H_{(1,\varepsilon)}^{|\alpha|-1}}^{2}+B\left\Vert \widetilde{q}_{\varepsilon}^{R}(t,\cdot)\right\Vert _{\overset{\circ}{H}{}_{(1,\varepsilon)}^{|\alpha|}}^{2}.\label{eqlemb1}
\end{equation}
 That ends the proof. \end{proof}

\subsubsection{Singular terms: $M\geq7/2$}

\label{subsectionsingular}

There remains to estimate the contributions of $\mathcal{F}$ and
$S_{\varepsilon}^{0,R,N}$. They are put aside since we need~$M\geq7/2$
to desingularize it. It is actually a technical assumption and it
should be relaxed to $M\geq3$ by proving more regularity on the approximated
solution $(q_{\varepsilon}^{a},v_{\varepsilon}^{a})$. \\

\paragraph{Estimate for operator $\mathcal{F}$.}

We start by estimating the contribution of the operator $\mathcal{F}$.
The main difference with the three previous operators ($\mathcal{H}$,
$\mathcal{B}$ and $\mathcal{D}$) is that the operator $\mathcal{F}$
acts on the velocity $\widetilde{v}_{\varepsilon}^{R}$. We have to
be prudent with terms of the form 
\[
\partial^{\alpha}(q_{\varepsilon}^{a}\,(\partial_{\theta}\widetilde{v}_{\varepsilon}^{1R}+\partial_{y}\widetilde{v}_{\varepsilon}^{2R})).
\]
 First it is more singular with respect to the number of derivatives
acting on the velocity. Once more, when estimating the product, we
have to use a control over the pressure $q_{\varepsilon}^{a}$ in
$L^{\infty}$-norm. We may lose some regularity to go back to the
anisotropic Sobolev norms. This term can be more singular in $L^{2}$-norm.
Thus it requires: \begin{lemma} \label{lemm0}Assume $M\geq7/2$.
Select a multi-index $\alpha\in\N^{2}$ with length smaller than $m+3$.
There exist two positive constants $C_{\mathcal{F}}^{1}$ and $S_{\mathcal{F}}$
such that for all $\varepsilon\in]0,1]$ and for all time $t\in[0,T_{\varepsilon}^{*}]$:
\[
\left|\left\langle \varepsilon^{\alpha_{1}}\partial^{\alpha}\left(\mathcal{F}\widetilde{v}_{\varepsilon}^{R}\right),\varepsilon^{\alpha_{1}}\partial^{\alpha}\widetilde{q}_{\varepsilon}^{R}\right\rangle \right|(t)\leq S_{\mathcal{F}}\left\Vert \widetilde{v}_{\varepsilon}^{R}(t,\cdot)\right\Vert _{H^{m+4}(\T\times\R)}^{2}+C_{\mathcal{F}}^{1}\left\Vert \widetilde{q}_{\varepsilon}^{R}(t,\cdot)\right\Vert _{\overset{\circ}{H}{}_{(1,\varepsilon)}^{|\alpha|}(\T\times\R)}^{2}.
\]
 \end{lemma}

\begin{proof}[Proof of Lemma \ref{lemm0}] Select a multi-index
$\alpha\in\N^{2}$ satisfying $|\alpha|\leq m+3$. We decompose $\mathcal{F}$
as follows: 
\begin{multline}
\left\langle \varepsilon^{\alpha_{1}}\partial^{\alpha}\ (\mathcal{F}\widetilde{q}_{\varepsilon}^{R}),\varepsilon^{\alpha_{1}}\partial^{\alpha}\widetilde{q}_{\varepsilon}^{R}\right\rangle =\left\langle \varepsilon^{M-2}\varepsilon^{\alpha_{1}}\partial^{\alpha}\left(\widetilde{v}_{\varepsilon}^{1R}\,\partial_{\theta}q_{\varepsilon}^{a}\right),\varepsilon^{\alpha_{1}}\partial^{\alpha}\widetilde{q}_{\varepsilon}^{R}\right\rangle \\
+\left\langle \varepsilon^{M-2}\varepsilon^{\alpha_{1}}\partial^{\alpha}\left(\varepsilon\,\widetilde{v}_{\varepsilon}^{2R}\,\partial_{y}q_{\varepsilon}^{a}\right),\varepsilon^{\alpha_{1}}\partial^{\alpha}\widetilde{q}_{\varepsilon}^{R}\right\rangle \\
+\left\langle C\,\varepsilon^{M-2}\varepsilon^{\alpha_{1}}\partial^{\alpha}\left(q_{\varepsilon}^{a}\,\partial_{\theta}\widetilde{v}_{\varepsilon}^{1R}\right),\varepsilon^{\alpha_{1}}\partial^{\alpha}\widetilde{q}_{\varepsilon}^{R}\right\rangle +\left\langle C\,\varepsilon^{M-2}\varepsilon^{\alpha_{1}}\partial^{\alpha}\left(q_{\varepsilon}^{a}\,\varepsilon\partial_{y}\widetilde{v}_{\varepsilon}^{2R}\right),\varepsilon^{\alpha_{1}}\partial^{\alpha}\widetilde{q}_{\varepsilon}^{R}\right\rangle \,.
\end{multline}
 The first and second term (respectively the third and fourth term)
can be computed following the same line. So we only provide the estimates
for the first term (respectively the third term).\\

$\circ$ \textit{The first term.} We start by estimating the contribution
of $\left\langle \varepsilon^{M-2}\varepsilon^{\alpha_{1}}\partial^{\alpha}\left(\widetilde{v}_{\varepsilon}^{1R}\,\partial_{\theta}q_{\varepsilon}^{a}\right),\varepsilon^{\alpha_{1}}\partial^{\alpha}\widetilde{q}_{\varepsilon}^{R}\right\rangle $:
\begin{align*}
\left\langle \varepsilon^{M-2}\varepsilon^{\alpha_{1}}\partial^{\alpha}\left(\widetilde{v}_{\varepsilon}^{1R}\,\partial_{\theta}q_{\varepsilon}^{a}\right),\varepsilon^{\alpha_{1}}\partial^{\alpha}\widetilde{q}_{\varepsilon}^{R}\right\rangle \leq\frac{1}{2}\left(\left\Vert \varepsilon^{M-2}\varepsilon^{\alpha_{1}}\partial^{\alpha}\left(\widetilde{v}_{\varepsilon}^{1R}\,\partial_{\theta}q_{\varepsilon}^{a}\right)\right\Vert _{L^{2}}^{2}+\left\Vert \varepsilon^{\alpha_{1}}\partial^{\alpha}\widetilde{q}_{\varepsilon}^{R}\right\Vert _{L^{2}}^{2}\right)
\end{align*}
 We use the Leibniz formula. Then apply the Minkowski inequality and
loss a power of $\varepsilon$ to get a control of the approximated
solution in the anisotropic Sobolev spaces. There exists a family
of positive constant $\{C_{\alpha,\beta}\}$ such that 
\begin{align*}
\left\Vert \varepsilon^{M-2}\varepsilon^{\alpha_{1}}\partial^{\alpha}\left(\widetilde{v}_{\varepsilon}^{1R}\,\partial_{\theta}q_{\varepsilon}^{a}\right)\right\Vert _{L^{2}} & \leq\varepsilon^{M-2}\sum_{0\leq\beta\leq\alpha}C_{\alpha,\beta}\left\Vert \varepsilon^{\alpha_{1}+\textcolor{marin}{1-1}}\partial^{\beta}\partial_{\theta}q_{\varepsilon}^{a}\right\Vert _{L^{\infty}}\left\Vert \partial^{\alpha-\beta}\widetilde{v}_{\varepsilon}^{1R}\right\Vert _{L^{2}}\,,\\
 & \leq\fbox{\ensuremath{\textcolor{marin}{\varepsilon^{M-3}}}}\sum_{0\leq\beta\leq\alpha}C_{\alpha,\beta}\left\Vert q_{\varepsilon}^{a}\right\Vert _{W_{(1,\varepsilon)}^{m+4,\infty}}\left\Vert v_{\varepsilon}^{R}\right\Vert _{H^{m+3}}\,.
\end{align*}
 Then, we use the imbeddings $\left\Vert q_{\varepsilon}^{a}(t,\cdot)\right\Vert _{W_{(1,\varepsilon)}^{m+4,\infty}}\leq\varepsilon^{-\frac{1}{2}}\left\Vert q_{\varepsilon}^{a}(t,\cdot)\right\Vert _{H_{(1,\varepsilon)}^{m+6}}$,
\begin{align*}
\left\Vert \varepsilon^{M-2}\varepsilon^{\alpha_{1}}\partial^{\alpha}\left(\widetilde{v}_{\varepsilon}^{1R}\,\partial_{\theta}q_{\varepsilon}^{a}\right)\right\Vert _{L^{2}} & \leq\fbox{\ensuremath{\textcolor{marin}{\varepsilon^{M-7/2}}}}\sum_{0\leq\beta\leq\alpha}C_{\alpha,\beta}\left\Vert q_{\varepsilon}^{a}\right\Vert _{H_{(1,\varepsilon)}^{m+6}}\left\Vert v_{\varepsilon}^{R}\right\Vert _{H^{m+3}}\,,\\
 & \leq C_{a}\sum_{0\leq\beta\leq\alpha}C_{\alpha,\beta}\left\Vert v_{\varepsilon}^{R}(t,\cdot)\right\Vert _{{H}{}^{m+3}}\,.
\end{align*}
 Finally for all $\varepsilon\in]0,1]$ and for all time $t\in[0,T_{\varepsilon}^{*}]$,
\begin{multline}
\left|\left\langle \varepsilon^{M-2}\varepsilon^{\alpha_{1}}\partial^{\alpha}\left(\widetilde{v}_{\varepsilon}^{1R}\,\partial_{\theta}q_{\varepsilon}^{a}\right),\varepsilon^{\alpha_{1}}\partial^{\alpha}\widetilde{q}_{\varepsilon}^{R}\right\rangle \right|(t)\\
\leq\frac{1}{2}\left(C_{a}^{2}\left(\sum_{0\leq\beta\leq\alpha}C_{\alpha,\beta}\right)^{2}\left\Vert v_{\varepsilon}^{R}(t,\cdot)\right\Vert _{{H}{}^{m+3}}^{2}+\left\Vert \varepsilon^{\alpha_{1}}\partial^{\alpha}\widetilde{q}_{\varepsilon}^{R}(t,\cdot)\right\Vert _{L^{2}}^{2}\right)\,.
\end{multline}

$\circ$ \textit{Estimate of the third term.} We now move to the estimate
of $\left\langle C\,\varepsilon^{M-2}\varepsilon^{\alpha_{1}}\partial^{\alpha}\left(q_{\varepsilon}^{a}\,\partial_{\theta}\widetilde{v}_{\varepsilon}^{1R}\right),\varepsilon^{\alpha_{1}}\partial^{\alpha}\widetilde{q}_{\varepsilon}^{R}\right\rangle $.
First, apply the Cauchy-Schwarz inequality together with the Young
inequality: 
\[
\left|\left\langle C\,\varepsilon^{M-2}\varepsilon^{\alpha_{1}}\partial^{\alpha}\left(q_{\varepsilon}^{a}\,\partial_{\theta}\widetilde{v}_{\varepsilon}^{1R}\right),\varepsilon^{\alpha_{1}}\partial^{\alpha}\widetilde{q}_{\varepsilon}^{R}\right\rangle \right|\leq\frac{1}{2}\left(\left\Vert C\,\varepsilon^{M-2}\varepsilon^{\alpha_{1}}\partial^{\alpha}\left(q_{\varepsilon}^{a}\,\partial_{\theta}\widetilde{v}_{\varepsilon}^{1R}\right)\right\Vert _{L^{2}}^{2}+\left\Vert \varepsilon^{\alpha_{1}}\partial^{\alpha}\widetilde{q}_{\varepsilon}^{R}\right\Vert _{L^{2}}^{2}\right).
\]
 To get rid with the contribution of the nonlinear term $\left\Vert C\,\varepsilon^{M-2}\varepsilon^{\alpha_{1}}\partial^{\alpha}\left(q_{\varepsilon}^{a}\,\partial_{\theta}\widetilde{v}_{\varepsilon}^{1R}\right)\right\Vert _{L^{2}}$
we start with a Leibniz formula, then: 
\begin{align*}
\left\Vert C\varepsilon^{M-2}\varepsilon^{\alpha_{1}}\partial^{\alpha}\left(q_{\varepsilon}^{a}\,\partial_{\theta}\widetilde{v}_{\varepsilon}^{1R}\right)\right\Vert _{L^{2}} & \leq\ C\varepsilon^{M-2}\sum_{0\leq\beta\leq\alpha}C_{\alpha,\beta}\left\Vert \varepsilon^{\alpha_{1}}\partial^{\alpha-\beta}\partial_{\theta}\widetilde{v}_{\varepsilon}^{1R}\,\partial^{\beta}q_{\varepsilon}^{a}\right\Vert _{L^{2}}\,,\\
 & \leq\ C\,\varepsilon^{M-2}\sum_{0\leq\beta\leq\alpha}C_{\alpha,\beta}\left\Vert \varepsilon^{\alpha_{1}}\partial^{\beta}q_{\varepsilon}^{a}\right\Vert _{L^{\infty}}\left\Vert \partial^{\alpha-\beta}\partial_{\theta}\widetilde{v}_{\varepsilon}^{1R}\right\Vert _{L^{2}}\,,\\
 & \leq\ C\,\varepsilon^{M-2}\sum_{0\leq\beta\leq\alpha}C_{\alpha,\beta}\left\Vert q_{\varepsilon}^{a}\right\Vert _{W_{(1,\varepsilon)}^{|\alpha|,\infty}}\left\Vert v_{\varepsilon}^{R}\right\Vert _{H^{m+4}}\,,\\
 & \leq\ C\,\fbox{\ensuremath{\textcolor{marin}{\varepsilon^{M-5/2}}}}\sum_{0\leq\beta\leq\alpha}C_{\alpha,\beta}\left\Vert q_{\varepsilon}^{a}\right\Vert _{H_{(1,\varepsilon)}^{m+5}}\left\Vert v_{\varepsilon}^{R}\right\Vert _{H^{m+4}}\,,\\
 & \leq\ C\, C_{a}\sum_{0\leq\beta\leq\alpha}C_{\alpha,\beta}\left\Vert v_{\varepsilon}^{R}\right\Vert _{H^{m+4}}\,.
\end{align*}
 Finally for all $\varepsilon\in]0,1]$ and for all time $t\in[0,T_{\varepsilon}^{*}]$
we have, 
\begin{multline}
\left|\left\langle C\varepsilon^{M-2}\varepsilon^{\alpha_{1}}\partial^{\alpha}\left(q_{\varepsilon}^{a}\,\partial_{\theta}\widetilde{v}_{\varepsilon}^{1R}\right),\varepsilon^{\alpha_{1}}\partial^{\alpha}\widetilde{q}_{\varepsilon}^{R}\right\rangle \right|\\
\leq\ \frac{1}{2}C^{2}\, C_{a}^{2}\,\left(\sum_{0\leq\beta\leq\alpha}C_{\alpha,\beta}\right)^{2}\left\Vert v_{\varepsilon}^{R}\right\Vert _{{H}{}^{m+4}}^{2}+\frac{1}{2}\left\Vert \widetilde{q}_{\varepsilon}^{R}\right\Vert _{\overset{\circ}{H}{}_{(1,\varepsilon)}^{|\alpha|}}^{2}\,.
\end{multline}
 It proves Lemma \ref{lemm0}. \end{proof}

\bigskip{}

\paragraph{Estimate of the source term $S_{\varepsilon}^{0,R,N}$.}

The term $S_{\varepsilon}^{0,R,N}$ contains all difficulties ever
met: 
\begin{description}
\item [{\quad{}}] -Singular terms (with respect to the number of derivatives)
such as $\nabla\widetilde{q}_{\varepsilon}^{R}$ are dealt by integrations
by parts to a cost on the regularity of the velocity. 
\item [{\quad{}}] -Singular terms in $\varepsilon$ in anisotropic spaces
such as $\partial_{\theta}\widetilde{q}_{\varepsilon}^{R}$ are dealt
assume the integer $M$ is large enough. 
\end{description}
\begin{lemma} \label{lemsourceq} Assume $M\geq7/2$. Let a multi-index
$\alpha\in\N^{2}$ with length smaller than $m+3$. There exist three
positive constants $C_{S}^{1}$, $C_{S}^{2}$ and $S_{S}$ such that
for all $\varepsilon\in]0,1]$ and for all time $t\in[0,T_{\varepsilon}^{*}]$:
\vspace{-0.3cm}
 
\begin{multline*}
\left|\left\langle \varepsilon^{\alpha_{1}}\partial^{\alpha}S_{\varepsilon}^{0,R,N},\varepsilon^{\alpha_{1}}\partial^{\alpha}\widetilde{q}_{\varepsilon}^{R}\right\rangle \right|(t)\leq C_{S}^{1}\left(1+\left\Vert v_{\varepsilon}^{R}(t,\cdot)\right\Vert _{H^{m+4}(\T\times\R)}^{2}\right)\left\Vert \widetilde{q}_{\varepsilon}^{R}(t,\cdot)\right\Vert _{\overset{\circ}{H}{}_{(1,\varepsilon)}^{|\alpha|}(\T\times\R)}^{2}\\
+C_{S}^{2}\left(1+\left\Vert v_{\varepsilon}^{R}(t,\cdot)\right\Vert _{H^{m+4}(\T\times\R)}^{2}\right)\left\Vert \widetilde{q}_{\varepsilon}^{R}(t,\cdot)\right\Vert _{{H}{}_{(1,\varepsilon)}^{|\alpha|-1}(\T\times\R)}^{2}\\
+S_{S}\left(\left\Vert v_{\varepsilon}^{R}(t,\cdot)\right\Vert _{H^{m+4}(\T\times\R)}^{2}+\varepsilon^{2(N-R)}\right).
\end{multline*}
 \end{lemma}

\begin{proof}[Proof of Lemma \ref{lemsourceq}] We decompose
$S_{\varepsilon}^{0,R,N}$ into: 
\begin{multline*}
\left\langle \varepsilon^{\alpha_{1}}\partial^{\alpha}S_{\varepsilon}^{0,R,N},\varepsilon^{\alpha_{1}}\partial^{\alpha}\widetilde{q}_{\varepsilon}^{R}\right\rangle =\left\langle \varepsilon^{\alpha_{1}}\partial^{\alpha}\varepsilon^{N-R}\left(\varepsilon^{-N}\mathcal{L}_{0}(\varepsilon,q_{\varepsilon}^{a},v_{\varepsilon}^{a})\right),\varepsilon^{\alpha_{1}}\partial^{\alpha}\widetilde{q}_{\varepsilon}^{R}\right\rangle \\
+\left\langle \varepsilon^{R+M-2}\varepsilon^{\alpha_{1}}\partial^{\alpha}\left(\widetilde{v}_{\varepsilon}^{1R}\,\partial_{\theta}\widetilde{q}_{\varepsilon}^{R}\right),\varepsilon^{\alpha_{1}}\partial^{\alpha}\widetilde{q}_{\varepsilon}^{R}\right\rangle +\left\langle \varepsilon^{R+M-2}\varepsilon^{\alpha_{1}}\partial^{\alpha}\left(\widetilde{v}_{\varepsilon}^{2R}\,\partial_{y}\widetilde{q}_{\varepsilon}^{R}\right),\varepsilon^{\alpha_{1}}\partial^{\alpha}\widetilde{q}_{\varepsilon}^{R}\right\rangle \\
+\left\langle C\,\varepsilon^{R+M-2}\varepsilon^{\alpha_{1}}\partial^{\alpha}\left(\widetilde{q}_{\varepsilon}^{R}\,\partial_{\theta}\widetilde{v}_{\varepsilon}^{1R}\right),\varepsilon^{\alpha_{1}}\partial^{\alpha}\widetilde{q}_{\varepsilon}^{R}\right\rangle +\left\langle C\,\varepsilon^{R+M-2}\varepsilon^{\alpha_{1}}\partial^{\alpha}\left(\widetilde{q}_{\varepsilon}^{R}\,\partial_{y}\widetilde{v}_{\varepsilon}^{2R}\right),\varepsilon^{\alpha_{1}}\partial^{\alpha}\widetilde{q}_{\varepsilon}^{R}\right\rangle .
\end{multline*}
 In what follows we study each of these contributions. Since contributions
of the second and third term (respectively fourth and fifth term)
are estimated following the same steps, we only write the estimates
for the second term (respectively the fourth term).\\

$\circ$ \textit{Contribution of the First term}. Since $q_{\varepsilon}^{a}$
is built as an approximated solution it satisfies Inequality \eqref{eql0h}.
Then, for all $\varepsilon\in]0,1]$ and for all time $t\in[0,T_{\varepsilon}^{*}]$:
\begin{align*}
 & \left|\left\langle \varepsilon^{N-R}\varepsilon^{\alpha_{1}}\partial^{\alpha}\left(\varepsilon^{-N}\mathcal{L}_{0}(\varepsilon,q_{\varepsilon}^{a},v_{\varepsilon}^{a})\right),\varepsilon^{\alpha_{1}}\partial^{\alpha}\widetilde{q}_{\varepsilon}^{R}\right\rangle \right|(t)\\
 & \hspace{5cm}\leq\frac{1}{2}\left(\varepsilon^{2(N-R)}\left\Vert \varepsilon^{-N}\mathcal{L}_{0}(\varepsilon,q_{\varepsilon}^{a},v_{\varepsilon}^{a})\right\Vert _{\overset{\circ}{H}{}_{(1,\varepsilon)}^{|\alpha|}}^{2}+\left\Vert \widetilde{q}_{\varepsilon}^{R}(t,\cdot)\right\Vert _{\overset{\circ}{H}{}_{(1,\varepsilon)}^{|\alpha|}}^{2}\right)\,,\\
 & \hspace{5cm}\leq\frac{1}{2}\left(\varepsilon^{2(N-R)}C_{\mathcal{L}}^{2}+\left\Vert \widetilde{q}_{\varepsilon}^{R}(t,\cdot)\right\Vert _{\overset{\circ}{H}{}_{(1,\varepsilon)}^{|\alpha|}}^{2}\right)\,.
\end{align*}

$\circ$ \textit{Contribution of the second term}. We now estimate
the contribution of 
\[
\left\langle \varepsilon^{R+M-2}\varepsilon^{\alpha_{1}}\partial^{\alpha}\left(\widetilde{v}_{\varepsilon}^{1R}\,\partial_{\theta}\widetilde{q}_{\varepsilon}^{R}\right),\varepsilon^{\alpha_{1}}\partial^{\alpha}\widetilde{q}_{\varepsilon}^{R}\right\rangle .
\]
 We need to clearly understand the regularity required on the pressure
and the velocity. We expand the derivatives of the product thanks
to the Leibniz formula and put aside the extremal terms in the summation.
There exists a family of positive constant $\{C_{\alpha,\beta}\}$
such that 
\begin{align}\label{eqtermeqsource}
 & \left|\left\langle \varepsilon^{R+M-2}\varepsilon^{\alpha_{1}}\partial^{\alpha}\left(\widetilde{v}_{\varepsilon}^{1R}\partial_{\theta}\widetilde{q}_{\varepsilon}^{R}\right),\varepsilon^{\alpha_{1}}\partial^{\alpha}\widetilde{q}_{\varepsilon}^{R}\right\rangle \right|\nonumber \\
 & \hspace{2cm}=\varepsilon^{R+M-2}\int_{\T\times\R}\varepsilon^{2\alpha_{1}}\ \partial^{\alpha}\widetilde{v}_{\varepsilon}^{1R}\ \partial_{\theta}\widetilde{q}_{\varepsilon}^{R}\ \partial^{\alpha}\widetilde{q}_{\varepsilon}^{R}d\theta\, dy\nonumber \\
 & \hspace{4cm}+\varepsilon^{R+M-2}\sum_{0<\beta<\alpha}C_{\alpha,\beta}\int_{\T\times\R}\varepsilon^{2\alpha_{1}}\ \partial^{\alpha-\beta}\widetilde{v}_{\varepsilon}^{1R}\ \partial_{\theta}\partial^{\beta}\widetilde{q}_{\varepsilon}^{R}\ \partial^{\alpha}\widetilde{q}_{\varepsilon}^{R}d\theta\, dy\nonumber \\
 & \hspace{6cm}-\varepsilon^{R+M-2}\int_{\T\times\R}\varepsilon^{2\alpha_{1}}\frac{\partial_{\theta}\widetilde{v}_{\varepsilon}^{1R}}{2}\left(\partial^{\alpha}\widetilde{q}_{\varepsilon}^{R}\right)^{2}d\theta\, dy\,.
\end{align}
 -\textit{First}, we deal with $-\varepsilon^{R+M-2}\int_{\T\times\R}\varepsilon^{2\alpha_{1}}\frac{\partial_{\theta}\widetilde{v}_{\varepsilon}^{1R}}{2}\left(\partial^{\alpha}\widetilde{q}_{\varepsilon}^{R}\right)^{2}d\theta\, dy$.
We perform as in Inequality \eqref{ineqipp}. For all $\varepsilon\in]0,1]$
and for all time $t\in[0,T_{\varepsilon}^{*}]$, 
\begin{multline}
\left|\varepsilon^{R+M-2}\int_{\T\times\R}\varepsilon^{2\alpha_{1}}\frac{\partial_{\theta}\widetilde{v}_{\varepsilon}^{1R}}{2}\left(\partial^{\alpha}\widetilde{q}_{\varepsilon}^{R}\right)^{2}d\theta\, dy\right|\leq\frac{1}{2}\left\Vert \partial_{\theta}\widetilde{v}_{\varepsilon}^{R}\right\Vert _{L^{\infty}}\left\Vert \varepsilon^{\alpha_{1}}\ \partial^{\alpha}\widetilde{q}_{\varepsilon}^{R}\right\Vert _{L^{2}}^{2}\leq\left\Vert \widetilde{q}_{\varepsilon}^{R}\right\Vert _{\overset{\circ}{H}{}_{(1,\varepsilon)}^{|\alpha|}}^{2}.\label{ineqqsource1}
\end{multline}
 -\textit{Then}, we consider the second extremal term $\varepsilon^{R+M-2}\int_{\T\times\R}\varepsilon^{2\alpha_{1}}\ \partial^{\alpha}\widetilde{v}_{\varepsilon}^{1R}\ \partial_{\theta}\widetilde{q}_{\varepsilon}^{R}\ \partial^{\alpha}\widetilde{q}_{\varepsilon}^{R}\, d\theta\, dy$.
We have to get a control over the family $\{\partial_{\theta}\widetilde{q}_{\varepsilon}^{R}\}_{\varepsilon}$.
To do so, we introduce a power of $\varepsilon$ (to a cost on~$M$).
\begin{align*}
 & \varepsilon^{R+M-2}\left|\left\langle \varepsilon^{2\alpha_{1}}\ \partial^{\alpha}\widetilde{v}_{\varepsilon}^{1R}(t,\cdot)\ \partial_{\theta}\widetilde{q}_{\varepsilon}^{R}(t,\cdot),\partial^{\alpha}\widetilde{q}_{\varepsilon}^{R}(t,\cdot)\right\rangle \right|\\
 & \hspace{4cm}\leq\varepsilon^{R+M-2}\left\Vert \partial_{\theta}\widetilde{q}_{\varepsilon}^{R}(t,\cdot)\right\Vert _{L^{\infty}}\left|\left\langle \varepsilon^{\alpha_{1}}\partial^{\alpha}\widetilde{v}_{\varepsilon}^{1R}(t,\cdot),\varepsilon^{\alpha_{1}}\partial^{\alpha}\widetilde{q}_{\varepsilon}^{R}(t,\cdot)\right\rangle \right|\,,\\
 & \hspace{4cm}\leq\frac{\varepsilon^{R+M-\textcolor{marin}{3}}}{2}\left\Vert \textcolor{marin}{\varepsilon}\partial_{\theta}\widetilde{q}_{\varepsilon}^{R}(t,\cdot)\right\Vert _{L^{\infty}}\left(\left\Vert v_{\varepsilon}^{R}(t,\cdot)\right\Vert _{{H}^{m+3}}^{2}+\left\Vert \widetilde{q}_{\varepsilon}^{R}(t,\cdot)\right\Vert _{\overset{\circ}{H}{}_{(1,\varepsilon)}^{|\alpha|}}^{2}\right)\,,\\
 & \hspace{4cm}\leq\frac{\textcolor{marin}{\varepsilon^{R+M-3}}}{2}\left\Vert \widetilde{q}_{\varepsilon}^{R}(t,\cdot)\right\Vert _{W_{(1,\varepsilon)}^{1,\infty}}\left(\left\Vert v_{\varepsilon}^{R}(t,\cdot)\right\Vert _{{H}^{m+3}}^{2}+\left\Vert \widetilde{q}_{\varepsilon}^{R}(t,\cdot)\right\Vert _{\overset{\circ}{H}{}_{(1,\varepsilon)}^{|\alpha|}}^{2}\right)\,.
\end{align*}
 Then we use the equivalence of norms \eqref{eqinfininorm} together
with assumption $M\geq7/2$, 
\begin{align}\label{ineqqsource2}
 & \varepsilon^{R+M-2}\left|\left\langle \varepsilon^{2\alpha_{1}}\ \partial^{\alpha}\widetilde{v}_{\varepsilon}^{1R}(t,\cdot)\ \partial_{\theta}\widetilde{q}_{\varepsilon}^{R}(t,\cdot),\partial^{\alpha}\widetilde{q}_{\varepsilon}^{R}(t,\cdot)\right\rangle \right|\nonumber \\
 & \hspace{4cm}\leq\frac{\fbox{\ensuremath{\textcolor{marin}{\varepsilon^{R+M-7/2}}}}}{2}\left\Vert \widetilde{q}_{\varepsilon}^{R}(t,\cdot)\right\Vert _{H_{(1,\varepsilon)}^{1}}\left(\left\Vert v_{\varepsilon}^{R}(t,\cdot)\right\Vert _{{H}^{m+3}}^{2}+\left\Vert \widetilde{q}_{\varepsilon}^{R}(t,\cdot)\right\Vert _{\overset{\circ}{H}{}_{(1,\varepsilon)}^{|\alpha|}}^{2}\right)\,,\nonumber \\
 & \hspace{4cm}\leq\left(\left\Vert v_{\varepsilon}^{R}(t,\cdot)\right\Vert _{{H}^{m+4}}^{2}+\left\Vert \widetilde{q}_{\varepsilon}^{R}(t,\cdot)\right\Vert _{\overset{\circ}{H}_{(1,\varepsilon)}^{|\alpha|}}^{2}\right)\,.
\end{align}
 -\textit{Finally}, there remains to get a bound for the sum appearing
in the Equation \eqref{eqtermeqsource}. The Gagliardo-Niremberg's
estimate only provides a control in terms of the $H_{(1,\varepsilon)}^{|\alpha|+1}$-norm
of the pressure. Here one can notice that $0<\beta<\alpha$ and so
$|\alpha-\beta|\leq m+2$. Thus $\partial^{\alpha-\beta}\widetilde{v}_{\varepsilon}^{R}$
is bounded in~$L^{\infty}$. For all $\varepsilon\in]0,1]$ and for
all time $t\in[0,T_{\varepsilon}^{*}]$, 
\begin{align}\label{ineqqsource}
 & \varepsilon^{M+R-2}\left|\left\langle \varepsilon^{2\alpha_{1}}\partial^{\alpha-\beta}\widetilde{v}_{\varepsilon}^{1R}\,\partial_{\theta}\partial^{\beta}\widetilde{q}_{\varepsilon}^{R},\partial^{\alpha}\widetilde{q}_{\varepsilon}^{R}\right\rangle \right|\nonumber \\
 & \hspace{3cm}\leq\frac{1}{2}\left(\fbox{\ensuremath{\textcolor{marin}{\varepsilon^{2(R+M-3)}}}}\left\Vert \varepsilon^{\alpha_{1}+1}\partial^{\alpha-\beta}\widetilde{v}_{\varepsilon}^{1R}\partial_{\theta}\partial^{\beta}\widetilde{q}_{\varepsilon}^{R}\right\Vert _{L^{2}}^{2}+\left\Vert \varepsilon^{\alpha_{1}}\partial^{\alpha}\widetilde{q}_{\varepsilon}^{R}\right\Vert _{L^{2}}^{2}\right)\nonumber \\
 & \hspace{3cm}\leq\frac{1}{2}\left(1+\left\Vert \widetilde{v}_{\varepsilon}^{R}(t,\cdot)\right\Vert _{H^{m+4}}^{2}\right)\left\Vert \widetilde{q}_{\varepsilon}^{R}(t,\cdot)\right\Vert _{{H}{}_{(1,\varepsilon)}^{|\alpha|}}^{2}
\end{align}
 -Finally plugging estimates \eqref{ineqqsource1}, \eqref{ineqqsource2}
and \eqref{ineqqsource} into \eqref{eqtermeqsource} we get for all
$\varepsilon\in]0,1]$ and for all time~$t\in[0,T_{\varepsilon}^{*}]$,
\vspace{-0.3cm}
 
\begin{multline*}
\left|\left\langle \varepsilon^{R+M-2}\varepsilon^{\alpha_{1}}\partial^{\alpha}\left(\widetilde{v}_{\varepsilon}^{1R}\,\partial_{\theta}\widetilde{q}_{\varepsilon}^{R}\right),\varepsilon^{\alpha_{1}}\partial^{\alpha}\widetilde{q}_{\varepsilon}^{R}\right\rangle \right|(t)\\
\leq\left(\sum_{0<\beta<\alpha}\frac{C_{\alpha,\beta}}{2}\left(\left\Vert \widetilde{v}_{\varepsilon}^{R}(t,\cdot)\right\Vert _{H^{m+4}}^{2}+1\right)+\frac{3}{2}\right)\left\Vert \widetilde{q}_{\varepsilon}^{R}(t,\cdot)\right\Vert _{{H}{}_{(1,\varepsilon)}^{|\alpha|}}^{2}+\left\Vert \widetilde{v}_{\varepsilon}^{R}(t,\cdot)\right\Vert _{H^{m+4}}^{2}\,.
\end{multline*}

$\circ$ \textit{Contribution of the fourth term}. To estimate $\left\langle C\varepsilon^{R+M-2}\varepsilon^{\alpha_{1}}\partial^{\alpha}\left(\widetilde{q}_{\varepsilon}^{R}\,\partial_{\theta}\widetilde{v}_{\varepsilon}^{1R}\right),\varepsilon^{\alpha_{1}}\partial^{\alpha}\widetilde{q}_{\varepsilon}^{R}\right\rangle $,
we first apply the Cauchy-Schwarz inequality, 
\[
\left|\left\langle C\varepsilon^{R+M-2}\varepsilon^{\alpha_{1}}\partial^{\alpha}\left(\widetilde{q}_{\varepsilon}^{R}\,\partial_{\theta}\widetilde{v}_{\varepsilon}^{1R}\right),\varepsilon^{\alpha_{1}}\partial^{\alpha}\widetilde{q}_{\varepsilon}^{R}\right\rangle \right\Vert \leq\frac{1}{2}\left(\left\Vert C\varepsilon^{R+M-2}\varepsilon^{\alpha_{1}}\partial^{\alpha}\left(\widetilde{q}_{\varepsilon}^{R}\,\partial_{\theta}\widetilde{v}_{\varepsilon}^{1R}\right)\right\Vert _{L^{2}}^{2}+\left\Vert \widetilde{q}_{\varepsilon}^{R}\right\Vert _{\overset{\circ}{H}{}_{(1,\varepsilon)}^{|\alpha|}}^{2}\right)\,.
\]
 The control over $\partial^{\alpha}\left(\widetilde{q}_{\varepsilon}^{R}\,\partial_{\theta}\widetilde{v}_{\varepsilon}^{1R}\right)$
could be done thanks to the Gagliardo-Niremberg inequality. However
it only gives a bound which depends on the $H_{(1,\varepsilon)}^{|\alpha|+1}$-norm
(respectively $H^{|\alpha|+1}$-norm) of the pressure (respectively
velocity). We have to be more accurate. We compute this contribution
thanks to the Leibniz formula: 
\begin{align*}
\left\Vert C\varepsilon^{R+M-2}\varepsilon^{\alpha_{1}}\partial^{\alpha}\left(\widetilde{q}_{\varepsilon}^{R}\,\partial_{\theta}\widetilde{v}_{\varepsilon}^{1R}\right)\right\Vert _{L^{2}} & \leq C\varepsilon^{R+M-2}\sum_{0\leq\beta\leq\alpha}C_{\alpha,\beta}\left\Vert \varepsilon^{\alpha_{1}}\ \partial^{\beta}\widetilde{q}_{\varepsilon}^{R}\ \partial^{\alpha-\beta}\partial_{\theta}\widetilde{v}_{\varepsilon}^{1R}\right\Vert _{L^{2}}\,.
\end{align*}
 Then we study the competition between the regularity over the pressure
$\widetilde{q}_{\varepsilon}^{R}$ and over the velocity~$\widetilde{v}_{\varepsilon}^{R}$.
To do so, we cut the sum in two parts depending on the length of $\beta$.
When $|\beta|$ is small enough, $\partial^{\beta}\widetilde{q}_{\varepsilon}^{R}$
is bounded in $L^{\infty}$ whereas when $|\beta|$ is large it is
$\partial^{\alpha-\beta}\partial_{\theta}\widetilde{v}_{\varepsilon}^{1R}$
which is bounded in $L^{\infty}$. Therefore we decompose the sum
into 
\begin{multline*}
\left\Vert C\varepsilon^{R+M-2}\varepsilon^{\alpha_{1}}\partial^{\alpha}\left(\widetilde{q}_{\varepsilon}^{R}\,\partial_{\theta}\widetilde{v}_{\varepsilon}^{1R}\right)\right\Vert _{L^{2}}\leq\ \underbrace{C\varepsilon^{R+M-2}\sum_{0\leq\beta\leq\alpha,0\leq|\beta|\leq1}C_{\alpha,\beta}\left\Vert \varepsilon^{\alpha_{1}}\ \partial^{\beta}\widetilde{q}_{\varepsilon}^{R}\ \partial^{\alpha-\beta}\partial_{\theta}\widetilde{v}_{\varepsilon}^{1R}\right\Vert _{L^{2}}}_{=:C_{\beta s}}\\
+\underbrace{C\varepsilon^{R+M-2}\sum_{0\leq\beta\leq\alpha,2\leq|\beta|\leq|\alpha|}C_{\alpha,\beta}\left\Vert \varepsilon^{\alpha_{1}}\ \partial^{\beta}\widetilde{q}_{\varepsilon}^{R}\ \partial^{\alpha-\beta}\partial_{\theta}\widetilde{v}_{\varepsilon}^{1R}\right\Vert _{L^{2}}}_{=:C_{\beta l}}\,.
\end{multline*}
 We start by estimating $C_{\beta s}$, 
\begin{align}\label{ineqsum1}
C_{\beta s} & \leq\ C\varepsilon^{R+M-2}\sum_{0\leq\beta\leq\alpha,0\leq|\beta|\leq1}C_{\alpha,\beta}\left\Vert \varepsilon^{\alpha_{1}}\partial^{\beta}\widetilde{q}_{\varepsilon}^{R}\right\Vert _{L^{\infty}}\left\Vert \partial^{\alpha-\beta}\partial_{\theta}\widetilde{v}_{\varepsilon}^{1R}\right\Vert _{L^{2}}\,,\nonumber \\
 & \leq\ C\varepsilon^{R+M-2}\sum_{0\leq\beta\leq\alpha,0\leq|\beta|\leq1}C_{\alpha,\beta}\left\Vert \widetilde{q}_{\varepsilon}^{R}\right\Vert _{W_{(1,\varepsilon)}^{1,\infty}}\left\Vert \widetilde{v}_{\varepsilon}^{1R}\right\Vert _{H^{m+4}}\,,\nonumber \\
 & \leq\ C\varepsilon^{R+M-2-\textcolor{marin}{1/2}}\sum_{0\leq\beta\leq\alpha,0\leq|\beta|\leq1}C_{\alpha,\beta}\left\Vert \widetilde{q}_{\varepsilon}^{R}\right\Vert _{H_{(1,\varepsilon)}^{1}}\left\Vert \widetilde{v}_{\varepsilon}^{1R}\right\Vert _{H^{m+4}}\,,\nonumber \\
 & \leq\ 2C\fbox{\ensuremath{\textcolor{marin}{\varepsilon^{R+M-5/2}}}}\sum_{0\leq\beta\leq\alpha,0\leq|\beta|\leq1}C_{\alpha,\beta}\left\Vert v_{\varepsilon}^{R}\right\Vert _{H^{m+4}}\,.
\end{align}
 Presently, we get a bound the second term $C_{\beta l}$ by 
\begin{align}\label{ineqsum2}
C_{\beta l} & \leq\, C\varepsilon^{R+M-2}\sum_{0\leq\beta\leq\alpha,2\leq|\beta|\leq|\alpha|}C_{\alpha,\beta}\left\Vert \varepsilon^{\alpha_{1}}\partial^{\beta}\widetilde{q}_{\varepsilon}^{R}\right\Vert _{L^{2}}\left\Vert \partial^{\alpha-\beta}\partial_{\theta}\widetilde{v}_{\varepsilon}^{1R}\right\Vert _{L^{\infty}},\nonumber \\
 & \leq\, C\varepsilon^{R+M-2}\sum_{0\leq\beta\leq\alpha,2\leq|\beta|\leq|\alpha|}C_{\alpha,\beta}\left\Vert \widetilde{q}_{\varepsilon}^{R}\right\Vert _{H_{(1,\varepsilon)}^{|\alpha|}}\left\Vert v_{\varepsilon}^{R}\right\Vert _{W^{m+2,\infty}},\nonumber \\
 & \leq C\varepsilon^{R+M-2}\sum_{0\leq\beta\leq\alpha,2\leq|\beta|\leq|\alpha|}C_{\alpha,\beta}\left\Vert v_{\varepsilon}^{R}\right\Vert _{H^{m+4}}\left\Vert \widetilde{q}_{\varepsilon}^{R}\right\Vert _{H_{(1,\varepsilon)}^{|\alpha|}}\,.
\end{align}
 Joining the two estimates \eqref{ineqsum1} and \eqref{ineqsum2}
together we deduce that for all $\varepsilon\in]0,1]$ and for all
time~$t\in[0,T_{\varepsilon}^{*}]$:
\begin{multline*}
\left|\left\langle C\varepsilon^{R+M-2}\varepsilon^{\alpha_{1}}\partial^{\alpha}\left(\widetilde{q}_{\varepsilon}^{R}\,\partial_{\theta}\widetilde{v}_{\varepsilon}^{1R}\right),\varepsilon^{\alpha_{1}}\partial^{\alpha}\widetilde{q}_{\varepsilon}^{R}\right\rangle \right|\leq\,2\, C^{2}\left(\sum_{0\leq\beta\leq\alpha,0\leq|\beta|\leq1}C_{\alpha,\beta}\right)^{2}\left\Vert \widetilde{v}_{\varepsilon}^{1R}\right\Vert _{H^{m+4}}^{2}\\
+\frac{1}{2}\left(C^{2}\left(\sum_{0\leq\beta\leq\alpha,2\leq|\beta|\leq|\alpha|}C_{\alpha,\beta}\right)^{2}\left\Vert v_{\varepsilon}^{R}\right\Vert _{H^{m+4}}^{2}+1\right)\left\Vert \widetilde{q}_{\varepsilon}^{R}\right\Vert _{H_{(1,\varepsilon)}^{|\alpha|}}^{2}\,.
\end{multline*}
 This achieves the proof. \end{proof}

\bigskip{}

Finally, Lemma \ref{lemallq} is a simple corollary of Lemmas \ref{lemcomh}-$\ldots$-\ref{lemsourceq}.

\subsubsection{Proof of Proposition \ref{lemestq}}

\label{subsubsectionprooflemma}

In this subsection we prove Proposition $\ref{lemestq}$. We prove,
by induction on the size $J$ that property~$\mathcal{Q}(J)$ defined
page \pageref{recurezqnl} is satisfied for $J\in\llbracket0,m+3\rrbracket$.
Since the proof for the initialization of the induction ($\mathcal{Q}(0)$
is true) and the increment ($\mathcal{Q}(J)\Rightarrow\mathcal{Q}(J+1)$),
we only the proof of the increment.\\
 One aspect of the proof is to give meaning to the \textit{a priori}
Inequality (\ref{ineqergJq}). \\

\begin{proof}[Proof of Proposition \ref{lemestq}]{We assume
$\mathcal{Q}(J)$ is true, for some $J\in\llbracket0,m+2\rrbracket$}.

$\circ$ Apply Lemma $\ref{lemallq}$, there exist three positive
constants $C_{J+1}^{1}$, $C_{J+1}^{2}$ and $C_{J+1}^{3}$ such that
for all $\varepsilon\in]0,1]$ and for all time $t\in[0,T_{\varepsilon}^{*}]$:
\vspace{-0.2cm}
 
\begin{equation}
\partial_{t}\left\Vert \widetilde{q}_{\varepsilon}^{R}(t,\cdot)\right\Vert _{H_{(1,\varepsilon)}^{J+1}(\T\times\R)}^{2}\leq\psi_{\varepsilon}^{J+1}(t)\left\Vert \widetilde{q}_{\varepsilon}^{R}(t,\cdot)\right\Vert _{\overset{\circ}{H}{}_{(1,\varepsilon)}^{J+1}}^{2}+\varphi_{\varepsilon}^{J+1}(t),\label{eqqsimp}
\end{equation}
 with for all $\varepsilon\in]0,1]$ and for all time $t\in[0,T_{\varepsilon}^{*}]$:
\begin{align*}
 & \psi_{\varepsilon}^{J+1}(t):=2C_{J+1}^{1}\left(1+\left\Vert \widetilde{v}_{\varepsilon}^{R}(t,\cdot)\right\Vert _{H^{m+4}}^{2}\right),\\
 & \varphi_{\varepsilon}^{J+1}(t):=2C_{J+1}^{2}\left(1+\left\Vert \widetilde{v}_{\varepsilon}^{R}(t,\cdot)\right\Vert _{H^{m+4}}^{2}\right)\left\Vert \widetilde{q}_{\varepsilon}^{R}(t,\cdot)\right\Vert _{H_{(1,\varepsilon)}^{J}}^{2}+2C_{J+1}^{3}\left(\left\Vert \widetilde{v}_{\varepsilon}^{R}(s,\cdot)\right\Vert _{H^{m+4}}^{2}+\varepsilon^{2(N-R)}\right).
\end{align*}

$\circ$ On the one hand, since $\widetilde{v}_{\varepsilon}^{R}$
is in $L_{t}^{2}H_{\theta,y}^{m+4}$ (see Proposition \ref{propestu}),
the family $\{\psi_{\varepsilon}^{J+1}\}_{\varepsilon\in]0,\varepsilon_{d}]}$
is bounded in $L^{1}$: 
\begin{align*}
\forall\,\varepsilon\in]0,\varepsilon_{d}],\ \forall\, t\in[0,T_{\varepsilon}^{*}],\qquad\int_{0}^{t}\psi_{\varepsilon}^{J+1}(s)ds\lesssim1.
\end{align*}
 On the other hand for all $\varepsilon\in]0,\varepsilon_{d}]$, applying
the assumption of induction $\mathcal{Q}(J)$ then Proposition~\ref{propestu},
the function $\varphi_{\varepsilon}^{J+1}$ is bounded in $L^{1}([0,t])$
for any $t\in[0,t_{\varepsilon}^{*}]$: 
\begin{align*}
\int_{0}^{t}\varphi_{\varepsilon}^{J+1}(s)ds & \lesssim\int_{0}^{t}\left(\left(1+\left\Vert \widetilde{v}_{\varepsilon}^{R}(s,\cdot)\right\Vert _{H^{m+4}}^{2}\right)\left\Vert \widetilde{q}_{\varepsilon}^{R}(s,\cdot)\right\Vert _{H_{(1,\varepsilon)}^{J}}^{2}\right.+\left.\left(\left\Vert \widetilde{v}_{\varepsilon}^{R}(s,\cdot)\right\Vert _{H^{m+4}}^{2}+\varepsilon^{2(N-R)}\right)\right)ds\,,\\
 & \lesssim\,\int_{0}^{t}\left(\varepsilon^{2w_{m}}+\left\Vert v_{\varepsilon}^{R}(s,\cdot)\right\Vert _{H^{m+4}}^{2}\right)\,\, ds\lesssim\,\varepsilon^{2w_{m}}t.
\end{align*}
 We can integrate the Inequality \eqref{eqqsimp} with respect to
the time and apply the Gronwall Lemma. For all $\varepsilon\in]0,\varepsilon_{d}]$
and for all time $t\in[0,T_{\varepsilon}^{*}]$: 
\[
\left\Vert \widetilde{q}_{\varepsilon}^{R}(t,\cdot)\right\Vert _{\overset{\circ}{H}{}_{(1,\varepsilon)}^{J+1}}^{2}\leq\sup_{s\in[0,t]}(\varphi_{\varepsilon}^{J+1}(s))\exp\left({\int_{0}^{t}\psi_{\varepsilon}^{J+1}(s)ds}\right)\lesssim\, t\,\varepsilon^{2w_{m}}\,.
\]
 It proves the induction. \end{proof}

\bibliographystyle{plain}
\bibliography{biblio}

\end{document}